\documentclass[a4paper, 10pt,reqno]{article}
\usepackage{amssymb, amsmath, amsfonts, amsthm, graphics,mathrsfs,mathtools,dsfont}
\usepackage{stmaryrd}
\usepackage[hmargin=3cm, vmargin = 3cm]{geometry}
\usepackage[T1]{fontenc}
\usepackage{lmodern}
\usepackage{colonequals}
\usepackage{tikz-cd} 
\usepackage{hyperref}
\usepackage[all]{xy}
\usepackage{authblk}

\let\emph\relax
\DeclareTextFontCommand{\emph}{\bfseries}

\theoremstyle{definition}
\newtheorem{Definition}{Definition}[section]

\newtheorem{Remark}[Definition]{Remark}

\theoremstyle{theorem}
\newtheorem{Theorem}[Definition]{Theorem}
\newtheorem{Lemma}[Definition]{Lemma}

\newtheorem{Proposition}[Definition]{Proposition}

\numberwithin{equation}{section}

\DeclareMathOperator{\support}{supp}
\DeclareMathOperator{\dive}{div}
\DeclareMathOperator{\curle}{curl}

\newcommand{\NN}{\mathbb{N}}

\newcommand{\RR}{\mathbb{R}}
\renewcommand{\SS}{\mathbb{S}}

\newcommand{\cA}{\mathcal{A}}
\newcommand{\cB}{\mathcal{B}}
\newcommand{\cC}{\mathcal{C}}

\newcommand{\cF}{\mathcal{F}}
\newcommand{\cG}{\mathcal{G}}
\newcommand{\cH}{\mathcal{H}}
\newcommand{\cI}{\mathcal{I}}

\newcommand{\cM}{\mathcal{M}}

\newcommand{\cP}{\mathcal{P}}

\newcommand{\cY}{\mathcal{Y}}
\newcommand{\cZ}{\mathcal{Z}}

\theoremstyle{definition}

\theoremstyle{definition}

\newcommand{\wsconverge}{\xrightharpoonup{w^*}}
\newcommand{\wconverge}{\xrightharpoonup{\,\,w\,}}
\newcommand{\sconverge}{\xrightarrow{\,\,s\,\,}}

\newcommand\abs[1]{\left|#1\right|}
\newcommand\norm[1]{\left\|#1\right\|}
\newcommand\normf[1]{\|#1\|}

\hypersetup{
hidelinks,
colorlinks=true,
linkcolor=blue,
citecolor=red,
}

\title{Well-Posedness for 2D Combustion Model in Bounded Domains and Serrin-Type Blowup Criterion}
\author{Jiawen Zhang\thanks{Email address: zhangjiawen@amss.ac.cn}}
\affil{\normalsize School of Mathematical Sciences, \\
University of Chinese Academy of Sciences, Bejing 100049, P.R. China}

\date{}

\begin{document}

\maketitle

\begin{abstract}
We investigate the 2D combustion model with Dirichlet boundary conditions and slip boundary conditions in bounded domains. The global existence of weak and strong solutions and the uniqueness of strong solutions are obtained provided the initial density is small in some precise sense. Using the energy method  and the estimates of boundary integrals, we obtain the a priori bounds of the density and velocity field. In addition, we prove the local existence of the strong solutions via iterative method and the contraction mapping theorem. Finally, we extend the well-known Serrin's blowup criterion to the 2D combustion model. Under the suitable boundary conditions, the Serrin's condition on the velocity can be removed in this criteria.
\end{abstract}

\section{Intrduction and main results}\label{Section1}

In this paper, we will study the following combustion model:
\begin{equation}\label{equation1.1}
\begin{cases}
\rho_t+\dive (\rho u)=0,\,\,\rho\geq 0,\\
(\rho u)_t+\dive(\rho u\otimes u)-\dive(2\mu D)+\nabla \pi=0,\\
\dive u=c_0\Delta\psi(\rho),\,\,\psi(\rho):=\rho^{-1},
\end{cases}
\end{equation}
for $t>0$ and $x\in\Omega$, where $\Omega\subset\RR^2$ is a bounded simply connected domain with smooth boundary. Here $u=(u_1,u_2)$, $\rho$ and $\pi$ stand for the unknown velocity field, density and pressure respectively, $c_0>0$ is a fixed constant and 
$$0<\mu=\mu(s)\in C^\infty[0,\infty).$$ 
We denote 
\begin{equation*}
D=D(u) = \frac{1}{2}(\nabla u+(\nabla u)^t)=\frac{1}{2}(\partial_iu_j+\partial_ju_i),
\end{equation*}
for $1\leq i,j\leq 2$.

From the physical viewpoint, combustion model is the low Mach number limit of the fully compressible Navier-Stokes equations 
\begin{equation}\label{FCNS}
\begin{cases}
\rho_t+\dive (\rho u)=0,\\
(\rho u)_t+\dive(\rho u\otimes u)-\dive\SS+\nabla p=0,\\
(\rho e)_t+\dive(\rho u e)-\dive(k\nabla \theta)+p\dive u=\SS:D(u),
\end{cases}\tag{FCNS}
\end{equation}
where $e,\, \theta,\,p$ stands for the internal enery, temperature and pressure respectively and $A:B$ strands for the inner product of matrices
\begin{equation*}
A:B:=\text{tr}(AB^t)=\sum_{i,j=1}^2a_{ij}b_{ji}.
\end{equation*}
$\SS$ is the viscous strain tensor given by
$$\SS=2\mu D(u)+\lambda\dive u I_{n\times n},$$
where $I_{n\times n}$ is the $n\times n$ indentity matrix. The thermal conductivity $k$ and the viscosity coefficients $\mu,\,\lambda$ are functions of $\rho$ and $\theta$. From Lions's book \cite{lions1}, if we define the Mach number $\epsilon$ as $|u|/\sqrt{p'(\rho)}$ and let $(\rho, u,\theta)$ be smooth solution of system \eqref{FCNS} corresponding to the small $\epsilon$, after rescaling the time variable by
$$\rho^\epsilon(x,t)=\rho\left(x,\frac{t}{\epsilon}\right),\quad u^\epsilon(x,t)=\frac{1}{\epsilon}\rho\left(x,\frac{t}{\epsilon}\right),\quad\theta^\epsilon(x,t)=\theta\left(x,\frac{t}{\epsilon}\right),$$
then, $(\rho^\epsilon,u^\epsilon,\theta^\epsilon)$ satisfies
\begin{equation}\label{FCNS2}
\begin{cases}
\rho^\epsilon_t+\dive (\rho^\epsilon u^\epsilon)=0,\\
(\rho^\epsilon u^\epsilon)_t+\dive(\rho u^\epsilon\otimes u^\epsilon)-\dive\SS^\epsilon+\epsilon^{-2}\nabla p^\epsilon=0,\\
(\rho^\epsilon e^\epsilon)_t+\dive(\rho^\epsilon u^\epsilon e^\epsilon)-\dive(k^\epsilon\nabla \theta^\epsilon)+p^\epsilon\dive u^\epsilon=\epsilon^2\SS^\epsilon\cdot D(u^\epsilon),
\end{cases}\tag{FCNS'}
\end{equation}
where 
\begin{equation*}
\begin{aligned}
\SS^{\epsilon}=2 \mu^\epsilon D(u^\epsilon)+\lambda^{\epsilon} \operatorname{div} u^{\epsilon} I_{n\times n}, 
\end{aligned}
\end{equation*}
and
\begin{equation*}
\begin{aligned}
p^{\epsilon}=p\left(x,\frac{t}{\epsilon}\right),\quad\mu^{\epsilon}=\frac{1}{\epsilon}\mu\left(x,\frac{t}{\epsilon}\right), \quad \lambda^{\epsilon}=\frac{1}{\epsilon}\lambda\left(x,\frac{t}{\epsilon}\right), \quad k^{\epsilon}=\frac{1}{\epsilon}k\left(x,\frac{t}{\epsilon}\right).
\end{aligned}
\end{equation*}
Considerig the ideal gas laws: 
\begin{equation}\label{air}
p=R\rho \theta, \quad e=C_V\theta,
\end{equation}
where $R,\, C_V$ denote the ideal gas constant and the specific heat constant, respectively. Then, letting the Mach number $\epsilon$ go to 0, the momentum equation $\eqref{FCNS2}_2$ implies that
$$
p^{\epsilon}=P(t)+\pi(t, x) \epsilon^2+o\left(\epsilon^2\right) .
$$
Plugging this formula into the energy equation $\eqref{FCNS2}_3$ entails that $P(t)$ is independent of $t$, provided $u^\epsilon$ and $\nabla \theta^\epsilon$ vanish at infinity. From now on, we shall denote this constant by $P_0$. Therefore, denoting $C_P=\gamma C_V=\gamma R /(\gamma-1)$, the low Mach number limit system reads
\begin{equation}\label{equation}
\begin{cases}
\rho C_P\left(\partial_t \theta+u \cdot \nabla \theta\right)-\operatorname{div}(k \nabla \theta)=0,\\
\rho u_t+\rho u \cdot \nabla u-\operatorname{div} \SS+\nabla \pi=0, \\
\gamma P_0 \operatorname{div} u=(\gamma-1) \operatorname{div}(k \nabla \theta).
\end{cases}
\end{equation}
Plugging \eqref{air} into \eqref{equation} with constant heat conductivity coefficient $k$ implies the following system
\begin{equation}\label{equationn}
\begin{cases}
\partial_t \rho+\dive(\rho u)=0, \\
\rho u_t+\rho u \cdot \nabla u-\operatorname{div} \SS+\nabla \pi=0, \\
\dive u={k(\gamma-1)}(R\gamma)^{-1} \Delta\rho^{-1},
\end{cases}
\end{equation}
which is exactly the equations \eqref{equation1.1}.

If we particularly take the diffusion coefficient $c_0=0$, \eqref{equation1.1} will become the classical non-homogeneous incompressible Navier-Stokes equations
\begin{equation}\label{equationN}
\begin{cases}
\rho_t+\dive (\rho u)=0,\,\,\rho\geq 0,\\
(\rho u)_t+\dive(\rho u\otimes u)-\dive(2\mu D)+\nabla \pi=0,\\
\dive u=0.
\end{cases}\tag{N}
\end{equation}

The study of the combustion model may date back to the 1980s. It has been introduced by A. Majda \cite{majda} and studied in particular by P. Embid \cite{embid} who has proved the local-in-time well-posedness of the system \eqref{equation1.1}. For the system \eqref{equation1.1} replacing $\eqref{equation1.1}_3$ by Fick's law with $\psi(\rho)=\log \rho$, the local well-posedness was considered by H. B. da Veiga \cite{daveiga}. Danchin-Liao \cite{danchin}  established the local existence and uniqueness of a solution in critical homogeneous Besov spaces provided the density is closed to a positive constant and they proved the local well-posedness in non-homogeneous Besov space arbitrarily large data.

On the other hand, there are also a large number of works investigating the global-in-time existence of weak and strong solutions for the combustion model. P. Secchi \cite{secchi} proved that there exists a unique global strong solution in the two-dimensional domain under Fick's law providing the diffusion coefficient $c_0$ is small enough. They also considered the limiting behavior of the solutions when $c_0\to 0$ for dimensions 2 and 3 and the convergence towards the corresponding solutions of \eqref{equationN}. Under the small initial data assumption, P. Lions \cite{lions2} showed, in $\RR^2$ or periodic boundary condition, that a small perturbation of a constant density gives a global existence of weak solutions without any restriction on the initial velocity. Danchin-Liao \cite{danchin} proved the existence of solutions in critical homogeneous Besov spaces by assuming the initial density is close to a constant and the initial velocity is small enough. Recently, W. Tan \cite{tan} proved the global existence of the weak and strong solutions of the system \eqref{equation1.1} with general viscosity coefficient $\mu(\rho)$ in $\eqref{equation1.1}_2$ and $\psi(\rho)$ in $\eqref{equation1.1}_3$ provided the density is closed to a positive constant in some precise sense. For large data, Bresch-Essoufi-Sy \cite{bresch2} showed the global existence of the weak solutions for the combustion model in dimensions 2 and 3 by taking $\mu(\rho)=\frac{c_0}{2}\log \rho$. In \cite{bresch1}, Bresch-Giovangigli-Zatorska relaxed the restriction on $\mu(\rho)$ by using the idea of the renormalized solution. 

If one takes the decomposition $u=v+c_0\nabla\rho^{-1}$ with $\dive v=0$ and converts the system \eqref{equation1.1} to the equations for $(\rho,v)$,  then \eqref{equation1.1} will be reduced to the Kazhikhov-Smagulov type model, see \eqref{equation3.2}. In \cite{caixiaoyun,aaa}, Cai-Liao-Sun established the global-in-time existence of strong solutions to the initial-boundary value problem of a 2D Kazhikhov-Smagulov type model for incompressible non-homogeneous fluids with mass diffusion for the arbitrary size of initial data. For works on the classical Kazhikhov-Smagulov's model, we refer the reader to \cite{antontsev,beirao}.

For the general non-homogeneous incompressible Navier-Stokes equations \eqref{equationN} with the viscosity coefficient $\mu(\rho)$ depending on $\rho$, global weak solutions were derived by P. Lions \cite{lions1}. Abidi-Zhang \cite{abidi} obtained the global strong solutions strictly away from vacuum whenever $\norm{u_0}_{L^2}\norm{\nabla u_0}_{L^2}$ and $\norm{\mu(\rho_0)-1}_{L^\infty}$ are small enough. For the initial density containing vacuum, Cho-Kim \cite{cho} established the existence of the local strong solutions under compatibility conditions similar to \cite{jun}. In addition, Huang-Wang \cite{huang}, J. Zhang \cite{zhang2015} established the global strong solutions with small $\norm{\nabla u_0}_{L^2}$ in 3D bounded domains. For the Cauchy problem, He-Li-L$\mathrm{\ddot{u}}$ \cite{he2021} obtained the global strong ones to  with small $\norm{u_0}_{\dot H^\beta}$ for some $\beta\in (1/2,1]$ and some extra restrictions on $\mu(\rho)$ via the exponential decay-in-time estimates. More recently, Cai-L$\mathrm{\ddot{u}}$-Peng \cite{cai2} studied the global existence of strong solutions in 3D exterior domains with nonslip or slip boundary conditions provided that the gradient of the initial velocity is suitably small.

Finally, for the study of the mechanism of blowup and structure of possible singularities of strong (or smooth) solutions to the Navier-Stokes system can be traced to Serrin's criterion \cite{serrin1962} on the Leray-Hopf weak solutions to the 3D incompressible homogeneous Navier-Stokes equations, which can showed that if a weak solution $u$ satisfies
\begin{equation}\label{serrin1.5}
u\in L^s(0,T;L^r),\quad \frac{2}{s}+\frac{3}{r}\leq 1,\quad 3<r\leq \infty,
\end{equation}
then it is regular. Later, He-Xin \cite{he2005} showed that the Serrin's criterion \eqref{serrin1.5} still holds even for the strong solution to the incompressible MHD equations. For non-homogeneous incompressible Navier–Stokes equations \eqref{equationN}, H. Kim \cite{kim2006} established the Serrin-type blowup criterion. They showed that if $(\rho,u)$ blows up at $T^*$, then
\begin{equation}
\lim_{t\to T^*}\norm{u}_{L^s(0,T;L^r_w)}=\infty\quad \text{for all}\quad\frac{2}{s}+\frac{3}{r}\leq 1,\quad 3<r\leq \infty.
\end{equation}
Recently, X. Zhong \cite{zhong2017} obtained a blowup criterion \eqref{serrin1.5} to the non-homogeneous incompressible heat conducting Navier–Stokes flows with non-negative density in bounded domain of $\RR^3$. For the compressible fluids, Huang-Li-Xin \cite{HLX} first extend Serrin's blow-up criterion to the barotropic compressible Navier-Stokes equations. Later, Xu-Zhang \cite{xu2012blow} extended the results of \cite{HLX} to the isentropic compressible MHD system
and Huang-Li-Wang \cite{huang2013serrin} improve the all previous blowup criterion results to the full compressible Navier--Stokes system.

However, for the general viscosity coefficient,  the theory of the combustion model in the bounded domain is still blank. Therefore, our goal is obtaining the global existence of solutions with small initial data and the local existence for \eqref{equation1.1} in the general domain under different initial-boundary conditions and  trying to extend Serrin's blow-up criterion to \eqref{equation1.1}.

More precisely, we impose the initial data
\begin{equation}\label{initial}
u_0(x):=u(x,0),\quad 0<\alpha\leq \rho_0(x):=\rho(x,0)\leq \beta<\infty,\quad x\in \Omega
\end{equation}
and one of the following boundary conditions:
\begin{enumerate}
\item[(1)] 
$\rho$ satisfies the Neumann condition and $u$ satisfies the slip boundary condition, that is,
\begin{equation}\label{A}
n\cdot\nabla \rho=0,\quad u\cdot n=0 \text{ and } \curle u=-n^\perp\cdot B\cdot u \quad\mathrm{on}\,\, \partial \Omega\times(0,T),\tag{A}
\end{equation}
where $n = (n_1,n_2)$ denotes the unit outer normal vector of the boundary $\partial \Omega$, $n^\perp = (n_2, -n_1)$ is the unit tangential vector on the boundary and  $B=B(x)$ is a bounded smooth symmetric matrix which is positive semi-definite;
\item[(2)]
$(\rho, u)$ satisfies the non-homogeneous Dirichlet condition, that is,
\begin{equation}\label{B}
\rho=\tilde\rho,\quad u=c_0\nabla\rho^{-1}\quad \mathrm{on}\,\,\partial\Omega\times(0,T),\tag{B}
\end{equation}
where $\tilde\rho$ is a positive constant such that $\alpha\leq\tilde\rho\leq \beta$;
\item[(3)]
$\rho$ satisfies the Neumann condition and $u$ satisfies the non-slip condition, that is,
\begin{equation}\label{C}
n\cdot\nabla \rho=0,\quad u=0 \quad\mathrm{on}\,\, \partial \Omega\times(0,T).\tag{C}
\end{equation}
\end{enumerate}

Before giving the main results, we explain some notations and conventions used throughout the paper.  For simplicity, we set
\begin{equation*}
\int f:=\int_\Omega f\,dx,\,\,\,\int_\partial f:=\int_{\partial\Omega} f\,dS,\,\,\,\iint f:=\iint_{Q_T} f\,dxdt,
\end{equation*}
where $Q_T:=\Omega\times(0,T)$, and
\begin{equation*}
f_\Omega:=\frac{1}{|\Omega|}\int f,
\end{equation*}
where $|E|$ stands for the Lebesgue measure of the measurable set $E$.

Also, for all integer $k$ and $1\leq p<\infty$, $W^{k,p}$ is the standard Sobolev spaces as defined as follows:
\begin{equation*}
\begin{cases}
L^p:=L^p(\Omega),\,\,W^{k,p}=W^{k,p}(\Omega),\\
H^k:=W^{k,2},\,\,H^\infty:=\bigcap_{k\geq 1} H^k,\\
W^{k,p}_0=\overline{C_0^\infty}\,\,\mathrm{closure\,\,in\,\,the\,\,norm \,\,of\,\,} W^{k,p},\\
\norm{\cdot}_{B_1\cap B_2}:=\norm{\cdot}_{B_1}+\norm{\cdot}_{B_2}\text{ for two Banach spaces }B_1\text{ and }B_2,\\
H^1_\omega:=\{u\in H^1: u\cdot n=0,\,\curle u=-n^\perp\cdot B\cdot u\text{ on }\partial\Omega\},\\
H^1_{nd}:=\{u\in H^1: u=c_0\nabla\rho^{-1}\text{ on }\partial\Omega\},\\
V^{0,2}:=\{u\in L^2:\dive u=0,\,u\cdot n=0\text{ on }\partial\Omega\},\\
V_0^{1,2}:=\{u\in H^1_0:\dive u=0\},\quad V_0^{-1,2}:=[V_0^{1,2}]^*.
\end{cases}
\end{equation*}

For $0<\gamma<1$, we denote by $C^\gamma(\overline\Omega)$ the standard H$\mathrm{\ddot{o}}$lder space and $\rho\in C^{\gamma,\frac{\gamma}{2}}(\overline Q_T)$ the parabolic one, that is,
\begin{equation*}
C^{\gamma,\frac{\gamma}{2}}(\overline Q_T):=\left\{f\in C(\overline Q_T): \sup_{(x,t),(x',t')\in \overline Q_T\atop (x,t)\neq(x',t')} \frac{|f(x,t)-f(x',t')|}{|x-x'|^\gamma+|t-t'|^\frac{\gamma}{2}}<\infty\right\}.
\end{equation*}
The weak, weak* and strong convergence of a sequence $\{f^n\}$ are respectively denoted by
$$f^n\wconverge f,\quad f^n\wsconverge f,\quad f^n\sconverge f.$$
Finally, the transpose gradient is given by
\begin{align*}
\nabla^\perp := (\partial_2,-\partial_1).
\end{align*}
With this notion, one can write
\begin{equation*}
\begin{cases}
\curle u=\nabla^\perp \cdot u,\\
\Delta u=\nabla\dive u+\nabla^\perp \curle u.
\end{cases}
\end{equation*}

Now, we give the definitions of weak solutions and strong ones.
\begin{Definition}[Weak Solutions]\label{definition1.1}
$(\rho,u)$ is called a global weak solution, if the following regularity properties hold:
\begin{equation}
\begin{cases}
\alpha\leq \rho\leq \beta,\\
\rho\in C([0,T];H^1)\cap L^2(0,T;H^2),\quad
\rho_t\in L^2(0,T;L^2),\\
\begin{cases}
u\in L^\infty(0,T;L^2)\cap L^2(0,T;H^1_\omega),\quad\,\,\text{case }\eqref{A},\\
u\in L^\infty(0,T;L^2)\cap L^2(0,T;H^1_{nd}),\quad\text{case }\eqref{B},
\end{cases}
\end{cases}
\end{equation}
and $(\rho,u)$ statisfies \eqref{equation1.1} in the sense of distributions for all $T\in(0,\infty)$. More precisely, $\eqref{equation1.1}_1$, $\eqref{equation1.1}_3$ hold almost everywhere in $\Omega\times(0,T)$ and $\eqref{equation1.1}_2$ is satisfied in the following sense:
\begin{equation}
\begin{aligned}
\iint \rho u\cdot\phi_t+\rho u\otimes u:\nabla\phi-2\mu D(u):D(\phi)=-\int \rho_0u_0\cdot\phi(x,0),
\end{aligned}
\end{equation}
for $\phi\in C^\infty(\overline Q_T)$ with $\dive \phi=0$, $\phi(x,T)=0$, $x\in \Omega$ and $\phi=0$ on $\partial\Omega\times(0,T)$.
\end{Definition}

\begin{Definition}[Strong Solutions]
If $(\rho,u,\pi)$ is a solution such that \eqref{equation1.1} holds almost everywhere in $\Omega\times(0,T)$, $T\in(0,\infty)$, such that
\begin{equation}\label{strong}
\begin{cases}
\alpha\leq \rho \leq \beta,\\
\rho\in C([0,T];H^2)\cap L^2(0,T;H^3),\\
\rho_t\in C([0,T];L^2)\cap L^2(0,T;H^1),\\
u\in C([0,T];H^1)\cap L^2(0,T;H^2),\\
 u_t\in L^2(0,T;L^2),\\
 \pi\in L^2(0,T;H^1),
\end{cases}
\end{equation}
we call $(\rho,u,\pi)$ the strong solution on $\Omega\times(0,T)$. In particular, if $(\rho,u,\pi)$ satisfies \eqref{strong} for all $T\in (0,\infty)$, we say that $(\rho,u,\pi)$ is a global strong solution of the system \eqref{equation1.1}.
\end{Definition}

Our main results sate as following. The first two theorems concern with the existence results for $(\rho,u)$ satisfying \eqref{A} or \eqref{B}.
\begin{Theorem}\label{theorem1.2}
Suppose that $u_0\in L^2$ and $(\rho_0,u_0)$ satisfies the following compatibility condition
\begin{equation}\label{compat}
\begin{cases}
\dive u_0=c_0\Delta\rho_0^{-1},&x\in \Omega\\
u_0\cdot n=n\cdot\nabla\rho_0^{-1},&x\in \partial\Omega
\end{cases}
\end{equation}
Assume that $(\rho,u)$ satisfies the condition $\eqref{A}$ or $\eqref{B}$, then there exist a positive constant $\delta$ which only depends on $\Omega$, $\alpha$, $\beta$, $c_0$ and $\norm{v_0}_{L^2}$ such that, if $\norm{\nabla\rho_0}_{L^2}\le\delta$, problem \eqref{equation1.1} and \eqref{initial} admits at least one gobal weak solution.
\end{Theorem}

\begin{Theorem}\label{theorem1.3}
Suppose that $(\rho_0,u_0)$ satisfies \eqref{compat} and $(\rho,u)$ satisfies the condition $\eqref{A}$ or $\eqref{B}$. Let $u_0\in H^1_\omega$ provided $u$ satisfying the condition $\eqref{A}$; $u_0\in H^1_{nd}$ provided $u$ satisfying the condition $\eqref{B}$. In addition, let $\pi$ satisfy the normalized condition
\begin{equation}\label{normal}
\int \pi=0.
\end{equation}
 Then, if $\norm{\nabla\rho_0}_{L^2}\le\delta$ with the same $\delta$ obtained in Theorem \ref{theorem1.2}, the problem \eqref{equation1.1} and \eqref{initial} admits a unique global strong solution $(\rho,u,\pi)$. 
\end{Theorem}

Next, for the case when $(\rho,u)$ satisfying the condition \eqref{C}, we have
\begin{Theorem}\label{theorem1.6}
Suppose that $u_0\in H^1_0$ and $(\rho_0, u_0)$ satisfies \eqref{compat}. Suppose that $(\rho,u)$ satisfies the condition $\eqref{C}$ and $\pi$ satisfies \eqref{normal}. There exists a positive constant $\delta$ which only depends on $\Omega$, $\alpha$, $\beta$, $c_0$ such that, if $\norm{\nabla u_0}_{L^2}\le\delta$, then the problem \eqref{equation1.1} and \eqref{initial} admits a unique global strong solution $(\rho,u,\pi)$.
\end{Theorem}

At last, we give the  local existence result and the corresponding Serrin-type blowup criterion.
\begin{Theorem}\label{theo:Serrin-type}
Assume that $(\rho_0,u_0)$ satisfies \eqref{compat} and $u_0\in H^1_\omega,\,H^1_{nd}$ provided $u$ satisfying the condition \eqref{A} and \eqref{B}, respectively. Let $\pi$ saitisfies the condition \eqref{normal}. Then there exists a positive time $T_1<\infty$ depending on $\Omega$, $c_0$, $\alpha$, $\beta$ and $\norm{u_0}_{H^1}$ so that the problem \eqref{equation1.1} and \eqref{initial} admits an unique strong solution $(\rho, u,\pi)$ on $\Omega\times(0,T_1)$.

Moreover, if $\mu(\rho)=\mu$ is a positive constant and $u_0\in H^1_0$, then, the same result holds for $(\rho,u)$ satisfying the condition \eqref{C}
\end{Theorem}

\begin{Theorem}\label{Theorem1.7}
If $(\rho, u,\pi)$ is a strong solution of \eqref{equation1.1} on $\Omega\times(0,T^*)$ and $T^*<\infty$ is the maximal time of existence, then, one has 
\begin{enumerate}
\item[(1)]
\begin{equation}\label{serrin1}
\lim_{T\to T^*}\norm{\nabla\rho}_{L^s(0,T;L^r)}=\infty
\end{equation}
provided $(\rho,u)$ satisfying the condition \eqref{A} or \eqref{B};
\item[(2)]
\begin{equation}\label{serrin2}
\lim_{T\to T^*}\norm{u}_{L^s(0,T;L^r)}=\infty
\end{equation}
provided $(\rho,u)$ satisfying the condition \eqref{C}.
\end{enumerate}
Here, $r$ and $s$ satisfy the relation
\begin{equation}\label{rs}
\frac{2}{s}+\frac{2}{r}\leq 1,\quad 2<r\leq \infty.
\end{equation}
\end{Theorem}

\begin{Remark}
The definition of $v_0$ in Theorem \ref{theorem1.2}--\ref{theorem1.3} will be given at the end of this section.
\end{Remark}

\begin{Remark}
Theorem \ref{theorem1.2}--\ref{theo:Serrin-type} are first results concerning with the weak and strong solutions for the combustion model in bounded domain. Theorem \ref{theorem1.6} can be seen as a kind of extension of the global existence results in \cite{huang, zhang2015}  with $\dive u=c_0\Delta\rho^{-1}$. Theorem \ref{Theorem1.7} can be regarded as an extension to the classical Serrin's condition.
\end{Remark}

\begin{Remark}
For some technical reasons, in the proof of Theorem \ref{theorem1.3}, we need the following consistency condition 
\begin{equation}\label{equation1.7}
\rho_0|_{\partial\Omega}=\tilde\rho
\end{equation}
to ensure the continuity of $\rho$, which is crucial to the higher order estimates of $v$, see subsection \ref{Section6} for details. On the other hand, one should notice that the restriction $\alpha\leq \tilde\rho\leq\beta$ and the condition \eqref{equation1.7} are not necessary for the proof of Theorem \ref{theorem1.2}. However, for simiplicity, we may always impose these requirements.
\end{Remark}

\begin{Remark}\label{remark18}
Noticing that, in Theorem \ref{theorem1.2}--\ref{theo:Serrin-type}, we only impose the regularity restrictions on $u_0$ for given initial data $(\rho_0,u_0)$. This is due to the compatiability condition \eqref{compat} from which one can find that the regularity of $\rho_0$ can be totally determinded by that of $u_0$. Indeed, for example, if $u_0\in H^1_0$ as we assumed in Theorem \ref{theorem1.6}, it follows from the following epllitic problem
\begin{equation*}
\begin{cases}
c_0\Delta\rho_0^{-1}=\dive u_0, &x\in \Omega,\\
n\cdot\nabla\rho_0^{-1}=0, &x\in \partial\Omega.
\end{cases}
\end{equation*}
that, for all $1<p<\infty$,
\begin{equation*}
\begin{cases}
\norm{\nabla\rho_0}_{L^p}\leq C(p)\norm{u_0}_{L^p},\\
\norm{\nabla\rho_0}_{H^1}\leq C\norm{\nabla u_0}_{L^2}.
\end{cases}
\end{equation*}
Thus, alonging with the fact that $\rho_0\in L^\infty$, $\rho_0\in H^2$. We will come to this point again many times in later sections.
\end{Remark}

At the end of this section, we make a short comment on the analysis of this paper. Formally speaking, we treat Theorem \ref{theorem1.2}--\ref{theorem1.6} via two different types of decomposition and the proofs of Theorem \ref{theo:Serrin-type}--\ref{Theorem1.7} are based on those of Theorem \ref{theorem1.3}--\ref{theorem1.6}.

 The proofs of  Theorem \ref{theorem1.2}--\ref{theorem1.3} are based on the decomposition $u=v+c_0\nabla\rho^{-1}$, which may convert system \eqref{equation1.1} into the Kazhikhov-Smagulov type model. In this case, one can find that $v$ satisfies either the Dirichlet boundary condition or the slip one. More precisely, we may first  write in view of $\eqref{equation1.1}_3$
\begin{equation}\label{equation3.1}
v = u-c_0\nabla\rho^{-1}.
\end{equation}
Of course, such $v$ can be found for given $(\rho,u)$ with the boundary condition \eqref{A} or \eqref{B}. Next, using \eqref{equation3.1}, we write
$$\rho u=\rho v+c_0\rho\nabla\rho^{-1}=\rho v-c_0\nabla\log\rho.$$
Therefore, combining this equality and \eqref{equation3.1}, the original system \eqref{equation1.1} can be changed into the following equivalent formulations:
\begin{equation}\label{equation3.2}
\begin{cases}
\rho_t +v\cdot\nabla\rho + c_0\rho^{-2}\abs{\nabla\rho}^2 - c_0\rho^{-1}\Delta\rho=0,\\
\quad\\
\begin{cases}
(\rho v)_t +\dive(\rho  v\otimes v) - \dive{[2\mu D(v)]}+ \nabla \pi_1=c_0\dive{\left(2\mu\nabla^2\rho^{-1}\right)}\\
- c_0 \dive{\left(\rho v\otimes\nabla\rho^{-1}\right)} -\dive\left(c_0\rho\nabla\rho^{-1}\otimes v\right)-c_0^2 \dive{\left(\rho \nabla\rho^{-1}\otimes\nabla\rho^{-1}\right)},
\end{cases}\\
\quad\\
\dive{v} = 0,
\end{cases}
\end{equation}
where $\pi_1=\pi-c_0(\log\rho)_t$ is a modified pressure. 

Next, we give a precise defintion for the initial-boundary value of $v$. For given initial data $(\rho_0,u_0)$ satisfying the initial conditions \eqref{compat}, one can deduce that there exists a unique $v_0$ satisfying
\begin{equation}
v(x,0):=v_0=u_0-c_0\nabla\rho_0^{-1},\quad\dive v_0=0,\quad v_0\cdot n=0\text{ on }\partial\Omega,
\end{equation}
sharing with the similar compatibility conditions of $u_0$, that is, $v_0|_{\partial\Omega}=0$ provided  $u_0\in H^1_{nd}$ and 
$$\curle v_0=-n^\perp\cdot B\cdot (v_0+c_0\nabla\rho_0^{-1})$$
provided $u_0\in H^1_\omega$. Furthermore, from the relation \eqref{equation3.1}, we can define the boundary conditions of $v$ as follows:
\begin{enumerate}
\item[$(1)$]
$v\cdot n=0$ and $\curle v=\curle u=-n^\perp\cdot B\cdot (v+c_0\nabla\rho^{-1})$ on $\partial\Omega\times(0,T)$, if $(\rho,u)$ satisfies the condition \eqref{A}. In this case, from Remark \ref{remark2.6} in Section \ref{Section2}, one has
$$\norm{v}_{H^2}\leq C(\norm{\Delta v}_{L^2}+\norm{\Delta\rho^{-1}}_{L^2}).$$
\item[$(2)$]
$v=0$ on $\partial\Omega\times(0,T)$, if $(\rho,u)$ satisfies the condition \eqref{B}. In this case, we have
$$\norm{v}_{H^2}\leq C\norm{\Delta v}_{L^2}.$$
\end{enumerate}

An interesting observation is that, once the solution $(\rho,v)$ of \eqref{equation3.2}, which is defined as in Definition \ref{definition1.1}, incorporating with the initial-boundary conditions given above, is established, one can expect to obtain $u$ from \eqref{equation3.1} and, consequently, $(\rho,u)$ becomes the solution of the original system \eqref{equation1.1}. Therefore, in Section \ref{section3}, we mainly establish the a priori estimates of $(\rho,v)$. The details for proving the existence of $(\rho,u)$ will be shown in Section \ref{proof6}.

To sum up, we may impose $(\rho,v)$ satisfying one of the following two boundary conditions
\begin{enumerate}
\item[(1)]
if $(\rho,u)$ satisfies \eqref{A}, we impose
\begin{equation}\label{A'}
n\cdot \nabla\rho=0,\quad v\cdot n=0\text{ and }\curle v=-n^\perp\cdot B\cdot (v+c_0\nabla\rho^{-1})\text{ on }\partial\Omega\times(0,T);\tag{A'}
\end{equation}
\item[(2)]
if $(\rho,u)$ satisfies \eqref{B}, we impose
\begin{equation}\label{B'}
\rho=\tilde\rho,\quad v=0\text{ on }\partial\Omega\times(0,T)\tag{B'}
\end{equation}
\end{enumerate}
and our strategy of the proof can be concluded as follows:
$$\mathrm{given}\,\,(\rho_0,u_0)\implies (\rho_0,v_0)\implies \exists\,(\rho,v)\implies\exists\,(\rho,u).$$

Unfortunately, for Theorem \ref{theorem1.6}, such decomposition may cause some serious problems when it comes to the boundary estimates, that is, if we extract $v$ as we did above, $v|_{\partial\Omega}=-c_0\nabla\rho^{-1}$, which will hinder us to integrate by parts. As a consequence, we  consider another type of decomposition $u=w+Q$ coming from Lemma \ref{lemma27}. In every case that follows, $w$ is divergence-free and enjoys vanished boundary condition and $Q$ can be dominanted by $\nabla\rho$, which allows us to overcome the bounardy integrals, see Section \ref{Section8} for details.

The scond part we are interested in is the local well-posedness for system \eqref{equation1.1}. To prove Theorem \ref{theo:Serrin-type}, we mainly follow the proof from Kim-Cho \cite{cho} by using the iterative appoarch. This method will be based on the linearized model associated with  \eqref{equation1.1} , we refer to Section \ref{Section7} for details.

To the proof of Theorem \ref{Theorem1.7}, at least for the case when $(\rho,v)$ satisfies condition \eqref{A'} or \eqref{B'}, the key obeservation is that, if $\rho$ is a weak solution of system \eqref{equation1.1} satisfying $\nabla\rho\in L^s(0,T;L^r)$, then $(\rho,v)$ is regular, since we can close the lower bounds for $(\rho,v)$ merely under the condition \eqref{serrin1}. Here is an interpretation for \eqref{serrin1} and \eqref{serrin2}: for $(\rho,u)$ satisfying the condition \eqref{A} or \eqref{B}, the $L^s(0,T;L^r)$-norm of $v$ does not blowup during the finite time $[0,T^*)$, which is parallel to the classical Serrin's condition for 2D non-homogeneous Navier-Stokes equations \eqref{equationN} (since, in such case, problem \eqref{equationN}, at least for $\rho$ away from the vacuum, automatically satisies the Serrin's condition and admits a unique global strong solution without any smallness assumption, here $v$ can be seen as the velocity field $u$ in \eqref{equationN}). However, for the case when $(\rho,u)$ satisfies the condition \eqref{C}, we can not get rid off the the blowup behavior of $v$, since, in this case, $v|_{\partial\Omega}=-c_0\nabla\rho^{-1}$, which leads to some issue on the boundary estimates, we will come to this point again in Section \ref{section9}.

At last, we explain some techniques used in Section \ref{section3} and \ref{section9}. Since our main difficulty arises from the boundary integrals, in order to overcome it, we adapt the ideas from Cai-Li \cite{cai}: observing that the condition $v\cdot n|_{\partial\Omega}=0$ leads to
$$v=(v\cdot n^\perp)n^\perp,$$
which implies that 
$$\int_\partial v\cdot \nabla f=\int_\partial (v\cdot n^\perp)n^\perp\cdot \nabla f=\int \nabla f\cdot \nabla^\perp(v\cdot n^\perp).$$
This observation can allow us to avoid some higher derivatives of $f$, which has advantages over directly using the trace inequality, since the latter needs the second order derivative of $f$.

The rest of this paper is organized as follows. In Section \ref{Section2}, we give some elementary results which will be used in later. Section \ref{section3} is devoted to the lower order estimates, compactness results for weak solutions and the higher order estimates for Theorem \ref{theorem1.2}--\ref{theorem1.3}, while Section \ref{Section8} is devoted to the a priori estimets for Theorem \ref{theorem1.6}. In Section \ref{Section7}, we will use the contraction mapping theorem to prove Theorem \ref{theo:Serrin-type} and, then, in Section \ref{proof6}, use this result to establish the global existence for Theorem \ref{theorem1.2}--\ref{theorem1.6}. At last, in Section \ref{section9}, we will give the proof of Theorem \ref{Theorem1.7}.

\section{Preliminaries}\label{Section2}
First, we give the following Gagliardo-Nirenberg's inequalities which will be frequently used throughout the whole paper.
\begin{Lemma}[Gagliardo-Nirenberg's inequality \cite{leoni,nirenberg}]\label{Lemma2.3}
For all $u_i\in H^1$, $i=1,2$, $q_1\in (2,\infty)$ and $q_2\in (4,\infty)$, there exist positive constants $C_i,\tilde C_i$ depending on $q_i,\Omega$, $i=1,2$, such that 
\begin{align*}
\norm{u_1}_{L^{q_1}}&\leq C_1\norm{u_1}_{L^2}^{2/{q_1}}\norm{\nabla u_1}_{L^2}^{1-2/{q_1}}+\tilde C_1\norm{u_1}_{L^2},\\
\norm{u_2}_{L^{q_2}}&\leq C_2\norm{u_2}_{L^{4}}^{4/{q_2}}\norm{\nabla u_2}_{L^2}^{1-4/{q_2}}+\tilde C_2\norm{u_2}_{L^2}.
\end{align*}
In particular, if $u_i$ satisifies $u_i\cdot n=0$ on $\partial\Omega$ or $(u_i)_\Omega=0$, then one can take $\tilde C_1=\tilde C_2=0$.
\end{Lemma}

\begin{Lemma}[\cite{aramaki,von}]\label{Lemma2.2}
Let $\Omega$ be a simply connected bounded domain in $\mathbb{R}^2$ with smooth boundary. Assume that $1<p<\infty$. There exists a positive constant $C=C(p,\Omega)$ such that
\begin{equation*}
\norm{\nabla u}_{L^p}\leq C\left(\norm{\dive{u}}_{L^p}+\norm{\curle{u}}_{L^p}\right),
\end{equation*}
for all $u\in W^{1,p}$ with $u\cdot n=0$ on $\partial \Omega$. Furthermore, for $u \in W^{2, p}$ with $u \cdot n=0$ on $\partial \Omega$, there exists a constant $C=C(p,\Omega)$ such that
\begin{equation*}
\norm{u}_{W^{2, p}} \leq C\left(\norm{\dive u}_{W^{1, p}}+\norm{\curle u}_{W^{1, p}}+\norm{u}_{L^p}\right) .
\end{equation*}
\end{Lemma}

\begin{Remark}\label{remark2.4}
For case of use, we list the following equivalent norms for $\rho$ satisfying the Neumann or the non-homogeneous Dirichlet condition. Let $1<p<\infty$, using Lemma \ref{Lemma2.3}--\ref{Lemma2.2}, if $\rho$ satsifies the Neumann condition, one has, for all $t\geq 0$,
\begin{gather*}
\rho_\Omega=(\rho_0)_\Omega,\\
\norm{\nabla \rho}_{L^p}\leq C\normf{\nabla^2\rho}_{L^p}\leq C\norm{\Delta\rho}_{L^p}\leq C\norm{\nabla\Delta\rho}_{L^p},
\end{gather*}
and
\begin{align*}
C^{-1}(\norm{\nabla\rho}_{L^p}+(\rho_0)_\Omega)\leq \norm{\rho}_{W^{1,p}} &\leq C(\norm{\nabla\rho}_{L^p}+(\rho_0)_\Omega),\\
C^{-1}(\norm{\Delta\rho}_{L^p}+(\rho_0)_\Omega)\leq \norm{\rho}_{W^{2,p}} &\leq C(\norm{\Delta\rho}_{L^p}+(\rho_0)_\Omega),\\
C^{-1}(\norm{\nabla\Delta\rho}_{L^p}+(\rho_0)_\Omega)\leq\norm{\rho}_{W^{3,p}}&\leq C(\norm{\nabla\Delta \rho}_{L^p}+(\rho_0)_\Omega),
\end{align*}
for some positive constant $C=C(p,\Omega)$. 

If $\rho$ satisfies the the non-homogeneous Dirichlet condition, then there exists a positive constant $C=C(p,\Omega)$ such that
\begin{align*}
C^{-1}\norm{\nabla\rho}_{L^p}\leq \norm{\rho-\tilde\rho}_{W^{1,p}} &\leq C\norm{\nabla\rho}_{L^p},\\
C^{-1}\norm{\Delta\rho}_{L^p}\leq \norm{\rho-\tilde\rho}_{W^{2,p}} &\leq C\norm{\Delta\rho}_{L^p},\\
C^{-1}\norm{\nabla\Delta\rho}_{L^p}\leq\norm{\rho-\tilde\rho}_{W^{3,p}}&\leq C\norm{\nabla\Delta \rho}_{L^p}.
\end{align*}
In both cases, the following Gagliardo-Nirenberg's inequalities are established 
\begin{align*}
\norm{\nabla\rho}_{L^{q_1}}&\leq C\norm{\nabla\rho}_{L^2}^{2/{q_1}}\norm{\Delta\rho}_{L^2}^{1-2/{q_1}},\\
\norm{\nabla\rho}_{L^{q_2}}&\leq C\norm{\nabla\rho}_{L^{4}}^{4/{q_2}}\norm{\Delta\rho}_{L^2}^{1-4/{q_2}},\\
\norm{\Delta\rho}_{L^{q_1}}&\leq C\norm{\Delta\rho}_{L^2}^{2/{q_1}}\norm{\nabla\Delta\rho}_{L^2}^{1-2/{q_1}}.
\end{align*}
where $q_1,\,q_2$ as in Lemma \ref{Lemma2.3}.
\end{Remark}

Next, for the problem
\begin{equation}\label{25}
\begin{cases}
\dive v=f, & x \in \Omega, \\ 
v=\Phi, & x \in \partial \Omega\end{cases}
\end{equation}
one has the following conclusion which will be frequently used to eliminate the non-homogeneity of equations in Section \ref{Section8}.
\begin{Lemma}[\cite{galdi}, Theorem III.3.3]\label{lemma27}
Suppose that $\Phi\cdot n=0$ on $\partial\Omega$ and $f_\Omega =0$. Then,
\begin{enumerate}
\item[1)]
If $\Phi=0$, there exists a bounded linear operator $\mathcal{B}=\left[\mathcal{B}_1, \mathcal{B}_2\right]$,
\begin{equation*}
\mathcal{B}: \{f\in L^{p}:f_\Omega =0\} \mapsto \left[W_0^{1,p}\right]^2
\end{equation*}
such that
\begin{equation*}\label{225}
\|\mathcal{B}[f]\|_{W^{1, p}} \leq C(p)\|f\|_{L^{p}},
\end{equation*}
for all $p \in(1, \infty)$, and the function $Q=\mathcal{B}[f]$ solves the problem \eqref{25}. Moreover, if $f=\dive g$ with a certain $g \in L^r,\left.g \cdot n\right|_{\partial \Omega}=0$, then for any $r \in(1, \infty)$
\begin{equation*}\label{226}
\|\mathcal{B}[f]\|_{L^r} \leq C(r)\|g\|_{L^r} .
\end{equation*}
$\cB$ is so-called the Bogovski\v i operator.
\item[2)]
If $f=0$, there exists a bounded linear operator $\cC=[\cC_1,\cC_2]$,
$$\cC: \{\Phi: \Phi\cdot n|_{\partial\Omega}=0,\,\,\dive\Phi\in L^p\}\mapsto  \left[W^{1,p}\right]^2$$
such that
$$\norm{\cC[\Phi]}_{W^{1,p}}\leq C(p)\norm{\dive \Phi}_{L^p},$$
for all $p\in (1,\infty)$ and the function $R=\cC[\Phi]$ sovles the problem \eqref{25}.
\end{enumerate}
\end{Lemma}
\begin{proof}
We only give a brief proof for $(2)$. By a simply change 
$$\tilde v=v-\Phi,$$
it follows from \eqref{25} that $\tilde v$ satisfies
\begin{equation}\label{25'}
\begin{cases}
\dive \tilde v=-\dive \Phi,&x\in \Omega,\\
\tilde v=0,&x\in \partial\Omega.
\end{cases}
\end{equation}
Thus, applying $(1)$ for \eqref{25'}, we finish the proof.
\end{proof}

Lemma \ref{lemma2.5}--\ref{lemma8.4} are a series of results relating to the Stokes system which are vital to the higher order estimates of $v$ and the construction of smooth initial data. These lemmas will be frequently used in Section \ref{section3}--\ref{section9}. 
\begin{Lemma}\label{lemma2.5}
Let $\Omega$ be a simply connected bounded domain in $\mathbb{R}^2$ with smooth boundary. Let $(u,p)$ satisfy the following equations
\begin{equation}\label{equation2.1}
\begin{cases}
-\Delta u+\nabla p=F, &x\in\Omega,\\
\dive u=0, &x\in\Omega,
\end{cases}
\end{equation}
where $F\in L^2$, $\int p=0$. There exists a positive constant $C$ depending only on $\Omega$ such that
\begin{enumerate}
\item[(1)]
if $u|_{\partial\Omega}=\Phi$, where $\Phi\in H^2$ is a function defined on $\Omega$, then
\begin{equation}\label{equation2.2'}
\norm{u}_{H^2} +\norm{p}_{H^1}\leq C(\normf{F}_{L^2}+\norm{\Phi}_{H^2});
\end{equation}
\item[(2)]
if $u\cdot n=0,\,\curle u=\varphi$ on $\partial\Omega$,
where $\varphi\in H^1$ is a function defined on $\Omega$, then
\begin{equation}\label{equation2.2}
\norm{u}_{H^2} +\norm{p}_{H^1}\leq C(\normf{F}_{L^2}+\norm{\varphi}_{H^1}).
\end{equation}
\end{enumerate}
\end{Lemma}
\begin{proof}
Since \eqref{equation2.2'} can be finded from \cite{galdi}, Theorem IV.6.1, we only prove the case of slip condition. Using the identity $\Delta u=\nabla \dive u+\nabla^\perp \curle u$ and integrating by parts, one has
\begin{equation*}
\int |\curle u|^2-\int_{\partial}\varphi(u\cdot n^\perp)=\int F\cdot u,
\end{equation*}
which implies that, using Lemma \ref{Lemma2.3}--\ref{Lemma2.2} and the trace inequality,
\begin{equation}\label{l2.6}
\norm{u}_{H^1}\leq C(\normf{F}_{L^2}+\norm{\varphi}_{H^1}).
\end{equation}
Since $\nabla p$ is bounded in $H^{-1}$, it follows form the condition $\int p=0$ that $p$ is bounded in $L^2$. Next, taking $\curle$ on the both side of $\eqref{equation2.1}_1$ leads to
$$-\Delta(\curle u-\varphi)=\curle F-\Delta\varphi,$$ 
with boundary condition $\curle u-\varphi=0$. Then, using the regularity result of elliptic partial differential equations, we have
\begin{equation*}
\norm{\curle u}_{H^1}\leq C\normf{\curle F-\Delta\varphi}_{H^{-1}}+C\norm{\varphi}_{H^1}\leq C(\normf{F}_{L^2}+\norm{\varphi}_{H^1}).
\end{equation*}
Then, using again Lemma \ref{Lemma2.2} and \eqref{l2.6} gives
\begin{equation}\label{equation2.4}
\norm{u}_{H^2}\leq C(\normf{F}_{L^2}+\norm{\varphi}_{H^1}+\norm{u}_{L^2}).
\end{equation}
and, consequently, the estimate of $p$ is followed easily. It remains to omit the terms $\norm{u}_{L^2}$ on the right-hand side of \eqref{equation2.4}. Indeed, this is a simple consequence of the uniqueness of \eqref{equation2.1} and we leave the proof to the reader.
\end{proof}

\begin{Remark}\label{remark2.6}
Slimilar results for the Laplace equations $-\Delta u= F$ instead of \eqref{equation2.1} with the same boundary conditions can be found in \cite{gilbarg}. 
\end{Remark}

Next, we give a lemma which indicates that $\rho\in C^{\gamma,\frac{\gamma}{2}}(\overline Q_T)$ for some $\gamma\in(0,1)$ provided $v$ satisfying the Serrin's condition. This result is critial to the estimate of $\Delta v$ which will be used in Section \ref{section3} and \ref{section9}. The observation is based on the fact that $\dive v=0$.
\begin{Lemma}[\cite{aaa,SUN2013}]\label{lemma61}
Let $v\in L^s(0,T;L^r)$ for some $r,s$ satisfying \eqref{rs}, $\dive v=0$, $v\cdot n=0$ and $\rho\in C([0,T];L^2)\cap L^2(0,T;H^1)$ be the weak solution of equation $\eqref{equation3.2}_1$ (in the sense of distributions), $\alpha\leq \rho\leq\beta$. Let $\rho$ satisfy either the Neumann condition 
$$n\cdot\nabla\rho=0\text{ on }\partial\Omega\times(0,T)$$
or the non-homogeneous Dirichlet condition
$$\rho=\tilde\rho\text{ on }\partial\Omega\times(0,T).$$
Suppose that $\rho_0\in C^{\gamma_0}(\overline\Omega)$ for some $\gamma_0\in (0,1)$, then $\rho$ is H$\ddot{o}$lder continuous. More precisely, $\rho\in C^{\gamma,\frac{\gamma}{2}}(\overline Q_T)$, for some $\gamma$ depending only on $\gamma_0$, $\alpha$ and $\beta$.
\end{Lemma}
\begin{proof}
We only give the proof for $\rho|_{\partial\Omega}=\tilde \rho$, since the case for $\rho$ satisfying the Neumann boundary condition has been proved in \cite{aaa,SUN2013}. Let $\zeta$ be a cut-off function, $\support\zeta\subset B_r\times[t_0,t_0+\tau]$, where $B_r$ is an arbitrary ball contained in $\Omega$ and $[t_0,t_0+\tau]\subset (0,T)$, $0<\tau<1$. Multiplying $\zeta^2(\rho-k)_+$ on $\eqref{equation3.2}_1$ and integrating by parts leads to
\begin{equation}\label{eq6.1}
\begin{aligned}
&\frac{1}{2}\sup_{t\in[t_0,t_0+\tau]}\norm{\zeta(\rho-k)_+}_{L^2}^2 +\nu\norm{\zeta\nabla(\rho-k)_+}_{L^2}^2\\
&\leq \frac{1}{2}\norm{\zeta(\rho-k)_+}_{L^2}^2(t_0)+C\int_{t_0}^{t_0+\tau}\int_\Omega \left(\abs{\nabla\zeta}^2+\zeta\abs{\zeta_t}\right)(\rho-k)_+^2\,dxdt \\
&\quad-\int_{t_0}^{t_0+\tau}\int_\Omega (v\cdot\nabla\rho)\zeta^2(\rho-k)_+\,dxdt.
\end{aligned}
\end{equation}
For the last term on the right-hand side of \eqref{eq6.1}, using Lemma \ref{Lemma2.3}, we have
\begin{equation*}
\begin{aligned}
&\abs{\int_{t_0}^{t_0+\tau}\int_\Omega( v\cdot\nabla\rho)\zeta^2(\rho-k)_+\,dxdt}\\
&=\abs{\int_{t_0}^{t_0+\tau}\int_\Omega (v\cdot\nabla\zeta)\zeta(\rho-k)_+^2\,dxdt}\\
&\leq\norm{v}_{L^{\frac{2r}{r-2}}_{t}L^r_x}\norm{\zeta(\rho-k)_+}_{L^r_{t}L^{\frac{2r}{r-2}}_x}\norm{\abs{\nabla\zeta}(\rho-k)_+}_{L^2_{t,x}}\\
&\leq	 C_\varepsilon\tau^{\frac{rs-2s-2r}{2rs}}\norm{\abs{\nabla\zeta}(\rho-k)_+}_{L^2_{t,x}}^2 + \varepsilon\left(\sup_{t\in[t_0,t_0+\tau]}\norm{\zeta(\rho-k)_+}_{L^2}^2 +\norm{\zeta\nabla(\rho-k)_+}_{L^2_{t,x}}^2\right),
\end{aligned}
\end{equation*}
which, alonging with \eqref{eq6.1}, implies that
\begin{equation}\label{eq4.4}
\begin{aligned}
&\sup_{t\in[t_0,t_0+\tau]}\norm{\zeta(\rho-k)_+}_{L^2}^2 +\nu\norm{\zeta\nabla(\rho-k)_+}_{L^2}^2\\
&\leq \norm{\zeta(\rho-k)_+}_{L^2}^2(t_0)+C\int_{t_0}^{t_0+\tau}\int_\Omega \left(\abs{\nabla\zeta}^2+\zeta\abs{\zeta_t}\right)(\rho-k)_+^2\,dxdt.
\end{aligned}
\end{equation}
The inequality above is valid for all $k\in \mathbb{R}$. Then, It follows from \cite{ladyzhenskaia} Theorem 10.1 that $\rho\in C^{\gamma,\frac{\gamma}{2}}(Q_T)$, for some $\gamma\in(0,1)$. 

For the boundary estimates, if $\rho=\tilde\rho$ on $\partial\Omega$, we still use $\zeta$ and choose arbitrary $B_r\times[t_0,t_0+\tau]\subset \RR^2\times[0,T]$, where $B_r$ may intersect $\Omega$. Then, \eqref{eq4.4} holds for $k$ sufficiently large, since $(\rho-k)_+$ has vanished boundary, which implies that $\rho\in C^{\gamma,\frac{\gamma}{2}}(\overline Q_T)$. 
\end{proof}

Once $\rho$ is H$\mathrm{\ddot{o}}$lder continuous, $\mu(\rho(x,t))$ is continuous on $\overline Q_T$ and, thus, we have the following estiamtes for the non-divergence type Stokes system.
\begin{Lemma}\label{Lemma5.2}
Let $(v,p)$ be a strong solution of the following Stokes system, 
\begin{equation}
\begin{cases}
-\mu(x)\Delta v +\nabla p = F,&x\in \Omega\\
\dive v=0,&x\in \Omega
\end{cases}
\end{equation}
where $\mu(x)\in C(\overline{\Omega})$, $\mu\in[\underline\mu,\overline\mu]$, $\int p=0$ and $F\in L^2$. Then there exists a positive constant $C$ depending only on $\underline\mu,\overline\mu$, continuity module of $\mu$ and $\Omega$ such that
\begin{enumerate}
\item[(1)]
if $u|_{\partial\Omega}=\Phi$, where $\Phi\in H^2$ is a function defined on $\Omega$, then
\begin{equation}
\norm{u}_{H^2} +\norm{p}_{H^1}\leq C(\normf{F}_{L^2}+\norm{\Phi}_{H^2});
\end{equation}
\item[(2)]
if $u\cdot n=0,\,\curle u=\varphi$ on $\partial\Omega$,
where $\varphi\in H^1$ is a function defined on $\Omega$, then
\begin{equation}\label{equa5.5}
\norm{u}_{H^2} +\norm{p}_{H^1}\leq C(\normf{F}_{L^2}+\norm{\varphi}_{H^1}).
\end{equation}
\end{enumerate}
\end{Lemma}
\begin{proof}
The proof of Lemma \ref{Lemma5.2} can be simply derived by using the freezing point argument, since we already have the conclusion when $\mu\equiv \mathrm{constant}$ from Lemma \ref{lemma2.5}. 
\end{proof}

Furthermore, in order to prove Lemma \ref{lemma8.2} (see Section \ref{Section8}), we need the following auxiliary lemma. The purpose for using such result will be explained in the proof of Lemma \ref{lemma8.2}.
\begin{Lemma}\label{lemma8.4}
Let $(v,p)$ be a strong solution of the following Stokes system, 
\begin{equation}\label{8.13}
\begin{cases}
-\dive[2\mu D(v)] +\nabla p = F,&x\in \Omega\\
\dive v=0,&x\in \Omega
\end{cases}
\end{equation}
where $\nabla\mu(\rho)\in L^4$, $\mu$ is smooth and $0<\underline\mu\leq \mu\leq \overline\mu<\infty$, $\int p=0$ and $F\in L^2$. Then there exists a positive constant $C$ depending only on $\underline\mu,\,\overline\mu$ and $\Omega$ such that
\begin{enumerate}
\item[(1)]
if $v|_{\partial\Omega}=\Phi$, where $\Phi\in H^2$ is a function defined on $\Omega$, then
\begin{equation}
\norm{v}_{H^2}+\norm{p}_{H^1}\leq C\left[\left(\norm{\nabla\mu}^{2}_{L^4}+1\right)\left(\norm{F}_{L^2}+\norm{\nabla\Phi}_{H^1}\right)+\norm{\nabla\mu}^{2}_{L^4}\norm{\nabla v}_{L^2}\right];
\end{equation}
\item[(2)]
if $v\cdot n=0,\,\curle v=\varphi$ on $\partial\Omega$,
where $\varphi\in H^1$ is a function defined on $\Omega$, then
\begin{equation}
\norm{v}_{H^2}+\norm{p}_{H^1}\leq C\left[\left(\norm{\nabla\mu}^{2}_{L^4}+1\right)\left(\norm{F}_{L^2}+\norm{\varphi}_{H^1}\right)+\norm{\nabla\mu}^{2}_{L^4}\norm{\nabla v}_{L^2}\right].
\end{equation}
\end{enumerate}
\end{Lemma}
\begin{proof}
We only give the proof for $(1)$. First of all, we can use Lemma \ref{lemma27} to find a function $R=\cC[\Phi]$ such that $\dive R=0$ and $R|_{\partial\Omega}=\Phi$, then $\eqref{8.13}_1$ becomes
\begin{equation*}
-\dive[2\mu(\rho)D(v-R)] +\nabla p = F+\dive[2\mu(\rho)D(R)].
\end{equation*}
Using standard energy approach and the fact $\norm{R}_{H^1}\leq C\norm{\nabla\Phi}_{L^2}$, one has
\begin{equation}\label{8.15}
\norm{\nabla v}_{L^2}+\norm{p}_{L^2}\leq C(\norm{F}_{H^{-1}}+\norm{\nabla\Phi}_{L^2})\leq C(\norm{F}_{L^2}+\norm{\nabla\Phi}_{H^1}).
\end{equation}
Next, rewritting $\eqref{8.13}_1$ into the form
\begin{equation*}
-\Delta v+\nabla\left[\frac{p}{\mu(\rho)}\right]=\frac{F}{\mu(\rho)}+\frac{2\mu'\nabla\rho\cdot D(v)}{\mu(\rho)} -\frac{p\mu'}{\mu(\rho)^2}\nabla\rho,
\end{equation*}
using Lemma \ref{lemma2.5}, we have 
\begin{equation*}
\norm{v}_{H^2}+\norm{p}_{H^1}\leq C\left[\norm{F}_{L^2}+\norm{\nabla\Phi}_{H^1}+\norm{\nabla\mu(\rho)}_{L^4}\left(\norm{\nabla v}_{L^4}+\norm{p}_{L^4}\right)\right],
\end{equation*}
which, using Lemma \ref{Lemma2.3} and \eqref{8.15}, leads to
\begin{equation*}
\norm{v}_{H^2}+\norm{p}_{H^1}\leq C\left[\left(\norm{\nabla\mu(\rho)}^{2}_{L^4}+1\right)\left(\norm{F}_{L^2}+\norm{\nabla\Phi}_{H^1}\right)+\norm{\nabla\mu(\rho)}^{2}_{L^4}\norm{\nabla v}_{L^2}\right].
\end{equation*}
Thus, we complete the proof.
\end{proof}

At last, in subsection \ref{section4} and Section \ref{proof6}, we need the following lemma.
\begin{Lemma}[Simon \cite{novotny, simon}]\label{Lemma2.6}
Let $X\hookrightarrow B\hookrightarrow Y$ be three Banach spaces with compact imbedding $X\hookrightarrow\hookrightarrow Y$. Further, let there eixst $0<\theta<1$ and $M>0$ such that
\begin{equation*}
\norm{v}_{B}\leq M\norm{v}_X^{1-\theta}\norm{v}_Y^\theta,\,\,\, \text{for all}\,\,v\in X\cap Y.
\end{equation*}
Denote for $T>0$,
\begin{equation*}
W(0,T):= W^{s_0,r_0}(0,T;X)\cap W^{s_1,r_1}(0,T;Y)
\end{equation*}
with $s_0,s_1\in\RR$, $r_1,r_0\in [1,\infty]$, and 
\begin{equation*}
s_\theta:=(1-\theta)s_0+\theta s_1,\,\,\frac{1}{r_\theta}:=\frac{1-\theta}{r_0}+\frac{\theta}{r_1},\,\,s^*:=s_\theta-\frac{1}{r_\theta}.
\end{equation*}
Assume that $s_\theta>0$ and $F$ is a bounded set in $W(0,T)$.
\begin{enumerate}
\item[(1)]
If $s^*\leq 0$, then $F$ is precompact in $L^p(0,T;B)$ for all $1\leq p<-\frac{1}{s^*}$.
\item[(2)]
If $s^*> 0$, then $F$ is precompact in $C([0,T];B)$. 
\end{enumerate}
\end{Lemma}

\section{A Priori Estimates (I): Case \eqref{A} and \eqref{B}}\label{section3}

In this section, we are going to establish the a priori bounds for $(\rho,v)$ which will be used in the proof of Theorem \ref{theorem1.2}. Throughout this section, let $T\in (0,\infty)$ and $(\rho,v)$ be a smooth solution to \eqref{equation3.2} with smooth data $(\rho_0,v_0)$. Moreover, in order to simplify the notation, we always denote by $\varepsilon$, $\varepsilon_i$, $i\in \NN_+$, the arbitrarily small number belongs to $(0,1/2]$, and we use the subscripts $C_\varepsilon$, $C_{\varepsilon_i}$ to emphasize the dependency of the constant $C$ on $\varepsilon$, $\varepsilon_i$

\subsection{Lower Order Estimates}\label{LOE}

The first lemma is a consequence of the standard maximal principle.
\begin{Lemma}\label{lemma4.1}
Let $\alpha\leq \rho_0\leq \beta$ and $(\rho,v)$ satisfy the condition \eqref{A'} or \eqref{B'},  one has 
$\alpha\leq\rho(x,t)\leq\beta$ for $x\in\Omega$ and all $t\in [0,T]$.
\end{Lemma}
\begin{proof}
We only prove the upper bound, since the lower one can be derived in a similar way. Using \eqref{equation3.1}, we convert the equation \eqref{equation3.2} into the form
\begin{equation}\label{equation4.1}
\rho_t +v\cdot\nabla\rho  - c_0\Delta \log\rho =0.
\end{equation}
If $(\rho,v)$ satisfies the condition \eqref{A'}, set $k = \beta$ and multiply \eqref{equation4.1} by $(\rho-k)_+:=\max\{\rho-k,0\}$. After integrating by parts, we obtain
\begin{equation}\label{equation4.2}
\frac{d}{dt}\int \frac{1}{2}(\rho-k)_+^2  + \int c_0\rho^{-1}\abs{\nabla(\rho-k)_+}^2 = 0,
\end{equation}
where we have used the identity
\begin{equation*}
\int v\cdot\nabla\rho (\rho-k)_+ = \int v\cdot\nabla(\rho-k)_+ (\rho-k)_+= 0,
\end{equation*}
since $v$ is divergence-free and $v\cdot n=0$ on $\partial\Omega$. Hence, integrating \eqref{equation4.2} from $0$ to $T$ and then, using \eqref{initial} implies that
\begin{equation*}
\sup_{t\in[0,T]}\int (\rho-k)_+^2  \leq \int (\rho_0-k)_+^2 = 0,
\end{equation*}
which implies that $\rho \leq k=\beta$ for all $x\in\Omega$ and $t\in [0,T]$. The case when $(\rho,v)$ satisfies the condition \eqref{B'} can be proved analogously. Thus, we complete the proof of Lemma \ref{lemma4.1}.
\end{proof}

Our main purpose in this subsection is establishing the lower order bounds. We aim to prove the following proposition.
\begin{Proposition}\label{theorem4.5}
Let $(\rho,v)$ satisfy the condition \eqref{A'} or \eqref{B'}. Suppose that 
\begin{equation}\label{condition3.3}
\sup_{t\in[0,T]}\norm{\nabla\rho}^2_{L^2} +\int_0^T\left(\norm{\nabla \rho}^4_{L^4}+\norm{\Delta\rho}^2_{L^2}\right)\,dt  \leq 2.
\end{equation}
There exists a positive constant $\delta$ depending on $\Omega$, $c_0$, $\alpha$, $\beta$ and $\norm{v_0}_{L^2}$ such that, if $\norm{\nabla\rho_0}_{L^2}\leq \delta$,
\begin{equation}\label{3344}
\sup_{t\in[0,T]}\norm{\nabla\rho}^2_{L^2} +\int_0^T\left(\norm{\nabla \rho}^4_{L^4}+\norm{\Delta\rho}^2_{L^2}\right)\,dt  \leq 1.
\end{equation}
\end{Proposition}

We give the proof of Proposition \ref{theorem4.5} in several steps. First, we estimate the first order derivative of $\rho$, which is given by the following lemma.
\begin{Lemma}\label{lemma43}
There exists a positive constant $C$ depending only on $c_0$ and $\beta$ such that, if $(\rho,v)$ satisfies the condition \eqref{A'}, 
\begin{equation}
\sup_{t\in[0,T]}\norm{\rho}_{L^2} +\norm{\nabla\rho}_{L^2(0,T;L^2)}  \leq C\norm{\rho_0}_{L^2};\label{equation4.3}
\end{equation}
if $(\rho,v)$ satisfies the condition \eqref{B'}, 
\begin{equation}
\sup_{t\in[0,T]}\norm{\rho- \tilde{\rho}}_{L^2} +\norm{\nabla\rho}_{L^2(0,T;L^2)}  \leq C\norm{\rho_0-\tilde{\rho}}_{L^2}.\label{equation4.4}
\end{equation}
\end{Lemma}
\begin{proof}
Multiplying \eqref{equation4.1} by $\rho$ (if $\rho$ satisfies the condition \eqref{B'}, multiply $\rho-\tilde\rho$), integrating over $\Omega$ and computing in the same way of Lemma \ref{lemma4.1}, one has
\begin{equation}\label{448}
\frac{d}{dt}\norm{\rho}_{L^2}^2 +2c_0\beta^{-1}\norm{\nabla\rho}_{L^2}^2\leq 0.
\end{equation}
Using Gr$\mathrm{\ddot{o}}$nwall's inequality leads to
\begin{equation*}
\sup_{t\in[0,T]}\norm{\rho}_{L^2} + \sqrt{2c_0\beta^{-1}}\norm{\nabla\rho}_{L^2(0,T;L^2)}  \leq \norm{\rho_0}_{L^2},
\end{equation*}
where we use the fact that $\rho\leq \beta$ from Lemma \ref{lemma4.1}. This completes the proof.
\end{proof}
Next, the following lemma shows that the second order derivative of $\rho$ can be dominated by the norm of $v$ provided $\norm{\nabla\rho}_{L^2}(t)$ is small enough.
\begin{Lemma}\label{Lemma4.3}
Let $(\rho,v)$ satisfy the condition \eqref{A'} or \eqref{B'}. Then there exist a positive constant $\delta_1$ depending on $\Omega$, $c_0$, $\alpha$, $\beta$ and a positive constant $C$ depending on $\Omega$, $c_0$, $\alpha$ and $\beta$ such that, if $\norm{\nabla\rho}_{L^2}(t)\leq \delta_1$ for all $t\in [0,T]$,
\begin{equation}\label{equa3.5}
\sup_{t\in[0,T]}\norm{\nabla\rho}^2_{L^2} +\int_0^T\left(\norm{\nabla\rho}_{L^4}^4+\norm{\Delta\rho}_{L^2}^2\right)\,dt  \leq \exp\left\{C\int_0^T\norm{v}^4_{L^4}\,dt\right\}\norm{\nabla\rho_0}_{L^2}^2.
\end{equation}
\end{Lemma}
\begin{proof}
Multiplying \eqref{equation3.2} by $(-\Delta\rho)$ and integrating over $\Omega$, we obtain
\begin{equation*}
\frac{d}{dt}\int \frac{1}{2}\abs{\nabla\rho}^2  +\int c_0\rho^{-1}\abs{\Delta\rho}^2=\int (v\cdot\nabla\rho)\Delta\rho - \int c_0\rho^{-2}\abs{\nabla\rho}^2\Delta\rho,
\end{equation*}
which implies that, using Lemma \ref{lemma4.1},
\begin{equation*}
\begin{aligned}
\frac{d}{dt}\int \frac{1}{2}\abs{\nabla\rho}^2 + c_0 \beta^{-1}\int \abs{\Delta\rho}^2&\leq C\int \left(\abs{\nabla\rho}^2+\abs{v}\abs{\nabla\rho}\right)\abs{\Delta\rho}\\
&\leq C\int \left(\abs{\nabla\rho}^4+\abs{v}^2\abs{\nabla\rho}^2\right) + c_0(2\beta)^{-1}\int \abs{\Delta\rho}^2.
\end{aligned}
\end{equation*}
Hence, by Lemma \ref{Lemma2.3} and \ref{lemma4.1}, we have
\begin{equation}\label{eq3.12}
\frac{d}{dt}\norm{\nabla\rho}_{L^2}^2 +\nu\norm{\Delta\rho}_{L^2}^2\leq C\norm{\nabla\rho}_{L^2}^2\norm{\Delta\rho}_{L^2}^2 + C\norm{v}^4_{L^4}\norm{\nabla\rho}^2_{L^2},
\end{equation}
 for some positive constant $\nu=\nu(c_0,\beta)$ and $C=C(\Omega,c_0,\alpha,\beta)$. Thus, if we choose $\delta_1=\nu^{1/2}(2C)^{-1/2}$ and set $\norm{\nabla\rho}_{L^2}(t)\leq \delta_1$ for all $t\in [0,T]$, using the Gr$\mathrm{\ddot{o}}$nwall's inequality, we can deduce from \eqref{eq3.12} and Lemma \ref{Lemma2.3} that
\begin{equation*}
\sup_{t\in[0,T]}\norm{\nabla\rho}_{L^2}^2 +\nu\int_0^T\left(\norm{\nabla\rho}_{L^4}^4+\norm{\Delta\rho}_{L^2}^2\right)\,dt\leq \exp\left\{C\int_0^T\norm{v}^4_{L^4}\,dt\right\}\norm{\nabla\rho_0}_{L^2}^2,
\end{equation*}
which concludes the proof of \eqref{equa3.5}. The case when $(\rho,v)$ satisfies the condition \eqref{B'} can be computed in the same way, since $\rho_t$ has vanished boundary.
\end{proof}

From the observation of Lemma \ref{Lemma4.3}, in order to derive the bounds for $\rho$, we need to control the $L^4(0,T;L^4)$ norm of $v$, which is given by the following lemma.
\begin{Lemma}\label{Lemma3.4}
Let $(\rho,v)$ satisfy the condition \eqref{A'} or \eqref{B'}. Suppose that condition \eqref{condition3.3} holds.
Then there exists a positive constant $C$ depending on $\Omega$, $c_0$, $\alpha$ and $\beta$ such that 
\begin{equation}\label{equa3.8}
\sup_{t\in[0,T]}\norm{v}^2_{L^2} +\int_0^T\left(\norm{v}^4_{L^4}+\norm{\nabla v}^2_{L^2}\right)\,dt\leq C(1+ \norm{v_0}^2_{L^2}).
\end{equation}
\end{Lemma}
\begin{proof}
In order to simplify our proof, we only consider the case when $\rho$ satisfies \eqref{B'} and $v$ satisfies \eqref{A'}, since other cases can be established in the same way and are much easier. We first deal with a special case for $\curle v=-n^\perp\cdot B\cdot v$ on the boundary.

We write $\eqref{equation3.2}_2$, using \eqref{equation3.1}, into the form
\begin{equation}\label{equa3.9}
\begin{aligned}
&\rho v_t+ \rho u\cdot\nabla v - \dive{[2\mu D(v)]}+ \nabla\pi_1 \\
&=   c_0\dive{\left(2\mu \nabla^2\rho^{-1}\right)}- c_0 \dive{\left(\rho v\otimes\nabla\rho^{-1}\right)}- c_0^2 \dive{\left(\rho \nabla\rho^{-1}\otimes\nabla\rho^{-1}\right)},
\end{aligned}
\end{equation}
Multiplying \eqref{equa3.9} by $v$ and integrating over $\Omega$, one has
\begin{equation}\label{4412}
\begin{aligned}
&\frac{d}{dt}\int\frac{1}{2}\rho\abs{v}^2-\int \dive{[2\mu D(v)]}\cdot v \\
&= \int c_0\dive{\left(2\mu \nabla^2\rho^{-1}\right)}\cdot v- \int c_0 \dive{\left(\rho v\otimes\nabla\rho^{-1}\right)}\cdot v\\
&\quad- \int c_0^2 \dive{\left(\rho \nabla\rho^{-1}\otimes\nabla\rho^{-1}\right)}\cdot v:=\sum_{i=1}^3I_i.
\end{aligned}
\end{equation}

Next, for the last term on the left-hand side of \eqref{4412}, we use again Lemma \ref{Lemma2.3} and \ref{lemma4.1} to get
\begin{equation}\label{equ412}
\begin{aligned}
-\int \dive{[2\mu D(v)]}\cdot v&=-\int 2\mu \Delta v\cdot v-\int 2\mu'\nabla\rho\cdot D(v)\cdot v \\
&=\int 2\mu |\curle v|^2+\int_\partial 2\mu v\cdot B\cdot v\\
&\quad+\int 2\mu'\nabla^\perp\rho\cdot v(\curle v)-\int 2\mu'\nabla\rho\cdot D(v)\cdot v,\\
&\geq\underline{\mu}\int\abs{\curle v}^2- \left(C_\varepsilon\norm{\nabla\rho}_{L^4}^4\norm{\sqrt\rho v}_{L^2}^2+\varepsilon\norm{\nabla v}_{L^2}^2\right),
\end{aligned}
\end{equation}
where $\underline{\mu}:=\min_{s\in[\alpha,\beta]}\mu(s)$ and the last inequality follows from the fact that $B$ is positive semi-definite. 

For $I_1$, we use the similar approach and write it in the component form, 
\begin{equation*}
\begin{aligned}
I_1&=\int_{\partial}- 2c_0\mu \partial_{ij}\rho^{-1}v_in_j+\int 2c_0\mu \partial_{ij}\rho^{-1}\partial_jv_i\\
&=\int_{\partial}- 4c_0\mu  \rho^{-3} \partial_{i}\rho\partial_{j}\rho v_in_j+
\int_{\partial} 2c_0\mu \rho^{-2}\partial_{ij}\rho v_in_j+\int 2c_0\mu \partial_{ij}\rho^{-1}\partial_jv_i :=\sum_{i=1}^3 J_i.
\end{aligned}
\end{equation*}
The estimate of $J_3$ can be simply derived by using Lemma \ref{Lemma2.3} and \ref{lemma4.1}, that is,
\begin{equation}\label{equa3.14}
\begin{aligned}
|J_3|&\leq C\abs{\int \mu \left(2\rho^{-3}\partial_{i}\rho\partial_{j}\rho- \rho^{-2}\partial_{ij}\rho\right)\partial_jv_i }\\
&\leq C_\varepsilon(\norm{\nabla\rho}_{L^4}^4+\norm{\Delta\rho}_{L^2}^2) + \varepsilon\norm{\nabla v}_{L^2}^2. 
\end{aligned}
\end{equation}
To the boundary parts $J_1$ and $J_2$, it suffices to estimate
\begin{align*}
J_1'&= \int_{\partial}\phi(\rho)\partial_{i}\rho\partial_{j}\rho v_in_j,\\
J_2'&=\int_{\partial}\phi(\rho)\partial_{ij}\rho v_in_j=-\int_{\partial}\phi(\rho) v_i\partial_in_j\partial_{j}\rho,
\end{align*}
where $\phi(\cdot)$ is a positive smooth function defined on $(0,\infty)$. Using Lemma \ref{Lemma2.3} and \ref{lemma4.1}, one has
\begin{equation}
\begin{aligned}
|J_1'|&= \abs{\int_{\partial}\phi(\rho)(n\cdot\nabla\rho)(v\cdot n^\perp)n^\perp\cdot\nabla\rho}\\
&=\abs{\int\nabla^\perp[\phi(\rho)(n\cdot\nabla\rho)]\cdot\nabla\rho(v\cdot n^\perp) + \int\phi(\rho)(n\cdot\nabla\rho)\nabla\rho\cdot\nabla^\perp(v\cdot n^\perp)}\\
&\leq C\left(\norm{\nabla\rho}_{L^4}^4 +\norm{\Delta\rho}_{L^2}^2\right) + C_{\varepsilon_1}\norm{\nabla\rho}_{L^4}^4\norm{\sqrt\rho v}_{L^2}^2+\varepsilon_1\norm{\nabla v}_{L^2}^2 
\end{aligned}
\end{equation}
and 
\begin{equation}\label{equa3.16}
\begin{aligned}
|J_2'|&= \abs{\int_{\partial}\phi(\rho)(v\cdot n^\perp)n^\perp\cdot\nabla n\cdot\nabla\rho}\\
&= \left|\int\nabla^\perp\phi(\rho)\cdot(\nabla n\cdot\nabla\rho)(v\cdot n^\perp)-\int_{\Omega}\phi(\rho)\nabla^\perp\cdot(\nabla n\cdot\nabla\rho)(v\cdot n^\perp)\,dx\right.\\
&\quad\left.-\int\phi(\rho)\nabla^\perp(v\cdot n^\perp)\cdot(\nabla n\cdot\nabla\rho)\right|\\
&\leq C\left(\norm{\nabla\rho}_{L^4}^4 +\norm{\Delta\rho}_{L^2}^2\right) + C_{\varepsilon_2}\norm{\nabla\rho}_{L^4}^4\norm{\sqrt\rho v}_{L^2}^2+\varepsilon_2\norm{\nabla v}_{L^2}^2. 
\end{aligned}
\end{equation}
Combining \eqref{equa3.14}--\eqref{equa3.16}, we deduce that 
\begin{equation}\label{equ3.12}
|I_1|\leq C\left(\norm{\nabla\rho}_{L^4}^4 +\norm{\Delta\rho}_{L^2}^2\right) + C_\varepsilon\norm{\nabla\rho}_{L^4}^4\norm{\sqrt\rho v}_{L^2}^2+\varepsilon\norm{\nabla v}_{L^2}^2.
\end{equation}
Similar computation can be applied for $I_2$ and $I_3$, that is,
 \begin{align}
&|I_2|\leq C_{\varepsilon_3}\norm{\nabla\rho}_{L^4}^4\norm{\sqrt\rho v}_{L^2}^2 +\varepsilon_3\norm{\nabla v}_{L^2}^2,\label{equ3.13}\\
&|I_3|\leq C\left(\norm{\nabla\rho}_{L^4}^4 +\norm{\Delta\rho}_{L^2}^2\right) + C_{\varepsilon_4}\norm{\nabla\rho}_{L^4}^4\norm{\sqrt\rho v}_{L^2}^2+\varepsilon_4\norm{\nabla v}_{L^2}^2.\label{equ3.14}
\end{align}
Therefore, we go back to the estimate of $v$, combining \eqref{4412}--\eqref{equ412} and \eqref{equ3.12}--\eqref{equ3.14} and then, using Lemma \ref{Lemma2.2} implies that 
\begin{equation}\label{eq3.28}
\frac{d}{dt}\norm{\sqrt\rho v}_{L^2}^2 + \nu \norm{\nabla v}_{L^2}^2\leq C\left(\norm{\nabla\rho}_{L^4}^4 +\norm{\Delta\rho}_{L^2}^2\right)\left( 1+ \norm{\sqrt\rho v}_{L^2}^2\right),
\end{equation}
for some constant $C$ depending on $\Omega$, $c_0$, $\alpha$ and $\beta$. In view of the condition \eqref{condition3.3}, we obtain the bound \eqref{equa3.8} via Gr$\mathrm{\ddot{o}}$nwall's inequality and Lemma \ref{Lemma2.3}. For the general case when $\curle v=-n^\perp\cdot B\cdot(v+c_0\nabla\rho^{-1})$ on $\partial\Omega\times (0,T)$, we can also obtain the desire bounds \eqref{equa3.8} by calculating the extra boundary term 
\begin{equation*}
\int_{\partial}2c_0\mu v\cdot B\cdot\nabla \rho^{-1}.
\end{equation*}
However, this term is nothing but a special case of $J'_2$ with $\nabla n$ replaced by $B$. Therefore, following the same computation of $J'_2$, we complete the proof for the general case.
\end{proof}

Now, we can turn back to prove Proposition \ref{theorem4.5}.

\begin{proof}[Proof of Proposition \ref{theorem4.5}]
Using Lemma \ref{Lemma3.4}, we obtain the bound of $v$, \eqref{equa3.8}, under the condition \eqref{condition3.3}.
Next, using Lemma \ref{Lemma4.3} and \eqref{equa3.8} leads to \eqref{equa3.5}, that is, if $\norm{\nabla\rho}_{L^2}(t)\leq \delta_1$ for all $t\in [0,T]$,
\begin{equation}
\sup_{t\in[0,T]}\norm{\nabla\rho}^2_{L^2} +\int_0^T\left(\norm{\nabla\rho}_{L^4}^4+\norm{\Delta\rho}_{L^2}^2\right)\,dt  \leq C\norm{\nabla\rho_0}^2_{L^2},
\end{equation}
where $C$ is a positive constant depending on $\Omega$, $c_0$, $\alpha$, $\beta$ and $\norm{v_0}_{L^2}$.

Hence, if we set the constant $\delta>0$ such that 
\begin{equation}
\delta = \min\left\{C^{-\frac{1}{2}}, \delta_1C^{-\frac{1}{2}}\right\},
\end{equation}
and let $\norm{\nabla\rho_0}_{L^2}\le\delta$, then it is easy to check that \eqref{3344} is established. Consequently, we complete the proof.
\end{proof}

\begin{Remark}\label{remark3.6}
It follows from the equation $\eqref{equation3.2}_1$ and $(\nabla\rho,v)\in L^4(0,T;L^4)$ that $\rho_t$ is bounded in $L^2(0,T;L^2)$. 
\end{Remark}

\subsection{Compactness Results}\label{section4}

Before establishing higher order estimates, we tend to prove the compactness results for $(\rho,v)$, which plays a crucial role in the proof of Theorem \ref{theorem1.2}, see Section \ref{proof6}. Concerning a sequence of weak solutions $(\rho^n,v^n)$ with $\pi_1$ replaced by $\pi_1^n$ satisfying the condition \eqref{A'} or \eqref{B'} and the initial conditions
\begin{equation}
\rho^n|_{t=0}=\rho^n_0,\,\,v^n|_{t=0}=v^n_0.
\end{equation}
We assume that $(\rho^n,v^n)$ satisfy, uniformly in $n\geq 1$, the a priori bounds that derived in the preceeding section and $\nabla\rho_0^n,v_0^n\sconverge \nabla\rho_0,v_0$ in $L^2$. Without loss of generality, extracting subsequences if necessary, we assume 
\begin{equation}\label{equat4.3}
\begin{cases}
\rho^n\wsconverge \rho\,\,&\mathrm{in}\,\,L^\infty(0,T;H^1),\\
\rho^n\wconverge \rho\,\,&\mathrm{in}\,\,L^2(0,T;H^2),\\
v^n\wsconverge v\,\,&\mathrm{in}\,\,L^\infty(0,T;L^2),\\
v^n\wconverge v\,\,&\mathrm{in}\,\, L^2(0,T;H^1).
\end{cases}
\end{equation}
We may now state our compactness results whose proof is followed by \cite{lions1}.
\begin{Lemma}\label{lemma5.1}
Under the hypothesis above, we have, for all $p\in[1,\infty)$,
\begin{align}
\rho^n\sconverge \rho\quad &\mathrm{in}\,\,C([0,T];L^p),\label{equation5.6}\\
v^n\sconverge v\quad&\mathrm{in}\,\,L^2(0,T;L^2).\label{equa4.7}
\end{align}
\end{Lemma}
\begin{proof}
Since we have $\eqref{equat4.3}_2$ and $\rho^n_t$ is bounded in $L^2(0,T;L^2)$ from Remark \ref{remark3.6}, \eqref{equation5.6} can be directly derived by using Lemma \ref{Lemma2.6}. To prove  \eqref{equa4.7}, observing that $\eqref{equation3.2}_2$ leads to
\begin{equation*}
\abs{\left<(\rho^nv^n)_t,\phi\right>}\leq C\norm{\phi}_{L^2(0,T;H^1)},
\end{equation*}
for all $\phi\in C_c^\infty(\Omega\times[0,T])$ such that $\dive\phi=0$, which implies that $(\rho^nv^n)_t$ is bounded in $L^2(0,T;V^{-1,2})$. On the other hand, since $\rho^nv^n$ is bounded in $L^2(0,T;H^1)$, it follows from Lemma \ref{Lemma2.6} that $\rho^nv^n$ is precompact in $L^2(0,T;L^2)$, that is,
\begin{equation}
\rho^nv^n\sconverge \overline{\rho v}\quad\text{in }L^2(0,T;L^2).
\end{equation}
Thanks to  $\eqref{equat4.3}_4$ and \eqref{equation5.6}, we have 
\begin{equation*}
\overline{\rho v}=\rho v,\quad v^n\sconverge v\quad\text{in }L^2(0,T;L^2),
\end{equation*}
which gives \eqref{equa4.7}.
\end{proof}

\subsection{Higher Order Estimates}\label{Section6}

In this subsection, we will show the a priori estimates for strong solutions of \eqref{equation3.2}. We still use the assumption at the begining of Section \ref{section3}. Furthermore, throughout this subsection, we always keep 
$$\norm{\nabla\rho_0}_{L^2}\leq \delta$$
 small enough so that Proposition \ref{theorem4.5} is valid. In a word, we have all the estimates of $(\rho,v)$ from Lemma \ref{lemma4.1}--\ref{Lemma3.4}.

For convenience, we set 
\begin{gather*}
F(t):=\norm{\nabla v}^2_{L^2}+\norm{\Delta\rho}^2_{L^2}+\norm{\rho_t}^2_{L^2},\\
G(t):=\norm{\Delta v}^2_{L^2}+\norm{v_t}^2_{L^2}+\norm{\nabla\Delta\rho}^2_{L^2}+\norm{\nabla\rho_t}^2_{L^2},\\
\cM_1(t):=\int_{\partial} \mu v\cdot B\cdot v,\\
\cM_2(t):=\int c_0\mu \nabla^\perp(v\cdot n^\perp)\cdot B\cdot \nabla \rho^{-1}.
\end{gather*}

We now state the proposition we are aimming for in this subsection.
\begin{Proposition}\label{theorem6.3}
Let $(\rho,v)$ satisfy the condition \eqref{A'} or \eqref{B'}. Then
\begin{align}
\sup_{t\in[0,T]}F(t)+\int_0^T\left(G(t)+\norm{\pi}_{H^1}^2\right)\,dt\leq C,\label{equation6.7}
\end{align}
where $C$ is a positive constant depending on $\Omega$, $\alpha$, $\beta$, $c_0$, $\norm{\rho_0}_{H^2}$ and $\norm{v_0}_{H^1}$.
\end{Proposition}

In order to prove Proposition \ref{theorem6.3}, we need several auxiliary lemmas. 
\begin{Lemma}\label{lemma6.4}
Under the assumptions at the begining of this section,
\begin{enumerate}
\item[(1)]
if $(\rho,v)$ satisfies the condition \eqref{A'}, then, for all $\varepsilon_1\in(0,1/2]$, 
\begin{equation}
\begin{aligned}
&\frac{d}{dt}\left(\norm{\Delta\rho}_{L^2}^2+\norm{\rho_t}_{L^2}^2\right)+\norm{\nabla\Delta\rho}_{L^2}^2 +\norm{\nabla\rho_t}_{L^2}^2\\
&\leq C_{\varepsilon_1} \cA_1(t)F(t)+\varepsilon_1\left(\norm{v}_{H^2}^2+\norm{v_t}_{L^2}^2\right);\label{equation67}
\end{aligned}
\end{equation}
\item[(2)]
if $(\rho,v)$ satisfies the condition \eqref{B'}, then
\begin{equation}\label{69}
\norm{\Delta\log\rho}^2_{L^2}\leq C\left(\norm{(\log\rho)_t}^2_{L^2}+\norm{\nabla v}^2_{L^2}\right)
\end{equation}
and for all $\varepsilon_2,\varepsilon_3\in (0,1]$, 
\begin{equation}\label{equation68}
\frac{d}{dt}\norm{(\log\rho)_t}_{L^2}^2 +\norm{\nabla(\log\rho)_t}_{L^2}^2
\leq C_{\varepsilon_2} \cA_2(t)\norm{(\log\rho)_t}_{L^2}^2+\varepsilon_2\norm{v_t}_{L^2}^2.
\end{equation}
\begin{equation}\label{equation69}
\begin{aligned}
\norm{\nabla\Delta\log\rho}_{L^2}^2&\leq C_{\varepsilon_3} \cA_3(t)\left(\norm{\Delta\log\rho}_{L^2}^2+\norm{\nabla v}_{L^2}^2\right)\\
&\quad+C_{\varepsilon_3}\norm{\nabla(\log\rho)_t}_{L^2}^2+\varepsilon_3\norm{\Delta v}_{L^2}^2.
\end{aligned}
\end{equation}
\end{enumerate}
Here, $C$, $C_{\varepsilon_1}-C_{\varepsilon_3}$ are positive constants depending on $\Omega$, $\alpha$, $\beta$, $c_0$ with $C_{\varepsilon_1}-C_{\varepsilon_3}$ extra depending on $\varepsilon_1$--$\varepsilon_3$ respectively, $\cA_1$--$\cA_3$ are all nonnegative integrable functions defined on $[0,\infty)$.
\end{Lemma}
\begin{proof}
We first consider the case when $(\rho,v)$ satisfies the condition \eqref{A'}. Taking $-(\nabla\Delta\rho)\nabla$ on the both sides of $\eqref{equation3.2}_1$ and integrating by parts, we have
\begin{equation}\label{eq4.1}
\begin{aligned}
\frac{d}{dt}\int \frac{1}{2}\abs{\Delta\rho}^2 + \int c_0\rho^{-1}\abs{\nabla\Delta\rho}^2&= \int\nabla\Delta\rho\cdot\nabla v\cdot\nabla\rho + \int v\cdot\nabla^2\rho\cdot\nabla\Delta\rho\\
&\quad-\int 2c_0\rho^{-3}\abs{\nabla\rho}^2\nabla\rho\cdot\nabla\Delta\rho+\int c_0 \rho^{-2}\nabla(\abs{\nabla\rho}^2)\cdot\nabla\Delta\rho\\
&\quad + \int c_0 \rho^{-2}\Delta\rho\nabla\rho\cdot\nabla\Delta\rho\\
&:=\sum_{i=1}^5 K_i.
\end{aligned}
\end{equation}

For $K_1$--$K_5$, we use Lemma \ref{Lemma2.3} and \ref{lemma4.1} to find that 
\begin{equation}\label{6611}
\begin{cases}
\abs{K_1}\leq C_{\varepsilon_1,\varepsilon_2}\norm{\nabla\rho}_{L^4}^4\norm{\nabla v}_{L^2}^2 + \varepsilon_1\norm{\nabla\Delta\rho}_{L^2}^2 +\varepsilon_2\norm{\Delta v}_{L^2}^2\\
\abs{K_2}\leq C_{\varepsilon_3}\norm{v}_{L^4}^4\norm{\Delta\rho}_{L^2}^2 + \varepsilon_3\norm{\nabla\Delta\rho}_{L^2}^2 \\
\abs{K_3}\leq C_{\varepsilon_4}\norm{\nabla\rho}_{L^6}^6+ \varepsilon_4\norm{\nabla\Delta\rho}_{L^2}^2\\
\quad\quad\,\leq C_{\varepsilon_4}\norm{\nabla\rho}_{L^4}^4\norm{\Delta\rho}_{L^2}^2+ \varepsilon_4\norm{\nabla\Delta\rho}_{L^2}^2\\
\abs{K_4}\leq C_{\varepsilon_5}\norm{\nabla\rho}_{L^4}^4\norm{\Delta\rho}_{L^2}^2+\varepsilon_5\norm{\nabla\Delta\rho}_{L^2}^2\\
\abs{K_5}\leq C_{\varepsilon_6}\norm{\nabla\rho}_{L^4}^4\norm{\Delta\rho}_{L^2}^2+\varepsilon_6\norm{\nabla\Delta\rho}_{L^2}^2.
\end{cases}
\end{equation}
Combining \eqref{eq4.1} and \eqref{6611}, we have, for some $\nu>0$,
\begin{equation}\label{equa5.11}
\frac{d}{dt}\norm{\Delta\rho}_{L^2}^2+\nu\norm{\nabla\Delta\rho}_{L^2}^2\leq C_\varepsilon\left(\norm{\nabla\rho}_{L^4}^4+\norm{v}_{L^4}^4\right)\norm{\Delta\rho}_{L^2}^2+C_\varepsilon\norm{\nabla\rho}_{L^4}^4\norm{\nabla v}_{L^2}^2+\varepsilon\norm{\Delta v}_{L^2}^2,
\end{equation}

Next, we estimate the bound of $\rho_t$. Differentiating $\eqref{equation3.2}_1$ with respect to $t$, one has
\begin{equation}\label{equa5.14}
\rho_{tt}- c_0\rho^{-1}\Delta\rho_t=-v_t\cdot\nabla\rho-v\cdot\nabla\rho_t+ 2c_0 \rho^{-3}\rho_t\abs{\nabla\rho}^2 - c_0 \rho^{-1}\left(\abs{\nabla\rho}^2\right)_t - c_0 \rho^{-2}\rho_t\Delta\rho.
\end{equation}
Multiplying $\rho_t$ on both sides of \eqref{equa5.14} and integrating over $\Omega$, we have
\begin{equation}\label{615}
\begin{aligned}
\frac{d}{dt}\int\frac{1}{2}\abs{\rho_t}^2 +\nu\int\abs{\nabla\rho_t}^2 &\leq\int \abs{v_t}\abs{\nabla\rho}\abs{\rho_t}+ C\int \abs{\rho_t}^2\abs{\nabla\rho}^2\\
&\quad+\int \abs{\rho_t}\abs{\nabla\rho_t}\abs{\nabla\rho}+C\int \abs{\rho_t}^2\abs{\Delta\rho}\\
&:= \sum_{i=1}^4L_i.
\end{aligned}
\end{equation}
Similarly, we use Lemma \ref{Lemma2.3} and \ref{lemma4.1} to obtain
\begin{equation}\label{616}
\begin{cases}
\abs{L_1}\leq C_{\varepsilon_1,\varepsilon_2}\norm{\nabla\rho}_{L^4}^4\norm{\rho_t}_{L^2}^2 +\varepsilon_1\norm{v_t}_{L^2}^2+\varepsilon_2\norm{\nabla\rho_t}_{L^2}^2,\\
\abs{L_2}\leq C_{\varepsilon_3}\norm{\nabla\rho}_{L^4}^4\norm{\rho_t}_{L^2}^2 +\varepsilon_3\norm{\nabla\rho_t}_{L^2}^2,\\
\abs{L_3}\leq C_{\varepsilon_4}\norm{\nabla\rho}_{L^4}^4\norm{\rho_t}_{L^2}^2 +\varepsilon_4\norm{\nabla\rho_t}_{L^2}^2,\\
\abs{L_4}\leq C_{\varepsilon_5}\norm{\Delta\rho}_{L^2}^2\norm{\rho_t}_{L^2}^2 +\varepsilon_5\norm{\nabla\rho_t}_{L^2}^2.
\end{cases}
\end{equation}
Thus, from \eqref{615} and \eqref{616}, one has, for some $\nu>0$,
\begin{equation}\label{equation616}
\frac{d}{dt}\norm{\rho_t}_{L^2}^2 +\nu\norm{\nabla\rho_t}_{L^2}^2
\leq C_\varepsilon\left(\norm{\nabla\rho}_{L^4}^4+\norm{\Delta\rho}_{L^2}^2\right)\norm{\rho_t}_{L^2}^2+\varepsilon\norm{v_t}_{L^2}^2,
\end{equation}
Combining \eqref{equa5.11} and \eqref{equation616} leads to the estimate \eqref{equation67}.

Next, we trun to the case when $(\rho,v)$ satisfies the condition \eqref{B'}. The main difficulty in this case is that, although we still have the estimate \eqref{equation616}, we can not use the energy method by integrating by parts to derive the bound of $\nabla \Delta\rho$. To overcome it, we estimate directly from $\eqref{equation3.2}_1$. More precisely, we first renormalize $\eqref{equation3.2}_1$ by writting 
\begin{equation}\label{5513}
(\log\rho)_t +v\cdot\nabla\log\rho - c_0\rho^{-1}\Delta\log\rho=0.
\end{equation}
Next, differentiating in $x$ on both sides of \eqref{5513}, one has
\begin{equation}\label{5514}
\nabla(\log\rho)_t +\nabla v\cdot\nabla\log\rho+v\cdot\nabla^2\log\rho +c_0\rho^{-2}\nabla\rho\Delta\log\rho- c_0\rho^{-1}\nabla\Delta\log\rho=0.
\end{equation}
Then, applying $L^2$-norm for $\nabla\Delta\log\rho$, then, using Lemma \ref{Lemma2.3} and \ref{lemma4.1} leads to
\begin{equation*}
\begin{aligned}
\norm{\nabla\Delta\log\rho}_{L^2}&\leq C\left(\norm{\nabla(\log\rho)_t}_{L^2} + \norm{\abs{\nabla v}\!\cdot\!\abs{\nabla\log\rho}}_{L^2} + \normf{\abs{v}\!\cdot\!|\nabla^2\log\rho|}_{L^2}+\normf{\abs{\nabla\rho}\!\cdot\!|\Delta\log\rho|}_{L^2}\right).
\end{aligned}
\end{equation*}
Thus, using Lemma \ref{Lemma2.3}, for all $\varepsilon_1,\varepsilon_2\in(0,1/2]$, there exists a constant $C_{\varepsilon_1,\varepsilon_2}$ depending on $\Omega$, $c_0$, $\alpha$, $\beta$, $\varepsilon_1$ and $\varepsilon_2$ such that 
\begin{equation*}
\begin{aligned}
\norm{\nabla\Delta\log\rho}_{L^2}^2&\leq C_{\varepsilon_1,\varepsilon_2}\left(\norm{\nabla(\log\rho)_t}_{L^2}^2  +\norm{v}_{L^4}^4\norm{\Delta\log\rho}_{L^2}^2+\norm{\nabla\rho}_{L^4}^4\norm{\Delta\log \rho}_{L^2}^2 \right)\\
&\quad+C_{\varepsilon_2} \norm{\nabla\rho}_{L^4}^4\norm{\nabla v}_{L^2}^2+\varepsilon_1\norm{\nabla\Delta\log\rho}_{L^2}^2 +\varepsilon_2\norm{\Delta v}_{L^2}^2,
\end{aligned}
\end{equation*}
Consequently,
\begin{equation}\label{eq4.9}
\begin{aligned}
\norm{\nabla\Delta\log\rho}_{L^2}^2&\leq C_\varepsilon\norm{\nabla(\log\rho)_t}_{L^2}^2+ C_\varepsilon\left(\norm{\nabla\rho}_{L^4}^4+\norm{v}_{L^4}^4\right)\norm{\Delta\log\rho}_{L^2}^2\\
&\quad+C_\varepsilon\norm{\nabla\rho}_{L^4}^4\norm{\nabla v}_{L^2}^2+\varepsilon\norm{\Delta v}_{L^2}^2,
\end{aligned}
\end{equation}
which gives \eqref{equation69}. 

It remains to show \eqref{69} and \eqref{equation68}. From \eqref{5513}, we use Lemma \ref{Lemma2.3} and \ref{lemma4.1} to obtain, 
\begin{equation*}
\begin{aligned}
\norm{\Delta\log\rho}^2_{L^2}&\leq C\norm{(\log\rho)_t}^2_{L^2} + C\norm{\nabla\log\rho}^2_{L^4}\norm{v}^2_{L^4}\\
&\leq C\norm{(\log\rho)_t}^2_{L^2} + C_\varepsilon\norm{\nabla\rho}^2_{L^2}\norm{v}^2_{L^2}\norm{\nabla v}^2_{L^2}+\varepsilon\norm{\Delta\log\rho}^2_{L^2},
\end{aligned}
\end{equation*}
that is,
\begin{equation*}
\norm{\Delta\log\rho}^2_{L^2}\leq C\left(\norm{(\log\rho)_t}^2_{L^2}+\norm{\nabla v}^2_{L^2}\right).
\end{equation*}

For \eqref{equation68}, we can follow the proof from \eqref{equa5.14} to \eqref{equation616} by applying $(\log\rho)_t\partial_t$ on both sided \eqref{5513} and integrating over $\Omega$, that is,
\begin{equation}\label{5516}
\begin{aligned}
&\frac{d}{dt}\int\frac{1}{2}|(\log\rho)_t|^2+\int c_0 \rho^{-1}|\nabla(\log\rho)_t|^2\\
&=-\int c_0\rho^{-1}\nabla(\log\rho)_t\cdot\nabla\log\rho(\log\rho)_t-\int v_t\cdot\nabla\log\rho(\log\rho)_t\\
&\quad+\int c_0 \rho^{-1}|(\log\rho)_t|^2|\nabla\log\rho|^2\\
&:=\sum_{i=1}^3P_i,
\end{aligned}
\end{equation}
where, applying Lemma \ref{Lemma2.3} and \ref{lemma4.1},
\begin{equation}\label{5517}
\begin{cases}
|P_1|&\!\!\!\!\leq C\norm{\nabla\log\rho}_{L^4}\norm{(\log\rho)_t}_{L^4}\norm{\nabla(\log\rho)_t}_{L^2}\\
&\!\!\!\!\leq C_{\varepsilon_1}\norm{\nabla\rho}_{L^4}^4\norm{(\log\rho)_t}_{L^2}^2+\varepsilon_1\norm{\nabla(\log\rho)_t}_{L^2}^2\\
|P_2|&\!\!\!\!\leq \norm{\nabla\log\rho}_{L^4}\norm{(\log\rho)_t}_{L^4}\norm{v_t}_{L^2}\\
&\!\!\!\!\leq C_{\varepsilon_2}\norm{\nabla\rho}_{L^4}^4\norm{(\log\rho)_t}_{L^2}^2+\varepsilon_2\norm{v_t}_{L^2}^2\\
|P_3|&\!\!\!\!\leq \norm{\nabla\log\rho}^2_{L^4}\norm{(\log\rho)_t}^2_{L^4}\\
&\!\!\!\!\leq C_{\varepsilon_3}\norm{\nabla\rho}_{L^4}^4\norm{(\log\rho)_t}_{L^2}^2+\varepsilon_3\norm{\nabla(\log\rho)_t}_{L^2}^2.
\end{cases}
\end{equation}
Combining \eqref{5516} and \eqref{5517} leads to
\begin{equation*}
\frac{d}{dt}\norm{(\log\rho)_t}_{L^2}^2+\nu\norm{\nabla(\log\rho)_t}_{L^2}^2\leq C_\varepsilon\norm{\nabla\rho}_{L^4}^4\norm{(\log\rho)_t}_{L^2}^2+\varepsilon\norm{v_t}_{L^2}^2.
\end{equation*}
This completes the proof of the lemma.
\end{proof}
\begin{Remark}\label{remark53}
One may find that we estimate $\log\rho$ instead of $\rho$ in the proof of $(\rho,v)$ satisfying \eqref{B'}. This is based on the observation that $\eqref{equation3.2}_1$ has the dissipative term $\Delta\log\rho$, that is,
\begin{equation*}
\rho_t +v\cdot\nabla\rho - c_0\Delta\log\rho=0.
\end{equation*}
Thus, such conversion can avoid the occurrence of the nonlinear term $|\nabla\rho|^2$, otherwise, if we estimate $\Delta\rho$ in the proof of \eqref{69}, we need additional smallness assumption on $\nabla\rho_0$ to handle $\norm{|\nabla\rho|^2}_{L^2}$, which is not what we expect (this point will also be seen in Section \ref{section9}).
\end{Remark}

The next lemma shows that $v$ can be bounded by the norm of $\rho$ provided $\norm{\nabla\rho_0}_{L^2}$ is small.
\begin{Lemma}\label{lemma6.5}
Under the assumptions at the begining of this section,
\begin{enumerate}
\item[(1)]
if $(\rho,v)$ satisfy the condition \eqref{A'}, then for all $\varepsilon\in(0,1/2]$,
\begin{equation}\label{equation620}
\begin{aligned}
&\frac{d}{dt}\left(\cM_1(t)+\normf{\sqrt{\mu}\curle v}_{L^2}^2\right)+\norm{v_t}_{L^2}^2+\frac{d}{dt}\cM_2(t)\\
&\leq C_\varepsilon \cA_4(t)F(t)+\varepsilon\left(\norm{\nabla\rho_t}^2_{L^2}+\norm{\nabla\Delta \rho}_{L^2}^2\right)+\cA_5(t),
\end{aligned}
\end{equation}
and
\begin{equation}\label{equation621}
\begin{aligned}
\norm{v}^2_{H^2}+\norm{\pi}_{H^1}^2 \leq \cA_6(t)F(t)+C\left(\norm{\nabla\Delta\rho}^2_{L^2}+\norm{v_t}^2_{L^2}+\norm{\nabla\rho_t}_{L^2}^2\right)+\cA_7(t),
\end{aligned}
\end{equation}
where $C$ and $C_\varepsilon$ are positive constants depending on $\Omega$, $c_0$, $\alpha$, $\beta$ with $C_\varepsilon$ extra depending on $\varepsilon$; $\cA_4$--$\cA_7$ are nonnegative integrable functions defined on $[0,\infty)$;
\item[(2)]
 if $(\rho,v)$ satisfies the conditin \eqref{B'}, one still has the estimates \eqref{equation620} and \eqref{equation621} with $\cM_1(t)=\cM_2(t)=0$ in \eqref{equation620}. More precisely,
 \begin{equation}\label{5520}
\begin{aligned}
&\frac{d}{dt}\norm{\sqrt{\mu}|D(v)|}_{L^2}^2+\norm{v_t}_{L^2}^2\\
&\leq C_\varepsilon\cA_8(t)\left(\norm{\nabla v}_{L^2}^2+\norm{\Delta\log\rho}_{L^2}^2\right)+\varepsilon\norm{\nabla\Delta\log\rho}_{L^2}^2,
\end{aligned}
\end{equation}
and 
\begin{equation}\label{5521}
\begin{aligned}
\norm{v}^2_{H^2}+\norm{\pi}_{H^1}^2&\leq C_\varepsilon\cA_9(t)\left(\norm{\nabla v}^2_{L^2}+\norm{\Delta\log\rho}^2_{L^2}\right) + \varepsilon\norm{\nabla\Delta\log\rho}^2_{L^2}\\
&\quad +C\left(\norm{v_t}^2_{L^2}+\norm{\nabla
(\log\rho)_t}^2_{L^2}\right).
\end{aligned}
\end{equation}
where $C$ and $C_\varepsilon$ as in $(1)$; $\cA_8$ and $\cA_9$ are nonnegative integrable functions defined on $[0,\infty)$.
\end{enumerate}
\end{Lemma}

\begin{proof}
Rewrite $\eqref{equation3.2}_2$ as 
\begin{equation}\label{eq4.18}
\begin{aligned}
&\rho v_t - \dive{[2\mu D(v)]} 
 + \nabla\pi_1\\
 &= - \rho u\cdot\nabla v+ c_0\dive{\left(2\mu\nabla^2\rho^{-1}\right)}  -c_0 \dive{\left(\rho v\otimes\nabla\rho^{-1}\right)}\\
 &\quad- c_0^2 \dive{\left(\rho \nabla\rho^{-1}\otimes\nabla\rho^{-1}\right)}.
\end{aligned}
\end{equation}
We first come to the case when $(\rho,v)$ satisfies the condition \eqref{A'} and consider the special case when $\curle v=-n^\perp \cdot B\cdot v$ on $\partial\Omega\times(0,T)$. Multiplying $v_t$ on both sides of \eqref{eq4.18} and integrating over $\Omega$, one gets
\begin{equation}\label{6623}
\begin{aligned}
&\int \rho\abs{v_t}^2 - \int \dive[2\mu D(v)]\cdot v_t\\
 &= - \int\rho u\cdot\nabla v\cdot v_t+ \int c_0\dive{\left(2\mu \nabla^2\rho^{-1}\right)}\cdot v_t -\int c_0 \dive{\left(\rho v\otimes\nabla\rho^{-1}\right)}\cdot v_t\\
 &\quad-\int c_0^2 \dive{\left(\rho \nabla\rho^{-1}\otimes\nabla\rho^{-1}\right)}\cdot v_t:=\sum	_{i=1}^4M_i.
\end{aligned}
\end{equation}

For the second term on the left-hand side of \eqref{6623}, we have
\begin{align*}
- \int \dive[2\mu D(v)]\cdot v_t&= - \int 2\mu \Delta v\cdot v_t- \int 2\mu' \nabla\rho\cdot D(v)\cdot v_t:=\sum_{i=1}^2Q_i.
\end{align*}
First, to estimate $Q_1$, we have
\begin{align*}
Q_1&= - \int 2\mu \nabla^\perp(\curle v)\cdot v_t\\
&=-\int_{\partial} 2\mu \curle v(v_t\cdot n^\perp)+\int \mu \frac{d}{dt}|\curle v|^2+\int \mu' \nabla^\perp \rho\cdot v_t(\curle v)\\
&=\frac{d}{dt}\left(\cM_1(t)+\normf{\sqrt{\mu} \curle v}_{L^2}^2\right)\\
&\quad-\int_{\partial} \mu_t(v\cdot B\cdot v)-\int \mu_t|\curle v|^2+\int \mu' \nabla^\perp \rho\cdot v_t(\curle v),
\end{align*}
where, for the last three terms, we use Lemma \ref{Lemma2.3} and \ref{lemma4.1} to obtain
\begin{align*}
\abs{\int_{\partial}\mu_t(v\cdot B\cdot v)}&=\abs{\int_{\partial}\mu_t(v\cdot n^\perp)n^\perp\cdot B\cdot v}\notag\\
&=\abs{\int \mu_t\nabla^\perp(v\cdot n^\perp)\cdot B\cdot v+\int v\cdot n^\perp \nabla^\perp\cdot[\mu_tB\cdot v]}\\
&\leq C_{\varepsilon_1}\left(\norm{\nabla\rho}_{L^4}^4+\norm{v}_{L^4}^4\right)\norm{\rho_t}_{L^2}^2+C_{\varepsilon_1}\norm{v}_{L^4}^4+ \varepsilon_1\left(\norm{\nabla\rho_t}^2_{L^2}+\norm{v}_{H^2}^2\right),\\
\abs{\int \mu_t\abs{\curle v}^2}&\leq C_{\varepsilon_2}\norm{\rho_t}_{L^2}^2\norm{\nabla v}_{L^2}^2+\varepsilon_2\norm{v}_{H^2}^2,\notag\\
\abs{\int \mu'\nabla^\perp \rho\cdot v_t(\curle v)}&\leq C_{\varepsilon_3}\norm{\nabla\rho}_{L^4}^4\norm{\nabla v}_{L^2}^2+\varepsilon_3\left(\norm{v_t}_{L^2}^2+\norm{v}_{H^2}^2\right),\notag
\end{align*}
which gives
\begin{equation}\label{624}
\begin{aligned}
Q_1&\geq\frac{d}{dt}\left(\cM_1(t)+\normf{\sqrt{\mu}\curle v}_{L^2}^2\right)\\
&\quad-C_{\varepsilon_4}\left(\norm{\rho_t}_{L^2}^2+\norm{v}_{L^4}^4+\norm{\nabla\rho}_{L^4}^4\right)\left(\norm{\nabla v}_{L^2}^2+\norm{\rho_t}_{L^2}^2\right)+C_{\varepsilon_4}\norm{v}_{L^4}^4\\
&\quad+\varepsilon_4\left(\norm{\nabla\rho_t}^2_{L^2}+\norm{v_t}_{L^2}^2+\norm{v}_{H^2}^2\right).
\end{aligned}
\end{equation}
On the other hand,
\begin{align}
|Q_2|\leq C_{\varepsilon_5}\norm{\nabla\rho}_{L^4}^4\norm{\nabla v}_{L^2}^2+\varepsilon_5\left(\norm{v_t}_{L^2}^2+\norm{v}_{H^2}^2\right).\label{625}
\end{align}
Combining \eqref{624} and \eqref{625} leads to
\begin{equation}\label{eq626}
\begin{aligned}
- \int \dive[2\mu D(v)]\cdot v_t&\geq \frac{d}{dt}\left(\cM_1(t)+\normf{\sqrt{\mu}\curle v}_{L^2}^2\right)\\
&\quad-C_{\varepsilon_6}\left(\norm{\rho_t}_{L^2}^2+\norm{v}_{L^4}^4+\norm{\nabla\rho}_{L^4}^4\right)\left(\norm{\nabla v}_{L^2}^2+\norm{\rho_t}_{L^2}^2\right)\\
&\quad+C_{\varepsilon_6}\norm{v}_{L^4}^4+\varepsilon_6\left(\norm{\nabla\rho_t}^2_{L^2}+\norm{v_t}_{L^2}^2+\norm{v}_{H^2}^2\right).
\end{aligned}
\end{equation}

Next, using Lemma \ref{Lemma2.3} and \ref{lemma4.1} again, we estimate $M_1-M_4$, that is,
\begin{align}
\abs{M_1}\leq C_{\varepsilon_7}\left(\norm{v}_{L^4}^4+\norm{\nabla\rho}_{L^4}^4\right)\norm{\nabla v}_{L^2}^2 +\varepsilon_7\left(\norm{v_t}_{L^2}^2 +\norm{v}_{H^2}^2\right),
\end{align}
\begin{equation}
\begin{aligned}
\abs{M_2}&=2c_0\abs{\int \mu'\partial_j\rho\partial_{ij}\rho^{-1}(v_t)_i+ \int \mu(\rho)\partial_{ijj}\rho^{-1}(v_t)_i}\\
&=2c_0 \abs{\int \mu'\partial_j\rho\partial_{ij}\rho^{-1}(v_t)_i-\int \mu'\partial_i\rho\partial_{jj}\rho^{-1}(v_t)_i}\\
&\leq C\int (\abs{\nabla\rho}^3+\abs{\nabla\rho}|\nabla^2\rho|)\abs{v_t}\\
&\leq C_{\varepsilon_8}\norm{\nabla\rho}_{L^4}^4\norm{\Delta \rho}_{L^2}^2 +\varepsilon_8\left(\norm{v_t}_{L^2}^2+\norm{\nabla\Delta \rho}_{L^2}^2\right),
\end{aligned}
\end{equation}
\begin{equation}\label{equa5.27}
\begin{aligned}
\abs{M_3+M_4}&=\abs{c_0 \int\partial_j\left( u_j\partial_i\log\rho\right)(v_t)_i}\\
&=\abs{c_0 \int\partial_j\left(\log\rho \partial_iv_j\right)(v_t)_i+c_0^2 \int\partial_j\left(\log\rho \partial_{ij}\rho^{-1}\right)(v_t)_i}\\
&\leq C\int \left(\abs{\nabla v}\abs{\nabla\rho}+\abs{\nabla\rho}^3+\abs{\nabla\rho}\abs{\nabla^2\rho}\right)\abs{v_t}\\
&\leq C_{\varepsilon_9}\norm{\nabla\rho}_{L^4}^4\left(\norm{\nabla v}_{L^2}^2+\norm{\Delta \rho}_{L^2}^2\right) +\varepsilon_9\left(\norm{v_t}_{L^2}^2 +\norm{v}_{H^2}^2+\norm{\nabla\Delta \rho}_{L^2}^2\right).
\end{aligned}
\end{equation}
Combining \eqref{eq626}--\eqref{equa5.27}, we have, for all $\varepsilon\in (0,1/2]$, there exists a positive constant $C_\varepsilon$ depending on $\Omega$, $c_0$, $\alpha$, $\beta$ and $\varepsilon$ such that
\begin{equation}\label{eq4.27}
\begin{aligned}
&\frac{d}{dt}\left(\cM_1(t)+\normf{\sqrt{\mu}\curle v}_{L^2}^2\right)+\norm{v_t}_{L^2}^2\\
&\leq C_\varepsilon\left(\norm{\rho_t}_{L^2}^2+\norm{v}_{L^4}^4+\norm{\nabla\rho}_{L^4}^4\right)F(t)+C_\varepsilon\norm{v}_{L^4}^4\\
&\quad+\varepsilon\left(\norm{\nabla\rho_t}^2_{L^2}+\norm{v}_{H^2}^2+\norm{\nabla\Delta \rho}_{L^2}^2\right).
\end{aligned}
\end{equation}

For general boundary case, that is, $\curle v=-n^\perp\cdot B\cdot(v+c_0\nabla\rho^{-1})$, it suffices to calculate the following extra term
\begin{equation}\label{631}
\begin{aligned}
\int_{\partial} \phi(\rho)v_t\cdot B\cdot \nabla \rho&=\int_{\partial} \phi(\rho)(v_t\cdot n^\perp)n^\perp\cdot B\cdot \nabla \rho\\
&=\int(v_t\cdot n^\perp)\nabla^\perp\phi(\rho)\cdot B\cdot \nabla \rho+\int \phi(\rho)(v_t\cdot n^\perp)\nabla^\perp\cdot (B\cdot \nabla \rho)\\
&\quad+\int\phi(\rho)\nabla^\perp(v_t\cdot n^\perp)\cdot B\cdot \nabla \rho:=\sum_{i=1}^3G_i,
\end{aligned}
\end{equation}
where $\phi(s):=c_0\mu(s)s^{-2}$. For the first two terms of \eqref{631}, using Lemma \ref{Lemma2.3} and \ref{lemma4.1}, we have
\begin{equation}
\abs{G_1+G_2}\leq C_{\varepsilon_1}\left(\norm{\nabla\rho}_{L^4}^4+\norm{\Delta\rho}_{L^2}^2\right)+\varepsilon_1\norm{v_t}_{L^2}^2.
\end{equation}
For $G_3$, since we can not handle the term $\nabla^\perp(v_t\cdot n^\perp)$, it shall be converted into
\begin{equation}\label{633}
\begin{aligned}
G_3=&\frac{d}{dt}\cM_2(t)-\int \phi(\rho)_t\nabla^\perp(v\cdot n^\perp)\cdot B\cdot \nabla \rho-\int \phi(\rho)\nabla^\perp(v\cdot n^\perp)\cdot B\cdot \nabla \rho_t\notag\\
\geq& \frac{d}{dt}\cM_2(t)-\left[C_{\varepsilon_2}\left(\norm{\rho_t}_{L^2}^2+\norm{\nabla\rho}_{L^4}^4\norm{\nabla v}_{L^2}^2+\norm{\nabla v}_{L^2}^2\right)+\varepsilon_2\left(\norm{v}_{H^2}^2+\norm{\nabla\rho_t}_{L^2}^2\right)\right].
\end{aligned}
\end{equation}
Thus, combining \eqref{633}--\eqref{eq4.27}, we deduce the estimate which is simliar with \eqref{eq4.27}, that is 
\begin{equation}\label{633}
\begin{aligned}
&\frac{d}{dt}\left(\cM_1(t)+\normf{\sqrt{\mu}\curle v}_{L^2}^2\right)+\norm{v_t}_{L^2}^2+\frac{d}{dt}\cM_2(t)\\
&\leq C_\varepsilon \cA_4(t)F(t)+\varepsilon\left(\norm{\nabla\rho_t}^2_{L^2}+\norm{v}_{H^2}^2+\norm{\nabla\Delta \rho}_{L^2}^2\right)+\cA_5(t),
\end{aligned}
\end{equation}
for some integrable functions $\cA_4$ and $\cA_5$ defined on $[0,\infty)$.

We still need estimate $\norm{v}_{H^2}$. Let us rewrite \eqref{eq4.18} as 
\begin{equation}\label{5533} 
- \mu \Delta v+ \nabla\pi= F,
\end{equation}
where 
\begin{equation*}
\begin{aligned}
 F&:=-\rho v_t+\nabla(\log\rho)_t- \rho u\cdot\nabla v+2\mu'\nabla\rho\cdot D(v)+2c_0\dive\left(\mu\nabla^2\rho^{-1}\right) \\
 &-c_0 \dive{\left(\rho v\otimes\nabla\rho^{-1}\right)} - c_0^2 \dive{\left(\rho \nabla\rho^{-1}\otimes\nabla\rho^{-1}\right)}
\end{aligned}
\end{equation*}
Since $\mu(\rho(x,t))$ is bounded contiuous on $\overline Q_T$ from Lemma \ref{lemma61}, it follows from Lemma \ref{Lemma5.2} with $\varphi= -n^\perp\cdot B\cdot (v+c_0\nabla\rho^{-1})$ that
\begin{equation}\label{eq4.31}
\norm{v}_{H^2} + \norm{\pi}_{H^1}\leq C\left(\norm{F}_{L^2}+\norm{\varphi}_{H^1}\right),
\end{equation}
where
\begin{equation*}
\begin{aligned}
\norm{\varphi}_{H^1}&\leq C\left(\norm{\nabla v}_{L^2}+\norm{\Delta\rho}_{L^2}+\norm{\nabla\rho}_{L^4}^2\right),\\
\norm{F}_{L^2}&\leq C\big(\norm{v_t}_{L^2}+\norm{\nabla
\rho_t}_{L^2}+\norm{\nabla\rho}^2_{L^4}\norm{\rho_t}_{L^2}+\norm{|v|\!\cdot\!|\nabla v|}_{L^2}+\normf{|\nabla\rho|\!\cdot\!|\nabla v|}_{L^2}+ \normf{\nabla\rho}_{L^6}^3 \\
&\quad+\normf{|\nabla\rho|\!\cdot\!|\nabla^2 \rho|}_{L^2}+\normf{|v|\!\cdot\!|\nabla^2 \rho|}_{L^2}+\normf{|v|\!\cdot\!|\nabla\rho|^2}_{L^2}\big)\\
&\leq C\left(\norm{v_t}_{L^2}+\norm{\nabla
\rho_t}_{L^2}\right) + C_\varepsilon\left(\norm{v}_{L^4}^2+\norm{\nabla\rho}_{L^4}^2\right)\left(\norm{\nabla v}_{L^2}+\norm{\Delta\rho}_{L^2}+\norm{\rho_t}_{L^2}\right) \\
&\quad+ \varepsilon\left(\norm{v}_{H^2}+\norm{\nabla\Delta\rho}_{L^2}\right).
\end{aligned}
\end{equation*}
Hence, using the Poincar\'e's inequality, we deduce from \eqref{eq4.31} that 
\begin{equation}\label{eq635}
\begin{aligned}
\norm{v}^2_{H^2}+\norm{\pi}_{H^1}^2& \leq C\left(\norm{v}_{L^4}^4+\norm{\nabla\rho}_{L^4}^4\right)\left(\norm{\nabla v}^2_{L^2}+\norm{\Delta\rho}^2_{L^2}+\norm{\rho_t}_{L^2}^2\right)\\
&\quad+C\left(\norm{v_t}^2_{L^2}+\norm{\nabla\rho_t}^2_{L^2}\right) +C\left(\norm{\nabla\Delta\rho}^2_{L^2}+\norm{\nabla v}_{L^2}^2\right),
\end{aligned}
\end{equation}
which gives \eqref{equation621}. Finally, substituting \eqref{equation621} into \eqref{633}, we obtain \eqref{equation620} and complete the proof of the lemma.

For $(\rho,v)$ satisfying condition \eqref{B'}, we first convert \eqref{eq4.18} into
\begin{equation}
\begin{aligned}
&\rho v_t - \dive{[2\mu D(v)]} 
 + \nabla\pi_1\\
 &= - \rho v\cdot\nabla v +c_0\nabla\log\rho\cdot\nabla v+ c_0\dive{\left[2\mu\rho^{-1} \left(\nabla^2\log\rho-\nabla\log\rho\otimes\nabla\log\rho\right)\right]} \\
&\quad +c_0 \dive{\left(v\otimes\nabla\log\rho\right)}
- c_0^2 \dive{\left(\rho^{-1} \nabla\log\rho\otimes\nabla\log\rho\right)}.
\end{aligned}
\end{equation}
Then, following the calculations from \eqref{eq4.18} to \eqref{eq4.27} and from \eqref{5533} to \eqref{eq635}, we can derive the similar estimates (even much easier, since $v_t$ is vanished on the boundary), that is, 
\begin{equation}\label{5537}
\begin{aligned}
\frac{d}{dt}\norm{\mu D(v)}_{L^2}^2+\nu\norm{v_t}_{L^2}^2&\leq C_\varepsilon\left(\norm{\nabla\rho}_{L^4}^4+\norm{v}_{L^4}^4+\norm{\rho_t}_{L^2}^2\right)\left(\norm{\nabla v}_{L^2}^2+\norm{\Delta\log\rho}_{L^2}^2\right)\\
&\quad+\varepsilon\left(\norm{\nabla\Delta\log\rho}_{L^2}^2+\norm{v}_{H^2}^2\right),
\end{aligned}
\end{equation}
\begin{equation}\label{5538}
\begin{aligned}
\norm{F}_{L^2}&\leq C\left(\norm{v_t}_{L^2}+\norm{\nabla
\log\rho_t}_{L^2}\right) + C_\varepsilon\left(\norm{v}_{L^4}^2+\norm{\nabla\rho}_{L^4}^2\right)\left(\norm{\nabla v}_{L^2}+\norm{\Delta\log\rho}_{L^2}\right) \\
&\quad+ \varepsilon\left(\norm{v}_{H^2}+\norm{\nabla\Delta\log\rho}_{L^2}\right)
\end{aligned}
\end{equation}
and, using Lemma \ref{Lemma5.2} with $\Phi=0$, together with \eqref{5538},
\begin{equation}\label{5539}
\begin{aligned}
\norm{v}^2_{H^2}+\norm{\pi}_{H^1}^2&\leq C\norm{F}_{L^2}^2\\
&\leq C_\varepsilon\left(\norm{v}_{L^4}^4+\norm{\nabla\rho}_{L^4}^4\right)\left(\norm{\nabla v}^2_{L^2}+\norm{\Delta\log\rho}^2_{L^2}\right) + \varepsilon\norm{\nabla\Delta\log\rho}^2_{L^2}\\
&\quad +C\left(\norm{v_t}^2_{L^2}+\norm{\nabla
(\log\rho)_t}^2_{L^2}\right).
\end{aligned}
\end{equation}
Thus, we complete the proof by plugging \eqref{5539} into \eqref{5537}.
\end{proof}

Now, combining Lemma \ref{lemma6.4}--\ref{lemma6.5}, we can complete the proof of Proposition \ref{theorem6.3}.
\begin{proof}[Proof of Proposition \ref{theorem6.3}]
We first prove the case when $(\rho,v)$ satisfies the condition \eqref{A'}. Using Lemma \ref{Lemma2.2}, \eqref{equation67} in Lemma \ref{lemma6.4} and \eqref{equation620}, \eqref{equation621} in Lemma \ref{lemma6.5} leads to
\begin{equation}\label{638}
\begin{aligned}
&\frac{d}{dt}\left(\normf{\sqrt{\mu}\curle v}_{L^2}^2+\norm{\Delta\rho}_{L^2}^2+\norm{\rho_t}_{L^2}^2+\cM_1(t)\right)+\norm{v_t}_{L^2}^2+\norm{\nabla\Delta\rho}_{L^2}^2 +\norm{\nabla\rho_t}_{L^2}^2\\
&\leq -\frac{d}{dt}\cM_2(t)+\tilde\cA_1(t)F(t)+\tilde\cA_2(t),\\
&\leq -\frac{d}{dt}\cM_2(t)+\tilde\cA_1(t)\left(\normf{\sqrt{\mu}\curle v}_{L^2}^2+\norm{\Delta\rho}_{L^2}^2+\norm{\rho_t}_{L^2}^2+\cM_1(t)\right)+\tilde\cA_2(t),
\end{aligned}
\end{equation}
where $\tilde\cA_1$ and $\tilde\cA_2$ are positive integrable functions defined on $[0,\infty)$. Using Gr$\mathrm{\ddot{o}}$nwall's inequality and Lemma \ref{Lemma2.2} once again, we deduce the bound
\begin{equation}\label{639}
\begin{aligned}
&\sup_{t\in[0,T]}F(t)+\int_0^T\left(\norm{\nabla\Delta\rho}_{L^2}^2 +\norm{\nabla\rho_t}_{L^2}^2+\norm{v_t}_{L^2}^2\right)\,dt\\
&\leq C(\norm{v_0}^2_{H^1}+\norm{\rho_0}^2_{H^2}+\norm{v_0}_{H^1}^2\norm{\rho_0}_{H^2}^2+1)\leq C,
\end{aligned}
\end{equation}
where we have used, denote by $\rho_{t,0}=\rho_t(x,0)$,
$$\norm{\rho_{t,0}}_{L^2}\leq \norm{\rho_0}_{H^2}+\norm{v_0}_{H^1}\norm{\rho_0}_{H^2},$$  
$\cM_1\geq 0$ for $B$ positively semi-definited and 
\begin{equation*}
\abs{e^{\int_0^T h(t)\,dt}\int_0^T \frac{d}{dt}\cM_2(t)e^{-\int_0^t h(s)\,ds}\,dt}\leq \varepsilon\sup_{t\in[0,T]}\norm{\nabla v}_{L^2}^2+C_\varepsilon\sup_{t\in[0,T]}\norm{\nabla \rho}_{L^2}^2,
\end{equation*}
where $h(t)$ is an integrable function on $[0,\infty)$. 

Next, integrating \eqref{equation621} over $[0,T]$ and using the bound \eqref{639} gives  
\begin{equation}
\int_0^T\left(\norm{v}_{H^2}^2+\norm{\pi}_{H^1}^2\right)\,dt \leq C, 
\end{equation}
which shows \eqref{equation6.7}.

To prove the case when $(\rho,v)$ satisfies the condition \eqref{B'}, we use \eqref{equation68} in Lemma \ref{lemma6.4} and \eqref{5520} in Lemma \ref{lemma6.5}. It follows from Lemma \ref{lemma4.1} and
$$2\int |D(v)|^2=\int |\nabla v|^2$$
that
\begin{equation}\label{6642}
\begin{aligned}
&\frac{d}{dt}\left(\norm{\nabla v}_{L^2}^2+\norm{(\log\rho)_t}_{L^2}^2\right)+\norm{v_t}_{L^2}^2
 +\norm{\nabla(\log\rho)_t}_{L^2}^2\\
 &\leq \tilde\cA_3(t)\left(\norm{\nabla v}_{L^2}^2+\norm{(\log\rho)_t}_{L^2}^2\right)+\varepsilon\norm{\nabla\Delta\log\rho}_{L^2}^2,
\end{aligned}
\end{equation}
for some nonegative integrable functions $\tilde\cA_3$. Using the Gr$\mathrm{\ddot{o}}$nwall's inequality, one has the bound
\begin{align}
&\sup_{t\in[0,T]}\norm{\nabla v}^2_{L^2}+\sup_{t\in[0,T]}\norm{(\log\rho)_t}^2_{L^2} +\int_0^T\left(\norm{v_t}_{L^2}^2
+\norm{\nabla(\log\rho)_t}_{L^2}^2\right)\,dt\notag\\
 &\leq C(\norm{v_0}^2_{H^1}+\norm{\rho_0}^2_{H^2}+\norm{v_0}^2_{H^1}\norm{\rho_0}^2_{H^2}+1)\leq C,\label{641}
\end{align}
With the aid of the estimate  \eqref{69}, \eqref{equation69}, \eqref{5521} and \eqref{641}, one has 
\begin{align}
\sup_{t\in[0,T]}\norm{\Delta\log\rho}^2_{L^2} +\int_0^T\left(\norm{\nabla\Delta\log\rho}_{L^2}^2 +\norm{v}_{H^2}^2+\norm{\pi}_{H^1}^2\right)\,dt\leq C.
\end{align}
At last, noticing that 
$$\Delta\rho=\rho\Delta\log\rho+\rho^{-1}|\nabla\rho|^2,$$
and
$$\nabla\Delta\rho=\nabla\rho\Delta\log\rho+\rho\nabla\Delta\log\rho-\rho^{-2}\nabla\rho|\nabla\rho|^2+2\rho^{-1}\nabla\rho\cdot\nabla^2\rho,$$
we complete the proof of \eqref{equation6.7}.
\end{proof}

\section{A Priori Estimates (II): Case \eqref{C}}\label{Section8}

In this section, we will prove Theorem \ref{theorem1.6} via a different approach. The main difficulty lines here is that, in this situation, $v\cdot n=0$ and $v=-c_0\nabla\rho^{-1}$ on $\partial\Omega$, which makes one impossible to handle the high order derivatives of $\rho$ appeared in the boundary integrals when we deal with the energy estimates of $v$. 

To over come it, we may take a different decomposition on $u$. This idea mainly comes from Lemma \ref{lemma27}, which pushes us to introduce a new function $Q=\cB[c_0\Delta\rho^{-1}]$ to eliminate the non-divergence-free condition on $u$. More precisely, we split $u$ into two parts, $u=w+Q$, and one can find that $w$ possesses the nice properties, that is, $w$ is divergence-free and $w|_{\partial\Omega}=0$. Therefore, we can use $w$ to get the energy estimates for system $\eqref{equation1.1}_2$.

Fortunately, in spite of this difficulty, we still has the estimates on $\rho$, which has been derived in Section \ref{section3} and \ref{Section6} such as \eqref{eq3.12}, \eqref{equation616}, etc. This is because those estimates only require $v\cdot n=0$ on $\Omega$. Then, using the relation $v=u-c_0\nabla\rho^{-1}$, one can easily change the norm of $v$ into that of $u$ and $\rho$.

Before giving the proof of Theorem \ref{theorem1.6}, we make some comments on the analysis in this section. In this section, we devote to establish the higher order estimates for $(\rho,u)$. One should notice that, if $Q=\cB[c_0\Delta\rho^{-1}]$, then $Q_t=\cB[c_0\Delta\rho^{-1}_t]$. With this fact, using Remark \ref{remark2.4} and Lemma \ref{lemma27}, one has the following estimates,  
\begin{equation}\label{8.2}
\begin{cases}
\norm{Q}_{L^p}\leq C\norm{\nabla\rho}_{L^p},\\
\norm{Q}_{H^1}\leq C\left(\norm{\Delta\rho}_{L^2}+\norm{\nabla\rho}_{L^{4}}^2\right),\\
\norm{Q_t}_{L^p}\leq C\left(\norm{\nabla\rho_t}_{L^p}+\norm{|\rho_t||\nabla\rho|}_{L^p}\right);
\end{cases}
\end{equation}
for $1<p<\infty$ and some constants $C$ depending only on $\Omega$, $c_0$ and $p$. In addition, the following equivalence will be used frequently, that is,
\begin{equation}\label{8.3}
\begin{aligned}
\norm{u}_{L^p}+\norm{\nabla\rho}_{L^p}&\sim\norm{v}_{L^p}+\norm{\nabla\rho}_{L^p},\\
\norm{\nabla u}_{L^p}+\norm{\Delta\rho}_{L^p}+\norm{\nabla\rho}_{L^{2p}}^2&\sim\norm{\nabla v}_{L^p}+\norm{\Delta\rho}_{L^p}+\norm{\nabla\rho}_{L^{2p}}^2.\\
\end{aligned}
\end{equation}

Now, we turn to the proof. The key of the proof is deriving the following proposition. Using the idea from \cite{huang}, we first assume the bounds \eqref{8.4} and obtain the a priori estimates of $(\rho,u)$, see Lemma \ref{lemma8.3}--\ref{lemma8.2}. Then, these bounds lead to smaller ones \eqref{8.5} provided $\norm{\nabla u_0}_{L^2}$ suitably small, which means that we can close the energy estimates of $(\rho,u)$ and, consequently, we complete the proof of the theorem.
\begin{Proposition}\label{prop8.1}
Suppose that $(\rho,u,\pi)$ is a smooth solution of \eqref{equation1.1}. There exists a positive constant $\delta$ depending only on $\Omega$, $\alpha$, $\beta$ and $c_0$ such that, if $\norm{\nabla u_0}_{L^2} \leq \delta$ and
\begin{equation}\label{8.4}
\sup_{t\in[0,T]}\norm{\nabla\rho}_{L^4}\leq 2,\quad\int_0^T\left(\norm{\Delta\rho}^4_{L^2}+\norm{\nabla u}^4_{L^2}\right)\,dt\leq 2\norm{\nabla u_0}_{L^2}^2,
\end{equation}
then the following estimates hold
\begin{equation}\label{8.5}
\sup_{t\in[0,T]}\norm{\nabla\rho}_{L^4}\leq 1,\quad\int_0^T\left(\norm{\Delta\rho}^4_{L^2}+\norm{\nabla u}^4_{L^2}\right)\,dt\leq \norm{\nabla u_0}_{L^2}^2.
\end{equation}
\end{Proposition}

In order to prove Proposition \ref{prop8.1}, we need the following estimates. 

\begin{Lemma}\label{lemma8.3}
Suppose that $(\rho,u,\pi)$ is a smooth solution of \eqref{equation1.1}. There exists some positive constant $C$ depending on $\Omega$, $\alpha$, $\beta$ and $c_0$ such that, for all $T\in (0,\infty)$, $\alpha\leq \rho\leq \beta$ and
\begin{equation}\label{eq87}
\sup_{t\in[0,T]}\normf{\rho-(\rho_0)_\Omega}^2_{L^2}+ \int_0^T\norm{\nabla\rho}^2_{L^2} \,dt \leq \norm{\rho_0-(\rho_0)_\Omega}_{L^2}^2.
\end{equation}
Furthermore, if $\norm{\nabla u_0}_{L^2}\leq 1$ and \eqref{8.4} holds, one has 
\begin{equation}\label{eq88}
\sup_{t\in[0,T]}\norm{\nabla\rho}_{L^2}^2+ \int_0^T\left(\norm{\nabla\rho}_{L^4}^4+ \norm{\Delta\rho}_{L^2}^2\right)\,dt\leq  C\norm{\nabla\rho_0}^2_{L^2}.
\end{equation}
\end{Lemma}

\begin{Lemma}\label{lemma8.2}
Suppose that $\norm{\nabla u_0}_{L^2}\leq 1$ and \eqref{8.4} is established, then one has 
\begin{align}
\sup_{t\in [0,T]}\cF(t)+\int_0^T\left(\cG(t)+\norm{\pi}_{H^1}^2\right)\,dt\leq C\norm{\nabla u_0}_{L^2}^2,\label{810}
\end{align}
where $\lambda$ and $C$ are positive constants depending on $\Omega$, $\alpha$, $\beta$ and $c_0$,
\begin{gather*}
\cF(t):=\norm{\nabla u}_{L^2}^2+\norm{\rho_t}_{L^2}^2+\norm{\Delta\rho}_{L^2}^2,\\ 
\cG(t):=\norm{u_t}_{L^2}^2+\norm{\Delta u}_{L^2}^2+\norm{\nabla\Delta\rho}_{L^2}^2+\norm{\nabla\rho_t}_{L^2}^2.
\end{gather*}
\end{Lemma}
\begin{Remark}
One should keep in mind that we always have
\begin{equation}
\norm{\rho_0-(\rho_0)_\Omega}_{L^2}\leq C\norm{\nabla\rho_0}_{L^2}\leq C\norm{u_0}_{L^2}\leq C\norm{\nabla u_0}_{L^2}.
\end{equation}
\end{Remark}

We temporarily assume that Lemma \ref{lemma8.3}--\ref{lemma8.2} are established and prove Proposition \ref{prop8.1}.
\begin{proof}[Proof of Proposition \ref{prop8.1}]
Since Lemma \ref{lemma8.2} is established, we have, for some $C_1>0$ depending only on $\Omega$,
\begin{equation*}
\sup_{t\in[0,T]}\norm{\nabla\rho}_{L^4}\leq C_1\sup_{t\in[0,T]}\norm{\Delta\rho}_{L^2}\leq C_1C\norm{\nabla u_0}_{L^2}\leq 1
\end{equation*}
provided
\begin{equation}
\norm{\nabla u_0}_{L^2}\leq \delta_1:= (C_1C)^{-1},
\end{equation}
and, by Lemma \ref{Lemma2.3},
\begin{align*}
\int_0^T\left(\norm{\Delta\rho}^4_{L^2}+\norm{\nabla u}^4_{L^2}\right)\,dt&\leq \left(\sup_{t\in[0,T]}\norm{\Delta\rho}_{L^2}^2+\sup_{t\in[0,T]}\norm{\nabla u}_{L^2}^2\right)\int_0^T \left(\norm{\Delta\rho}^2_{L^2}+\norm{\nabla u}^2_{L^2}\right)\,dt\\
&\leq C^2\norm{\nabla u_0}_{L^2}^4\leq \norm{\nabla u_0}_{L^2}^2
\end{align*}
provided
\begin{equation}
\norm{\nabla u_0}_{L^2}\leq \delta_2:=C^{-1},
\end{equation}
where $C$ is the constant in Lemma \ref{lemma8.2}.

Thus, if we choose 
$$\norm{\nabla u_0}_{L^2}\leq \delta:=\min\left\{1,\delta_1,\delta_2\right\},$$
then the proof of Proposition \ref{prop8.1} is completed.
\end{proof}

Now, we trun back to prove Lemma \ref{lemma8.3}--\ref{lemma8.2}. Since Lemma \ref{lemma8.3} has already been proved in Section \ref{section3}, we only give the proof for Lemma \ref{lemma8.2}.

\begin{proof}[Proof of Lemma \ref{lemma8.2}]
We start with the lower order estimate of $u$. Multiplying $w$ on the both sides of $\eqref{equation1.1}_2$ and integrating over $\Omega$, one has
\begin{equation}\label{equation8.16}
\begin{aligned}
\frac{d}{dt}\int \frac{1}{2}\rho|u|^2+\int 2\mu |D(u)|^2&=\int \rho u_t\cdot Q +\int \rho u\cdot \nabla u\cdot Q-\int \dive[2\mu D(u)]\cdot Q\\
&=\int \rho u_t\cdot Q +\int \rho u\cdot \nabla u\cdot Q+\int 2\mu D(u)\cdot \nabla Q\\
&:=\sum_{i=1}^3 S_i,
\end{aligned}
\end{equation}
where, using estimates \eqref{8.2}, Lemma \ref{Lemma2.3} and \ref{lemma4.1},
\begin{equation}\label{equation8.17}
\begin{cases}
|S_1|\leq C\norm{Q}_{L^2}\norm{u_t}_{L^2}\leq C_{\varepsilon_1}\norm{\nabla\rho}^2_{L^2}+\varepsilon_1\norm{u_t}^2_{L^2},\\
|S_2|\leq C\norm{u}_{L^4}\norm{\nabla u}_{L^2}\norm{Q}_{L^4}\leq C_{\varepsilon_2}\norm{\Delta\rho}^4_{L^2}\norm{u}_{L^2}^2+\varepsilon_2\norm{\nabla u}_{L^2}^2,\\
|S_3|\leq C\norm{\nabla Q}_{L^2}\norm{\nabla u}_{L^2}\leq C_{\varepsilon_3}\left(\norm{\Delta\rho}^2_{L^2}+\norm{\nabla\rho}^4_{L^4}\right)+\varepsilon_3\norm{\nabla u}^2_{L^2}.
\end{cases}
\end{equation}
Here, we still use the notation $\varepsilon_i\in (0,1/2]$ and the constant $C_{\varepsilon_i}$ as before. Combining \eqref{equation8.16} and \eqref{equation8.17} leads to, $\exists\, \nu>0$,
\begin{equation}\label{818}
\begin{aligned}
\frac{d}{dt}\norm{u}_{L^2}^2+\nu\norm{\nabla u}_{L^2}^2&\leq C_\varepsilon\left(\norm{\Delta\rho}^4_{L^2}\norm{u}_{L^2}^2+\norm{\Delta\rho}^2_{L^2}+\norm{\nabla\rho}_{L^4}^4\right)\\
&\quad+C_\varepsilon\norm{\nabla\rho}_{L^2}^2+\varepsilon\norm{u_t}_{L^2}^2,
\end{aligned}
\end{equation}

For the estimate of $u_t$, multiplying $w_t$ on the both sides of $\eqref{equation1.1}_2$, one has
\begin{equation}\label{8818}
\begin{aligned}
\int\rho|u_t|^2+\frac{d}{dt}\int\mu |D(u)|^2&=-\int\rho u_t\cdot Q_t-\int\rho u\cdot \nabla u\cdot w_t\\
&\quad+\int\mu_t|D(u)|^2-\int\dive[2\mu D(u)]\cdot Q_t\\
&:=\sum_{i=1}^4U_i.
\end{aligned}
\end{equation}
Using Lemma \ref{Lemma2.3} and \ref{lemma4.1}, \eqref{8.2}--\eqref{8.4} and Poincar\'e's inequality, we have
\begin{equation}
\norm{Q_t}^2_{L^2}\leq C\left(\norm{\nabla\rho_t}_{L^2}^2+\norm{\rho_t}_{L^4}^2\norm{\nabla\rho}_{L^4}^2\right)\leq C\norm{\nabla\rho_t}_{L^2}^2
\end{equation}
and, thus,
\begin{equation}\label{888}
\begin{cases}
|U_1|&\!\!\!\!\leq C\norm{Q_t}_{L^2}\norm{u_t}_{L^2}\leq  C_{\varepsilon_1}\norm{\nabla\rho_t}^2_{L^2}+\varepsilon_1\norm{u_t}^2_{L^2},\\
|U_2|&\!\!\!\!\leq C\norm{u}_{L^4}\norm{\nabla u}_{L^4}\norm{w_t}_{L^2}\\
&\!\!\!\!\leq C_{\varepsilon_2}\norm{u}_{L^4}^2\norm{\nabla u}_{L^4}^2+\varepsilon_2\norm{u_t}_{L^2}^2+C\norm{\nabla\rho_t}_{L^2}^2\\
&\!\!\!\!\leq C_{\varepsilon_2,\varepsilon_3}\norm{\nabla u}_{L^2}^4\norm{\nabla u}_{L^2}^2+C\norm{\nabla\rho_t}_{L^2}^2+\varepsilon_2\norm{u_t}_{L^2}^2+\varepsilon_3\norm{\Delta u}_{L^2}^2,\\
|U_3|&\!\!\!\!\leq C\norm{\rho_t}_{L^2}\norm{\nabla u}^2_{L^4}\leq C_{\varepsilon_4}\norm{\nabla u}_{L^2}^2\norm{\rho_t}_{L^2}^2+\varepsilon_4\norm{\Delta u}^2_{L^2}\\
&\!\!\!\!\leq C_{\varepsilon_4}\norm{\nabla u}_{L^2}^4\norm{\rho_t}_{L^2}^2+\varepsilon_4\left(\norm{\nabla\rho_t}^2_{L^2}+\norm{\Delta u}^2_{L^2}\right),\\
|U_4|&\!\!\!\!\leq C\left(\norm{\nabla\rho}_{L^4}\norm{\nabla u}_{L^4}+\normf{\nabla^2 u}_{L^2}\right)\norm{\nabla\rho_t}_{L^2}\\
&\!\!\!\!\leq C_{\varepsilon_5}\norm{\Delta\rho}_{L^2}^4\norm{\nabla u}_{L^2}^2+C_{\varepsilon_5}\norm{\nabla\rho_t}_{L^2}^2+\varepsilon_5\norm{\Delta u}_{L^2}^2.
\end{cases}
\end{equation}
Thus, combining \eqref{8818} and \eqref{888}, one has
\begin{equation}\label{equation821}
\begin{aligned}
\frac{d}{dt}\normf{\sqrt{\mu}D(u)}_{L^2}^2+\nu\norm{u_t}_{L^2}^2&\leq C_\varepsilon\left(\norm{\nabla u}_{L^2}^4+\norm{\Delta\rho}_{L^2}^4\right)\norm{\nabla u}_{L^2}^2+C_\varepsilon\norm{\nabla u}_{L^2}^4\norm{\rho_t}_{L^2}^2\\
&\quad+C_\varepsilon\norm{\nabla\rho_t}_{L^2}^2+\varepsilon\norm{\Delta u}_{L^2}^2.
\end{aligned}
\end{equation}

From the observation of \eqref{equation821}, one have to derive the estimate of $\Delta u$, or that of $\Delta v$. Unfortunately, we can not directly use, for example, \eqref{eq635} in the Section \ref{Section6}. The main obstacle here is that  \eqref{eq635} strongly depends on the conitnuity of $\rho$ and we have not closed the lower bounds of $v$ yet, so that we can not apply Lemma \ref{lemma61}--\ref{Lemma5.2} (notice that Lemma \ref{lemma61} requiring $v\in L^s(0,T;L^r)$). Consequently, we apply Lemma \ref{lemma8.4} with $\Phi=-c_0\nabla\rho^{-1}$ on 
\begin{equation}\label{e7718}
-\dive[2\mu D(v)]+ \nabla\pi= F,
\end{equation}
where 
\begin{equation*}
\begin{aligned}
 F=-\rho u_t- \rho u\cdot\nabla u+c_0\dive\left(2\mu \nabla^2\rho^{-1}\right)
\end{aligned}
\end{equation*}
and, using condition \eqref{8.4}, Lemma \ref{Lemma2.3} and \ref{lemma4.1} and Poincar\'e's inequality,
\begin{equation}\label{88.21}
\begin{aligned}
\norm{F}_{L^2}&\leq C\norm{u_t}_{L^2}+ C_{\varepsilon_2}\norm{u}_{L^4}^2\norm{\nabla u}_{L^2}+C\norm{\nabla\rho}_{L^4}^2\norm{\Delta\rho}_{L^2}\\
&\quad+ \varepsilon_2\norm{\Delta u}_{L^2}+C\norm{\nabla\Delta\rho}_{L^2}\\
&\leq C\norm{u_t}_{L^2}+ C_{\varepsilon_2}\left(\norm{u}_{L^4}^2+\norm{\nabla\rho}_{L^4}^2\right)\left(\norm{\nabla u}_{L^2}+\norm{\Delta\rho}_{L^2}\right) \\
&\quad+ \varepsilon_2\norm{v}_{H^2}+C\norm{\nabla\Delta\rho}_{L^2},
\end{aligned}
\end{equation}
where we have applied the estimate \eqref{8.3} and 
\begin{equation}\label{uv}
\begin{aligned}
\norm{\Delta u}_{L^2}&\leq C\left(\norm{\Delta v}_{L^2}+\norm{\nabla\Delta\rho^{-1}}_{L^2}\right)\\
&\leq C\left(\norm{\Delta v}_{L^2}+\norm{\nabla\rho}_{L^4}^2\norm{\Delta\rho}_{L^2}+\norm{\nabla\Delta\rho}_{L^2}\right)\quad \text{use }\eqref{8.4}\\
&\leq C\left(\norm{\Delta v}_{L^2}+\norm{\nabla\Delta\rho}_{L^2}\right).
\end{aligned}
\end{equation}
Then, using Lemma \ref{Lemma2.3}, Lemma \ref{lemma8.4} and Poincar\'e's inequality, combining the condition \eqref{8.4} and the estimates \eqref{88.21}, we can derive a similar estimate of \eqref{eq635}, that is,
\begin{equation*}
\norm{v}_{H^2}^2+\norm{\nabla \pi}_{L^2}^2\leq C\norm{u_t}^2_{L^2}+C\left(\norm{u}_{L^4}^4+\norm{\nabla\rho}_{L^4}^4\right)\left(\norm{\nabla u}^2_{L^2}+\norm{\Delta\rho}^2_{L^2}\right)+C\norm{\nabla\Delta\rho}_{L^2}^2.
\end{equation*}
Using again \eqref{uv} we derive that
\begin{equation}\label{821}
\begin{aligned}
\norm{\Delta u}_{L^2}^2+\norm{\nabla \pi}_{L^2}^2&\leq C\norm{u_t}^2_{L^2}+C\left(\norm{u}_{L^4}^4+\norm{\nabla\rho}_{L^4}^4\right)\left(\norm{\nabla u}^2_{L^2}+\norm{\Delta\rho}^2_{L^2}\right)\\
&\quad+C\norm{\nabla\Delta\rho}_{L^2}^2.
\end{aligned}
\end{equation}
Then, substituting \eqref{821} into \eqref{equation821}, $\exists\,\nu>0$,
\begin{equation}\label{822}
\begin{aligned}
\frac{d}{dt}\norm{\nabla u}_{L^2}^2+\nu\norm{\Delta u}_{L^2}^2+\nu\norm{u_t}_{L^2}^2&\leq C_\varepsilon\left(\norm{\nabla u}_{L^2}^4+\norm{\Delta\rho}_{L^2}^4\right)\cF(t)\\
&\quad+C_\varepsilon \norm{\nabla\rho_t}_{L^2}^2+\varepsilon\norm{\nabla\Delta\rho}_{L^2}^2.
\end{aligned}
\end{equation}

Next, in order to close the estimate \eqref{822}, we turn to get the bounds of $\nabla\rho_t$ and $\nabla\Delta\rho$. From the estimate \eqref{equation616} and 
$$
\begin{aligned}
\norm{v_t}_{L^2}&\leq C(\norm{u_t}_{L^2}+\norm{\nabla\rho^{-1}_t}_{L^2})\\
&\leq C\left(\norm{u_t}_{L^2}+\norm{\nabla\rho_t}_{L^2}+\norm{\nabla\rho}_{L^4}^2\norm{\rho_t}_{L^2}\right)\quad \text{use }\eqref{8.4}\\
&\leq C(\norm{u_t}_{L^2}+\norm{\nabla\rho_t}_{L^2})
\end{aligned}
$$
we have
\begin{equation}\label{8824}
\frac{d}{dt}\norm{\rho_t}_{L^2}^2 +\nu\norm{\nabla\rho_t}_{L^2}^2
\leq C_\varepsilon\left(\norm{\nabla\rho}_{L^4}^4+\norm{\Delta\rho}_{L^2}^2\right)\norm{\rho_t}_{L^2}^2+\varepsilon\norm{u_t}_{L^2}^2.
\end{equation}

On the other hand, for $\nabla\Delta\rho$, since the estimate \eqref{equa5.11} we derived in the Section \ref{Section6} is still valid, replacing $v$ by $u$ and $\nabla\rho$, one has
\begin{equation}
\begin{aligned}
\frac{d}{dt}\norm{\Delta\rho}_{L^2}^2+\nu\norm{\nabla\Delta\rho}_{L^2}^2&\leq C_\varepsilon\left(\norm{\Delta\rho}_{L^2}^4+\norm{\nabla u}_{L^2}^4\right)\left(\norm{\Delta\rho}_{L^2}^2+\norm{\nabla u}_{L^2}^2\right)+\varepsilon \norm{\Delta u}_{L^2}^2.
\end{aligned}
\end{equation}
Using this inequality together with \eqref{821}, we eliminate term $\Delta u$ and, then, we subsititute it, alonging with\eqref{8824}, into \eqref{822} to deduce that
\begin{equation}\label{827}
\frac{d}{dt}\cF(t)+\nu \cG(t)\leq C\left(\norm{\nabla u}_{L^2}^4+\norm{\Delta\rho}_{L^2}^4+\norm{\Delta\rho}_{L^2}^2\right)\cF(t).
\end{equation}

Since we have $\norm{\nabla\rho_0}_{H^1}\leq \norm{\nabla u_0}_{L^2}$ from Remark \ref{remark18}, applying Gr$\mathrm{\ddot{o}}$nwall's inequality for \eqref{827} and using Lemma \ref{lemma8.3}, we have
\begin{align}
\sup_{t\in [0,T]}\cF(t)+\int_0^T\cG(t)\,dt\leq C\left(\norm{\nabla u_0}_{L^2}^2+\norm{\nabla u_0}_{L^2}^4\right)\leq C\norm{\nabla u_0}_{L^2}^2,
\end{align}
which, turning back to \eqref{821} to get the bound for $\pi$, implies the estimate \eqref{810}. 
Therefore, we complete the proof of Lemma \ref{lemma8.2}.
\end{proof}

\section{Proof of Theorem \ref{theo:Serrin-type}}\label{Section7}

In this section, we devote to accomplish the proofs of Theorem \ref{theo:Serrin-type} in several steps. Our proofs are basically relying on the approach in \cite{cho}. In subsection \ref{LP},  we are going to solve the linearized system and give some basic uniform estimates, which is critical for the existence proofs in next few subsections. Next, in subsection \ref{PR}--\ref{AS}, we will construct an approximate system and use the contraction mapping theorem to show that it admits a unique smooth solution. Finally, in subsection \ref{P}, we will prove Theorem \ref{theo:Serrin-type}.

\subsection{Linearized Problem}\label{LP}

Consider  the following linearized problem
\begin{equation}\label{linearized}
\begin{cases}
\rho_t+\Phi\cdot\nabla\rho-\dive(\varphi^{-1}\nabla\rho)=0,\\
\rho u_t+\rho(\Phi+\nabla\varphi^{-1})\cdot\nabla u-\dive(2\mu D)+\nabla p=0,\\
\dive u=c_0\Delta\rho^{-1},\\
\rho|_{t=0}=\rho_0,\quad u|_{t=0}=u_0,\\ \alpha\leq\rho_0\leq\beta,\,\,(\rho_0,u_0)\in [C^\infty(\overline\Omega)]^4\text{ satisfying }\eqref{compat},\\
(\rho,u)\text{ satisfies one of the bundary conditions }\eqref{A}-\eqref{C},
\end{cases}\tag{L.P.}
\end{equation}
where $\int p=0$, $\mu=\mu(x,t)\in H^1(0,T;C^\infty(\overline \Omega))$ is a positive function.

\begin{Lemma}[Linearized problem]\label{LP}
Assume that the hypotheses of Theorem \ref{theo:Serrin-type} are satisfied by the data $(\rho_0,u_0)$. If $\Phi$ and $\varphi$ satisifes the following conditions
\begin{equation}\label{2215}
\begin{cases}
\Phi\in C([0,T];H^1)\cap L^2(0,T;H^2),\quad \Phi_t\in L^2(0,T;L^2),\\
\varphi\in C([0,T];H^2)\cap L^2(0,T;H^3),\quad \varphi_t\in C([0,T];L^2)\cap L^2(0,T;H^1),\\
c^{-1}\leq \varphi\leq c,\quad \dive\Phi=0\text{ in }\Omega,
\end{cases}
\end{equation}
then there exists a unique global strong solution $(\rho,u,p)$ to the problem \eqref{linearized} satisfying \eqref{strong}.
\end{Lemma}
\begin{proof}
We only give the a priori estimates. The unique sovablity is obvious, since we can first solve $\eqref{linearized}_1$ by the theories of linear parabolic equations, see \cite{ladyzhenskaia} and, then, derive $u$ from $\eqref{linearized}_2$, see \cite{cho, solonnikov}. 

Firstly, the sup-bound and lower order estimates of $\rho$ has been proved in Section \ref{section3}. See Lemma \ref{lemma4.1} and \ref{lemma43}, that is,
\begin{equation}\label{554.}
\alpha\leq \rho\leq \beta,\quad \norm{\rho}_{L^2}^2+\int_0^t\norm{\nabla\rho}_{L^2}^2\,ds\leq C(\Omega,c)\,\,(\text{or }C(\Omega,c,\tilde\rho)).
\end{equation}

Next, we multiply $-\Delta\rho$ on $\eqref{linearized}_1$ and integrate over $\Omega$, then, using Lemma \ref{Lemma2.3}, we have
\begin{equation*}
\begin{aligned}
\frac{d}{dt}\norm{\nabla\rho}_{L^2}^2+\nu\norm{\Delta\rho}_{L^2}^2&\leq C\int \left(|\Phi|+|\nabla\varphi|\right)|\nabla\rho||\Delta\rho|\\
&\leq C\left(\norm{\Phi}_{L^4}^4+\norm{\nabla\varphi}_{L^4}^4\right)\norm{\nabla\rho}_{L^2}^2+\frac{\nu}{2}\norm{\Delta\rho}_{L^2}^2.
\end{aligned}
\end{equation*}
Thus, 
\begin{equation}\label{554}
\frac{d}{dt}\norm{\nabla\rho}_{L^2}^2+\nu\norm{\Delta\rho}_{L^2}^2\leq C\left(\norm{\Phi}_{L^4}^4+\norm{\nabla\varphi}_{L^4}^4\right)\norm{\nabla\rho}_{L^2}^2.
\end{equation}
By virtue of Gr$\mathrm{\ddot{o}}$nwall's inequality, we derive
\begin{equation}\label{2218}
\norm{\nabla\rho}_{L^2}^2(t)+\nu\int_0^t\norm{\Delta\rho}_{L^2}^2\,ds\leq \norm{\nabla\rho_0}_{L^2}^2\exp\left\{C\int_0^t\left(\norm{\Phi}_{L^4}^4+\norm{\nabla\varphi}_{L^4}^4\right)\,ds\right\}.
\end{equation}
To get the higher order bounds, if $n\cdot\nabla\rho=0$ on $\partial\Omega$, we apply $-\nabla\Delta\rho\nabla$ on both sides of $\eqref{linearized}_1$ and integrate over $\Omega$, then
\begin{equation*}
\begin{aligned}
\frac{d}{dt}\int \frac{1}{2}\abs{\Delta\rho}^2 + \int \varphi^{-1}\abs{\nabla\Delta\rho}^2&= \int\nabla\Delta\rho\cdot\nabla \Phi\cdot\nabla\rho + \int \Phi\cdot\nabla^2\rho\cdot\nabla\Delta\rho\\
&\quad-\int 2 \varphi^{-3}|\nabla\varphi|^2\nabla\rho\cdot\nabla\Delta\rho+\int\varphi^{-2}\nabla\varphi\cdot\nabla^2\rho\cdot\nabla\Delta\rho\\
&\quad +\int\varphi^{-2}\nabla\rho\cdot\nabla^2\varphi\cdot\nabla\Delta\rho- \int\varphi^{-2}\Delta\rho\nabla\varphi\cdot\nabla\Delta\rho.
\end{aligned}
\end{equation*}
Then, applying Lemma \ref{Lemma2.3} and Poincar\,e's inequality, we have
\begin{equation*}
\begin{aligned}
&\frac{d}{dt}\norm{\Delta\rho}_{L^2}^2+\nu\norm{\nabla\Delta\rho}_{L^2}^2\\
&\leq \left(\norm{\nabla\Phi}_{L^4}\norm{\nabla\rho}_{L^4}+\norm{\Phi}_{L^4}\norm{\Delta\rho}_{L^4}\right)\norm{\nabla\Delta\rho}_{L^2}+C\norm{\nabla\varphi}_{L^8}^2\norm{\nabla\rho}_{L^4}\norm{\nabla\Delta\rho}_{L^2}\\
&\quad+C\left(\norm{\nabla\varphi}_{L^4}\norm{\Delta\rho}_{L^4}+\norm{\nabla^2\varphi}_{L^4}\norm{\nabla\rho}_{L^4}\right)\norm{\nabla\Delta\rho}_{L^2}\\
&\leq C\left(\norm{\Phi}_{W^{1,4}}^4+\norm{\varphi}_{W^{2,4}}^4+\norm{\nabla\varphi}_{L^8}^8\right)\norm{\Delta\rho}_{L^2}^2   +\frac{\nu}{2}\norm{\nabla\Delta\rho}_{L^2}^2,
\end{aligned}
\end{equation*}
that is
\begin{equation}\label{557}
\frac{d}{dt}\norm{\Delta\rho}_{L^2}^2+\nu\norm{\nabla\Delta\rho}_{L^2}^2\leq C\left(\norm{\Phi}_{W^{1,4}}^4+\norm{\varphi}_{W^{2,4}}^4+\norm{\nabla\varphi}_{L^8}^8\right)\norm{\Delta\rho}_{L^2}^2,
\end{equation}
which, using Gr$\mathrm{\ddot{o}}$nwall's inequality, leads to
\begin{equation}\label{558}
\begin{aligned}
&\norm{\Delta\rho}_{L^2}^2(t)+\nu\int_0^t\norm{\nabla\Delta\rho}_{L^2}^2\,ds\\
&\leq C\norm{\Delta\rho_0}_{L^2}^2\exp\left\{\int_0^t\left(\norm{\Phi}_{W^{1,4}}^4+\norm{\varphi}_{W^{2,4}}^4+\norm{\nabla\varphi}_{L^8}^8\right)\,ds\right\}.
\end{aligned}
\end{equation}

For $\rho$ satisfying the non-homogeneous Dirichlet condition,  that is, $\rho|_{\partial\Omega}=\tilde\rho$, taking $\rho_t\partial_t$ on $\eqref{linearized}_1$, one has
\begin{equation*}
\begin{aligned}
\frac{d}{dt}\int\frac{1}{2}\abs{\rho_t}^2 +\nu\int\abs{\nabla\rho_t}^2 &\leq\int \abs{\Phi_t}\abs{\nabla\rho}\abs{\rho_t}+ C\int \abs{\varphi_t}|\rho_t||\nabla\varphi||\nabla\rho|\\
&\quad+C\int \abs{\rho_t}\abs{\nabla\rho_t}\abs{\nabla\varphi}+C\int \abs{\rho_t}\abs{\nabla\varphi_t}\abs{\nabla\rho}\\
&\quad+C\int |\varphi_t||\rho_t|\abs{\Delta\rho}.
\end{aligned}
\end{equation*}
Simiarly, it follows from Lemma \ref{Lemma2.3} that
\begin{equation*}
\begin{aligned}
&\frac{d}{dt}\norm{\rho_t}_{L^2}^2+\nu\norm{\nabla\rho_t}_{L^2}^2\\
&\leq C\left(\norm{\Phi_t}_{L^2}+\norm{\nabla\varphi}_{L^4}\norm{\varphi_t}_{L^4}\right)\norm{\nabla\rho}_{L^4}\norm{\rho_t}_{L^4}\\
&\quad+C\left(\norm{\nabla\varphi}_{L^4}\norm{\nabla\rho_t}_{L^2}+\norm{\nabla\rho}_{L^4}\norm{\nabla\varphi_t}_{L^2}\right)\norm{\rho_t}_{L^4}+C\norm{\Delta\rho}_{L^2}\norm{\varphi_t}_{L^4}\norm{\rho_t}_{L^4}\\
&\leq \left(\norm{\Phi_t}_{L^2}^2+\norm{\varphi_t}_{L^4}^4+\norm{\nabla\varphi}_{L^4}^4+\norm{\nabla\varphi_t}_{L^2}^2\right)\norm{\rho_t}_{L^2}^2+\frac{\nu}{2}\norm{\nabla\rho_t}_{L^2}^2\\
&\quad+C\left(\norm{\Phi_t}^2_{L^2}+\norm{\varphi_t}^4_{L^4}+\norm{\nabla\varphi}^4_{L^4}+\norm{\nabla\varphi_t}^2_{L^2}\right)\norm{\nabla\rho}_{L^2}^2+C\norm{\Delta\rho}_{L^2}^2.
\end{aligned}
\end{equation*}
that is,
\begin{equation}\label{5510}
\begin{aligned}
&\frac{d}{dt}\norm{\rho_t}_{L^2}^2+\nu\norm{\nabla\rho_t}_{L^2}^2\\
&\leq C\left(\norm{\Phi_t}_{L^2}^2+\norm{\varphi_t}_{L^4}^4+\norm{\nabla\varphi}_{L^4}^4+\norm{\nabla\varphi_t}_{L^2}^2\right)\left(\norm{\rho_t}_{L^2}^2+\norm{\nabla\rho}_{L^2}^2\right)+C\norm{\Delta\rho}_{L^2}^2.
\end{aligned}
\end{equation}
Then, using Gr$\mathrm{\ddot{o}}$nwall's inequality and \eqref{2218}, one has
\begin{equation}\label{2224}
\begin{aligned}
\norm{\rho_t}_{L^2}^2(t)+\nu\int_0^t\norm{\nabla\rho_t}_{L^2}^2\,ds\leq C(\Omega,c,\alpha,\beta,\Phi,\varphi,u_0).
\end{aligned}
\end{equation}
Noticing that \eqref{2224} also holds for the Neumann case. 

Next, we take $L^2$-norm on $\eqref{linearized}_1$ and use Lemma \ref{Lemma2.3} to get
\begin{equation}\label{equation5512}
\begin{aligned}
\norm{\Delta\rho}_{L^2}^2
\leq C\norm{\rho_t}_{L^2}^2+C\left(\norm{\Phi}_{L^4}^4+\norm{\nabla\varphi}_{L^4}^4\right)\norm{\nabla\rho}_{L^2}^2
\end{aligned}
\end{equation}
and take $\nabla$ on both sides of $\eqref{linearized}_1$ to get
\begin{equation}\label{equation5513}
\begin{aligned}
\norm{\nabla\Delta\rho}_{L^2}^2\leq  C\norm{\nabla\rho_t}_{L^2}^2+C\left(\norm{\Phi}_{W^{1,4}}^4+\norm{\nabla\varphi}_{L^8}^8+\norm{\nabla\varphi}_{L^4}^4\right)\norm{\Delta\rho}_{L^2}^2.
\end{aligned}
\end{equation}
Thus, we use \eqref{equation5512} and \eqref{equation5513}, alonging with \eqref{2218} and \eqref{5510}, to deduce that
\begin{equation}\label{2227}
\begin{aligned}
&\norm{\Delta\rho}_{L^2}^2(t)+\int_0^t\norm{\nabla\Delta\rho}_{L^2}^2\,ds\leq  C(\Omega,c,\alpha,\beta,\Phi,\varphi,u_0).
\end{aligned}
\end{equation}

In conclusion, for both cases, it follows from \eqref{554.}, \eqref{2218}, \eqref{558}, \eqref{2224} and \eqref{2227} that
\begin{equation}\label{equation515}
\begin{aligned}
\norm{\rho_t}_{L^2}^2(t)+\norm{\nabla\rho}_{H^1}^2(t)+\int_0^t\left(\norm{\rho_t}_{H^1}^2+\norm{\nabla\rho}_{H^2}^2\right)\,ds\leq C(\Omega,c,\alpha,\beta,\Phi,\varphi,u_0).
\end{aligned}
\end{equation}

The next part is estimating $v$ (for case \eqref{A} or \eqref{B}) or $u$ (for case \eqref{C}). We first treat the case for $u$ satisfying \eqref{C}. Note that $\eqref{linearized}_1$ is equivalent to
\begin{equation*}
\rho_t+\dive[\rho(\Phi+\nabla\varphi^{-1})]=\Delta(\rho\varphi^{-1}).
\end{equation*}
Thus, if we multiply $\eqref{linearized}_2$ by $w:=u-Q$ with $Q=\cB[c_0\Delta\rho^{-1}]$ and integrate over $\Omega$, we have
\begin{equation*}
\begin{aligned}
&\frac{d}{dt}\int \frac{1}{2}\rho|u|^2+\int 2\mu|D(u)|^2\\
&=\int \Delta\left({\rho}{\varphi}^{-1}\right)\frac{|u|^2}{2}+\int \rho u_t\cdot Q+\int \rho (\Phi+\nabla\varphi^{-1})\cdot\nabla u\cdot Q+\int 2\mu D(u)\cdot \nabla Q.
\end{aligned}
\end{equation*}
Then, using Lemma \ref{Lemma2.3} and \ref{Lemma2.2}, we have
\begin{equation*}
\begin{aligned}
\frac{d}{dt}\norm{\sqrt{\rho}u}_{L^2}^2+\nu\norm{\nabla u}_{L^2}^2&\leq C\left(\norm{\Delta \rho}_{L^2}+\norm{\nabla\rho}_{L^4}\norm{\nabla\varphi}_{L^4}+\norm{\nabla\varphi}_{L^4}^2+\norm{\Delta\varphi}_{L^2}\right)\norm{u}_{L^4}^2\\
&\quad+C\norm{\Phi+\nabla\varphi^{-1}}_{L^4}\norm{Q}_{L^4}\norm{\nabla u}_{L^2}+C\norm{\nabla Q}_{L^2}\norm{\nabla u}_{L^2}\\
&\quad +C\norm{Q}_{L^2}\norm{u_t}_{L^2}\\
&\leq C\left(\norm{\Delta \rho}^2_{L^2}+\norm{\nabla\rho}^4_{L^4}+\norm{\nabla\varphi}_{L^4}^4+\norm{\Delta\varphi}_{L^2}^2\right)\norm{u}_{L^2}^2\\
&\quad+C\norm{\Phi+\nabla\varphi^{-1}}^4_{L^4}\norm{Q}_{L^2}^2+C\norm{\nabla Q}_{L^2}^2+\frac{\nu}{2}\norm{\nabla u}_{L^2}^2\\
&\quad +C_\varepsilon\norm{Q}_{L^2}^2+\varepsilon\norm{u_t}_{L^2}^2,
\end{aligned}
\end{equation*}
that is,
\begin{equation}\label{equation518}
\begin{aligned}
\frac{d}{dt}\norm{\sqrt{\rho}u}_{L^2}^2+\nu\norm{\nabla u}_{L^2}^2&\leq C\left(\norm{\Delta \rho}^2_{L^2}+\norm{\nabla\rho}^4_{L^4}+\norm{\nabla\varphi}_{L^4}^4+\norm{\Delta\varphi}_{L^2}^2\right)\norm{u}_{L^2}^2\\
&\quad+C\norm{\Phi+\nabla\varphi^{-1}}^4_{L^4}\norm{\nabla\rho}_{L^2}^2+C_\varepsilon\norm{Q}_{H^1}^2+\varepsilon\norm{u_t}_{L^2}^2,\\
&\leq C\left(\norm{\Delta \rho}^2_{L^2}+\norm{\nabla\rho}^4_{L^4}+\norm{\nabla\varphi}_{L^4}^4+\norm{\Delta\varphi}_{L^2}^2\right)\norm{u}_{L^2}^2\\
&\quad+C\norm{\Phi+\nabla\varphi^{-1}}^4_{L^4}\norm{\nabla\rho}_{L^2}^2+C_\varepsilon\left(\norm{\nabla \rho}^4_{L^4}+\norm{\nabla\rho}_{H^1}^2\right)\\
&\quad +\varepsilon\norm{u_t}_{L^2}^2,
\end{aligned}
\end{equation}

Next, for the estimate of $u_t$, multiplying $w_t=u_t-Q_t$ on the both sides of $\eqref{linearized}_2$, one has
\begin{equation*}
\begin{aligned}
\int\rho|u_t|^2+\frac{d}{dt}\int\mu|D(u)|^2&=-\int\rho u_t\cdot Q_t-\int\rho(\Phi+\nabla\varphi^{-1})\cdot \nabla u\cdot w_t\\
&\quad+\int\mu_t|D(u)|^2-\int\dive[2\mu D(u)]\cdot Q_t.
\end{aligned}
\end{equation*}
Using again Lemma \ref{Lemma2.3}, applying Poincar\'e's inequality and the fact that
\begin{equation*}
\norm{Q_t}^2_{L^2}\leq C\left(\norm{\nabla\rho_t}_{L^2}^2+\norm{\rho_t}_{L^4}^2\norm{\nabla\rho}_{L^4}^2\right),
\end{equation*}
we obtain 
\begin{equation}\label{equation520}
\begin{aligned}
&\norm{u_t}_{L^2}^2+\frac{d}{dt}\norm{\sqrt{\mu}D(u)}_{L^2}^2\\
&\leq C\left(\norm{\Phi+\nabla\varphi^{-1}}_{L^4}^2+\norm{\mu_t}_{L^2}+\norm{\nabla\mu}_{L^4}^2\right)\norm{\nabla u}^2_{L^4}+C\norm{Q_t}_{L^2}^2+C\norm{Q_t}_{L^2}\norm{\Delta u}_{L^2}\\
&\leq C_{\varepsilon}\left(\norm{\Phi+\nabla\varphi^{-1}}_{L^4}^4+\norm{\mu_t}^2_{L^2}+\norm{\nabla\mu}_{L^4}^4+\norm{\nabla\rho}_{L^4}^4\right)\norm{\nabla u}^2_{L^2}+\varepsilon\norm{\Delta u}_{L^2}^2\\
&\quad+C_\varepsilon\left(\norm{\nabla\rho}_{L^4}^4\norm{\rho_t}_{L^2}^2+\norm{\nabla\rho_t}_{L^2}^2\right)
\end{aligned}
\end{equation}

To estimate $\Delta u$, we change $\eqref{linearized}_2$ into the form
\begin{equation*}
-\mu\Delta v+\nabla p=2\nabla\mu\cdot D(u)-\rho u_t-\rho(\Phi+\nabla\varphi^{-1})\cdot\nabla u+2c_0\mu\nabla\Delta\rho^{-1},
\end{equation*}
which, using Lemma \ref{Lemma5.2}, leads to
\begin{equation}\label{equation5515}
\begin{aligned}
\norm{v}_{H^2}^2+\norm{p}_{H^1}^2&\leq C\norm{u_t}_{L^2}^2+C\norm{\nabla \Delta\rho^{-1}}_{L^2}^2\\
&\quad+C\left(\norm{\nabla\mu}_{L^4}^2+\norm{\Phi+\nabla \varphi^{-1}}_{L^4}^2\right)\norm{\nabla u}_{L^4}^2,
\end{aligned}
\end{equation}
that is, using Lemma \ref{Lemma2.3},
\begin{equation}\label{0517}
\begin{aligned}
\norm{\Delta u}_{L^2}^2+\norm{p}_{H^1}^2&\leq C\norm{u_t}_{L^2}^2+C\norm{\nabla \Delta\rho^{-1}}_{L^2}^2\\
&\quad+C\left(\norm{\nabla\mu}_{L^4}^4+\norm{\Phi+\nabla \varphi^{-1}}_{L^4}^4\right)\norm{\nabla u}_{L^2}^2.
\end{aligned}
\end{equation}
Then, using this bound together with \eqref{equation520}, we have
\begin{equation}\label{equation5.19}
\begin{aligned}
&\norm{u_t}_{L^2}^2+\frac{d}{dt}\norm{\sqrt{\mu}D(u)}_{L^2}^2\\
&\leq C_{\varepsilon}\left(\norm{\Phi+\nabla\varphi^{-1}}_{L^4}^4+\norm{\mu_t}^2_{L^2}+\norm{\nabla\mu}_{L^4}^4+\norm{\nabla\rho}_{L^4}^4\right)\norm{\nabla u}^2_{L^2}+\varepsilon\norm{\nabla\Delta \rho^{-1}}_{L^2}^2\\
&\quad+C\norm{\nabla\rho}_{L^4}^4\norm{\rho_t}_{L^2}^2+C\norm{\nabla\rho_t}_{L^2}^2
\end{aligned}
\end{equation}
Finally, combining \eqref{equation518} and \eqref{equation5.19}, then, using Gr$\mathrm{\ddot{o}}$nwall's inequality and the bound \eqref{equation515}, we obtain the a priori estimates for $u$.

For case \eqref{A} or \eqref{B}, as we have said at the end of Section \ref{Section1}, we convert \eqref{linearized} into
\begin{equation}\label{2231}
\begin{cases}
\rho_t+\Phi\cdot\nabla\rho-\dive(\varphi^{-1}\nabla\rho)=0,\\
\quad\\
\begin{cases}
\rho v_t+\rho(\Phi+\nabla\varphi^{-1})\cdot\nabla v-\dive(2\mu D(v))+\nabla p\\
=c_0\nabla(\log\rho)_t-c_0\rho(\Phi+\nabla\varphi^{-1})\cdot\nabla^2\rho^{-1}+c_0\dive(2\mu \nabla^2\rho^{-1}),
\end{cases}\\
\quad\\
\dive v=0.
\end{cases}
\end{equation}
Then, we can apply the energy arguements analogous to the case \eqref{C}. More precisely, multiplying $v$ on both sides of $\eqref{2231}_2$ and integrating over $\Omega$, we have, for all $\varepsilon\in (0,1/2]$,
\begin{equation}\label{2232}
\begin{aligned}
&\frac{d}{dt}\int\frac{1}{2}\rho\abs{v}^2-\int \dive{[2\mu D(v)]}\cdot v \\
&=\int \Delta(\rho\varphi^{-1})\frac{|v|^2}{2}+ \int c_0\dive{\left(2\mu\nabla^2\rho^{-1}\right)}\cdot v- \int c_0\rho(\Phi+\nabla\varphi^{-1})\cdot\nabla^2\rho^{-1}\cdot v\\
&\leq C\left(\norm{\Delta \rho}_{L^2}+\norm{\nabla\rho}_{L^4}\norm{\nabla\varphi}_{L^4}+\norm{\nabla\varphi}_{L^4}^2+\norm{\Delta\varphi}_{L^2}\right)\norm{v}_{L^4}^2\\
&\quad+C\norm{\nabla\mu}_{L^4}\norm{\nabla^2\rho^{-1}}_{L^2}\norm{v}_{L^4}+C\norm{\nabla\Delta\rho^{-1}}_{L^2}\norm{v}_{L^2}\\
&\quad+C\left(\norm{\Phi}_{L^4}+\norm{\nabla\varphi}_{L^4}\right)\norm{\nabla^2\rho^{-1}}_{L^2}\norm{v}_{L^4}\\
&\leq C_\varepsilon\left(\norm{\Delta \rho}^2_{L^2}+\norm{\nabla\rho}^4_{L^4}+\norm{\nabla\mu}_{L^4}^4+\norm{\Phi}_{L^4}^4+\norm{\nabla\varphi}_{L^4}^4+\norm{\Delta\varphi}^2_{L^2}\right)\norm{v}_{L^2}^2\\
&\quad +\varepsilon\left(\norm{\nabla v}_{L^2}^2+\norm{\nabla\Delta\rho^{-1}}_{L^2}^2\right)
\end{aligned}
\end{equation}
where we have used Poincar\'e's inequality for the last inequality. For the term $-\int \dive{[2\mu(\rho)D(v)]}\cdot v$, we directly use the results in Lemma \ref{Lemma3.4}, that is, for case \eqref{B'}, 
\begin{equation}\label{2233}
-\int \dive{[2\mu D(v)]}\cdot v=\int 2\mu |D(v)|^2\geq \nu\norm{\nabla v}_{L^2}^2,\text{ (use Lemma \ref{Lemma2.2})},
\end{equation}
while, for case \eqref{A'}, 
\begin{equation}\label{2234}
\begin{aligned}
-\int \dive{[2\mu D(v)]}\cdot v\geq\nu\norm{\nabla v}_{L^2}^2- \left(C_\varepsilon\norm{\nabla\mu}_{L^4}^4\norm{\sqrt\rho v}_{L^2}^2+C_\varepsilon\norm{\Delta\rho^{-1}}_{L^2}^2+\varepsilon\norm{\nabla v}_{L^2}^2\right).
\end{aligned}
\end{equation}
Thus, combining \eqref{2232}--\eqref{2234}, in both cases, 
\begin{equation}\label{equation5522}
\begin{aligned}
&\frac{d}{dt}\norm{\sqrt{\rho}v}_{L^2}^2+\nu\norm{\nabla v}_{L^2}^2\\
&\leq C\left(\norm{\Delta \rho}^2_{L^2}+\norm{\nabla\rho}^4_{L^4}+\norm{\nabla\mu}_{L^4}^4+\norm{\Phi}_{L^4}^4+\norm{\nabla\varphi}^4_{L^4}+\norm{\Delta\varphi}^2_{L^2}\right)\norm{v}_{L^2}^2\\
&\quad +C\norm{\nabla\Delta\rho^{-1}}_{L^2}^2
\end{aligned}
\end{equation}
which, using Gr$\mathrm{\ddot{o}}$nwall's inequality and \eqref{equation515}, gives
\begin{equation}\label{equation5523}
\norm{v}_{L^2}^2(t)+\int_0^t\norm{\nabla v}_{L^2}^2\,ds\leq C(\Omega,c,\alpha,\beta,\Phi,\varphi,\rho_0,v_0).
\end{equation}

Next, multiplying $\eqref{2231}_2$ by $v_t$ and integrating over $\Omega$, one has, using Lemma \ref{Lemma2.3},
\begin{equation}\label{equation5524}
\begin{aligned}
&\int \rho\abs{v_t}^2 - \int \dive[2\mu D(v)]\cdot v_t\\
&= - \int\rho (\Phi+\nabla\varphi^{-1})\cdot\nabla v\cdot v_t+ \int c_0\dive{\left(2\mu\nabla^2\rho^{-1}\right)}\cdot v_t \\
&\quad- \int c_0\rho(\Phi+\nabla\varphi^{-1})\cdot\nabla^2\rho^{-1}\cdot v_t\\
&\leq C\left(\norm{\Phi}_{L^4}+\norm{\nabla \varphi}_{L^4}\right)\norm{\nabla v}_{L^4}\norm{v_t}_{L^2}\\
&\quad+ C\left(\norm{\nabla\mu}_{L^4}\norm{\nabla^2\rho^{-1}}_{L^4}+\norm{\nabla\Delta\rho^{-1}}_{L^2}\right)\norm{v_t}_{L^2}\\
&\quad +C\left(\norm{\Phi}_{L^4}+\norm{\nabla\varphi}_{L^4}\right)\norm{\nabla^2\rho^{-1}}_{L^4}\norm{v_t}_{L^2}\\
&\leq C_\varepsilon\left(\norm{\Phi}^4_{L^4}+\norm{\nabla\varphi}^4_{L^4}\right)\norm{\nabla v}_{L^2}^2+C_\varepsilon\norm{\nabla\Delta\rho^{-1}}^2_{L^2}\\
&\quad+C_\varepsilon\left(\norm{\Phi}_{L^4}^4+\norm{\nabla\varphi}_{L^4}^4+\norm{\nabla\mu}_{L^4}^4\right)\norm{\Delta\rho^{-1}}^2_{L^2}+\varepsilon\left(\norm{v_t}_{L^2}^2+\norm{v}_{H^2}^2\right).
\end{aligned}
\end{equation}
For the term $- \int \dive[2\mu D(v)]\cdot v_t$, if $(\rho,v)$ satisfies the condition \eqref{A'}, we use the proof from \eqref{eq626} to \eqref{633},
\begin{equation}\label{5525}
\begin{aligned}
- \int \dive[2\mu D(v)]\cdot v_t&\geq \frac{d}{dt}\left(\cM_1(t)+\normf{\sqrt{\mu}\curle v}_{L^2}^2\right)+\frac{d}{dt}\cM_2(t)\\
&\quad-C_{\varepsilon}\left(\norm{\mu_t}_{H^1}^2+\norm{\nabla\mu}_{L^4}^4+1\right)\norm{v}_{H^1}^2\\
&\quad-C_{\varepsilon}\norm{\nabla\mu}_{L^4}^4\norm{\nabla\rho^{-1}}_{L^2}^2-C_\varepsilon\norm{\Delta\rho^{-1}}_{L^2}^2\\
&\quad-\varepsilon\left(\norm{v_t}_{L^2}^2+\norm{v}_{H^2}^2+\norm{\nabla\rho^{-1}_t}_{L^2}^2\right).
\end{aligned}
\end{equation}
Recalling that
\begin{equation*}
\cM_1(t)=\int_{\partial} \mu v\cdot B\cdot v,\quad\cM_2(t)=\int c_0\mu \nabla^\perp(v\cdot n^\perp)\cdot B\cdot \nabla \rho^{-1}.
\end{equation*}
For case \eqref{B'}, it is much easier, 
\begin{equation}
\begin{aligned}
- \int \dive[2\mu D(v)]\cdot v_t&=\int\mu\frac{d}{dt}|D(v)|^2\\
&=\frac{d}{dt}\int\mu|D(v)|^2-\int\mu_t|D(v)|^2\\
&\geq \frac{d}{dt}\int\mu|D(v)|^2-\left(C_\varepsilon\norm{\mu_t}_{L^2}^2\norm{\nabla v}_{L^2}^2+\varepsilon\norm{v}_{H^2}^2\right).
\end{aligned}
\end{equation}
Furthermore, from \eqref{equation5515}, we have
\begin{equation}\label{equation5527}
\begin{aligned}
\norm{v}_{H^2}^2+\norm{p}_{H^1}^2&\leq C\norm{v_t}_{L^2}^2+C\norm{\nabla\log\rho_t}_{L^2}^2+C\norm{\Delta\rho^{-1}}_{L^2}^2\\
&\quad+C\left(\norm{\nabla\mu}_{L^4}^4+\norm{\Phi}_{L^4}^4+\norm{\nabla \varphi}_{L^4}^4\right)\left(\norm{\nabla v}_{L^2}^2+\norm{\Delta\rho^{-1}}_{L^2}^2\right).
\end{aligned}
\end{equation}

Therefore, combining \eqref{equation5524}--\eqref{equation5527}, without loss of generality, one has 
\begin{equation}\label{equation5528}
\begin{aligned}
&\norm{v_t}_{L^2}^2+ \frac{d}{dt}\left(\cM_1(t)+\normf{\sqrt{\mu}\curle v}_{L^2}^2\right)+\frac{d}{dt}\cM_2(t)\\
&\leq C_\varepsilon\left(\norm{\Phi}^4_{L^4}+\norm{\nabla\varphi}^4_{L^4}+\norm{\mu_t}_{H^1}^2+\norm{\nabla\mu}_{L^4}^4+1\right)\norm{v}_{H^1}^2+C_\varepsilon\norm{\nabla\Delta\rho^{-1}}^2_{L^2}\\
&\quad+C_\varepsilon\left(\norm{\Phi}_{L^4}^4+\norm{\nabla\varphi}_{L^4}^4+\norm{\nabla\mu}_{L^4}^4\right)\norm{\Delta\rho^{-1}}^2_{L^2}+\varepsilon\left(\norm{\nabla\log\rho_t}_{L^2}^2+\norm{\nabla\rho^{-1}_t}_{L^2}^2\right).
\end{aligned}
\end{equation}

Finally, using Gr$\mathrm{\ddot{o}}$nwall's inequality, \eqref{equation515} and \eqref{equation5523}, we deduce that
\begin{equation}
\norm{\nabla v}_{L^2}^2(t)+\int_0^t\norm{v_t}_{L^2}^2\,ds\leq C(\Omega,c,\alpha,\beta,\Phi,\varphi,\rho_0,v_0).
\end{equation}
Therefore, we complete the proof of Lemma \ref{LP}.
\end{proof}

\subsection{Preliminary Reductions}\label{PR}

We claim that it is enough to prove the existence results for smooth initial data $(\rho_0,u_0)$ satisfying the compatiblity conditions \eqref{compat}. Once this is established, for general data $(\rho_0, u_0)$, we can build a sequence of smooth initial data $(\rho_0^n, u_0^n)$ such that it converges to $(\rho_0, u_0)$ in some appropriate functional spaces. Then, we can obtain a corresponding sequence of solutions $(\rho^n, v^n, \pi^n)$ (or $(\rho^n,u^n,\pi^n)$), which is uniformly bounded  with respect of $n$, satisfying the initial data $(\rho_0^n, v_0^n)$ (or $(\rho_0^n, u_0^n)$). We may use the weak convergence method and compactness reults to deduce that $(\rho^n, v^n, \pi^n)$ (or $(\rho^n,u^n,\pi^n)$) converges to $(\rho, v,\pi)$ (or $(\rho,u,\pi)$) in some functional spaces. As a result, $(\rho, v,\pi)$ (or $(\rho,u,\pi)$) will be the solution we expect, which proves our claim.

Now, we explain how we obtain such smooth data. We begin with $\alpha\leq \rho_0\leq\beta$, $u_0\in H^1_\omega$ (the case $u_0\in H^1_{nd}$ or $H^1_0$ can be done analogously). First, as we have said in Remark \ref{remark18}, we can derive that $\rho_0\in H^2$ from the compatiability condition \eqref{compat}
\begin{equation}
\begin{cases}
\Delta\rho_0^{-1}=c_0^{-1}\dive u_0,&x\in \Omega,\\
n\cdot\nabla\rho_0^{-1}=0,&x\in \partial\Omega.
\end{cases}
\end{equation}
Consequently, we get $v_0\in H^1$ by setting $v_0=u_0-c_0\nabla\rho_0^{-1}$. Then, we can construct a smooth sequence $(\hat\rho^n_0,\hat v^n_0)\in [C^\infty(\overline\Omega)]^4$ via flatten method and partition of unity such that 
\begin{equation}\label{eq5.1}
\begin{gathered}
\hat\rho^n_0\sconverge \rho_0\quad\mathrm{in}\,\,H^2,\quad \hat v_0^n\sconverge v_0\quad\mathrm{in}\,\,H^1.
\end{gathered}
\end{equation}
For details, see \cite{evans} Chapter 5.

However, the sequence $(\hat\rho^n_0,\hat v_0^n)$ may be failed to satisfy the boundary conditions and divergence-free condition, which means that we need further construction. First of all, we solve the following ellptic problem
\begin{equation*}
\begin{cases}
\Delta\rho^n_0=\Delta\hat\rho^n_0,&x\in \Omega,\\
n\cdot \nabla\rho^n_0=0,&x\in \partial\Omega,\\
(\rho_0^n)_\Omega=(\rho_0)_\Omega.
\end{cases}
\end{equation*}
Of course, for each $n\geq 1$, $\rho_0^n\in C^\infty(\overline\Omega)$ is unique and 
\begin{equation}
\norm{\nabla(\rho^n_0-\rho^m_0)}_{H^1}\leq C\norm{\nabla(\hat\rho^n_0-\hat\rho^m_0)}_{H^1}\to 0,\quad \text{as }n,m\to \infty.
\end{equation}
It follows from \eqref{eq5.1} that $\{\rho^n_0\}$ is a Cauchy sequence, and, thus, $\rho^n_0\sconverge \rho_0$ in $H^2$. Using Sobolev embedding theorem, $H^2\hookrightarrow C(\overline\Omega)$, we deduce that $\rho^n_0$ converges uniformly to $\rho_0$ and thus, without loss of generality, we may assume that $\rho^n_0\in [\alpha,\beta]$.

Next, to construct $v_0^n$, we borrow from the construction method in \cite{lions1}, Appendix A. More precisely, consider following Stokes problem of $v_0^n$
\begin{equation*}
\begin{cases}
-\Delta v_0^n +\nabla p^n= -\Delta \hat v_0^n,& x\in\Omega,\\
\dive v_0^n =0,&x\in \Omega,\\
v_0^n\cdot n =0, \,\,\curle v_0^n=-n^\perp\cdot\!B\cdot\![v_0^n+c_0\nabla(\rho_0^n)^{-1}],&x\in\partial\Omega,
\end{cases}
\end{equation*}
where $\int p^n =0$ and $\{\rho^n_0\}$ is the smooth sequence we just obtain. In view of  Lemma \ref{lemma2.5}, there exists a unique smooth solution $(v_0^n, p^n)\in [C^\infty(\overline\Omega)]^4$ such that
\begin{equation}\label{7.18}
\normf{v_0^n}_{H^1}+\normf{p^n}_{L^2}\leq C(\normf{\hat v_0^n}_{H^1}+\norm{\rho_0^n}_{H^2}).
\end{equation}
Thus, we obtain a Cauchy sequence 
\begin{equation*}
\normf{v_0^n-v_0^m}_{H^1}+\normf{p^n-p^m}_{L^2}\leq C(\normf{\hat v_0^n-\hat v_0^m}_{H^1}+\norm{\rho_0^n-\rho_0^m}_{H^2})\longrightarrow 0,\,\,\, \mathrm{as}\,\, n,m\to \infty,
\end{equation*}
because of \eqref{eq5.1} and the strong covergence of $\{\rho^n_0\}$. Without loss of generality, let
\begin{equation*}
v_0^n\sconverge \overline v_0\,\,\,\mathrm{in}\,\, H^1 \,\,\,\,\mathrm{and} \,\,\,\,p^n \sconverge p\,\,\mathrm{in} \,\, L^2.
\end{equation*}
Then, $V_0:=\overline v_0-v_0$ solves
\begin{equation*}
\begin{cases}
-\Delta V_0 +\nabla p=0,&x\in\Omega,\\
\dive V_0 =0,&x\in \Omega,\\
V_0\cdot n =0 ,\,\,\curle V_0=-n^\perp\cdot B\cdot V_0,&x\in\partial\Omega.
\end{cases}
\end{equation*}
It follows form the uniqueness of Stokes equations that $V_0\equiv 0$, that is, $\overline v_0=v_0$. Thus, we have found a smooth divergence-free sequence $v_0^n$, which satisfies the condtion \eqref{A'}, that converges strongly to $v_0$ in $H^1$.

If we treat the case \eqref{C}, we just turn back to $u^n_0$ by setting
\begin{equation*}
u^n_0:=v_0^n+c_0\nabla(\rho^n_0)^{-1}.
\end{equation*}
Then, it is easy to check that $u^n_0\in C^\infty(\overline\Omega)$, $u^n_0|_{\partial\Omega}=0$ and $(\rho^n_0,u^n_0)$ satisfies the compatiablity condition \eqref{compat}.

\subsection{Approximate System}\label{AS}

In order to get the existence for \eqref{equation1.1}, we first try to establish the smooth solutions for the following system: 
\begin{equation}\label{origin}
\begin{cases}
\rho_t+v_\eta\cdot\nabla\rho-c_0\dive\left(\rho_\eta^{-1}\nabla\rho\right)=0,\\
\rho u_t+\rho u_\eta\cdot\nabla u-\dive[2\mu_\epsilon D(u)]+\nabla\pi=0,\\
\dive u=c_0\Delta\rho^{-1},\quad \epsilon,\,\eta\in (0,1],\\
\rho|_{t=0}=\rho_0,\quad u|_{t=0}=u_0,\\
 \alpha\leq\rho_0\leq\beta,\,\,(\rho_0,u_0)\in [C^\infty(\overline\Omega)]^4\text{ satisfying }\eqref{compat},\\
u_0,\,(\rho,u)\text{ satisfies one of the bundary conditions }\eqref{A}-\eqref{C}.
\end{cases}\tag{A.P.}
\end{equation}
Let us give an explaination about the new elements in \eqref{origin}. We define
\begin{equation*}
u_\eta:= v_\eta+\rho_\eta,\quad \mu_\epsilon:=\mu(\rho_\epsilon),
\end{equation*}
and $\rho_\epsilon, \rho_\eta, v_\eta$ are constructed as we did in preceeding subsection, that is, $\rho_\epsilon, \rho_\eta, v_\eta\in C^\infty(\overline\Omega)$, $\dive v_\eta=0$, $\rho_\epsilon, \rho_\eta\in [\alpha,\beta]$ and $\rho_\epsilon, \rho_\eta, v_\eta$ satisfying corresponding boundary conditions.

For convenience, we collect some bounds here which will be used later. Obviously, we always have $\mu_\epsilon\in C^\infty(\overline \Omega)$ for every fixed $t$, $\epsilon$ and, for all $1\leq r\leq \infty$, $k\in \NN$,
\begin{equation}\label{base}
\begin{gathered}
\norm{\mu_\epsilon}_{L^r}\leq C(r,\Omega)\norm{\rho}_{L^\infty},\quad \norm{\nabla\mu_\epsilon}_{L^r}\leq C(r,\epsilon,\Omega)\norm{\rho}_{L^\infty},\\
 \norm{\nabla^k\rho_\eta}_{L^r}\leq C(k,r,\eta,\Omega)\norm{\rho}_{H^1},\quad \norm{\nabla^k\rho_\epsilon}_{L^r}\leq C(k,r,\epsilon,\Omega)\norm{\rho}_{H^1},\\
 \norm{\nabla^k v_\eta}_{L^r}\leq C(k,r,\eta,\Omega)\norm{v}_{L^2}.\\
\end{gathered}
\end{equation}
Also, we have the following uniform controls, for all $1\leq q<\infty$,
\begin{equation}\label{eq5.3.5}
\begin{aligned}
&\norm{v_\eta}_{W^{\ell,q}}\leq C\norm{v}_{W^{\ell,q}},\quad \ell=0,1,\\
&\norm{\rho_\eta}_{W^{\ell,q}}\leq C\norm{\rho}_{W^{\ell,q}}, \quad \ell=0,1,2.
\end{aligned}
\end{equation}

Our aim is proving the following theorem.
\begin{Theorem}\label{theoLP}
For every fixed $\epsilon,\,\eta\in (0,1]$, the problem \eqref{origin} admits an unique smooth solution on $\overline Q_{T_1}$ for some positive time $T_1$.
\end{Theorem}

Our proof is organized as follows. In the first part, we use iteration arguements and contraction mapping theorem to establish the unique smooth solution of \eqref{origin} for every fixed $\eta$ and $\epsilon$. Then, we recover the original system \eqref{equation1.1} by letting $\eta$, $\epsilon$ tend to $0$ in turn with help of the uniform estimates.

\subsubsection{Uniform Bounds}

\begin{equation}\label{regu}
\begin{cases}
\rho^n_t+v_\eta^{n-1}\cdot\nabla\rho^n-c_0\dive\left[(\rho^{n-1}_\eta)^{-1}\nabla\rho^n\right]=0,\\
\rho^{n} u^n_t+\rho^{n}u_\eta^{n-1}\cdot\nabla u^n-\dive[2\mu_\epsilon^nD(u^n)]+\nabla\pi^n=0,\\
\dive u^n=c_0\Delta(\rho^n)^{-1},
\end{cases}
\end{equation}
with the initial-boundary conditions
\begin{gather}
(\rho^n,u^n)(x,0)=(\rho_0,u_0),\quad \text{in }\Omega,\label{1}\\
n\cdot\nabla\rho^n=0,\quad u^n=0\quad\text{on }\partial\Omega\times(0,T).\label{2}	
\end{gather}
where we use the following notations
\begin{equation*}
\mu^n=\mu(\rho^n),\quad Q^n=\cB[c_0\Delta(\rho^n)^{-1}],\quad v^n=u^{n}+c_0\nabla(\rho^n)^{-1},\quad w^n=u^n-Q^n
\end{equation*}

To prove the existence for \eqref{regu}, we construct approximate solutions as follows. We first define $(\rho^0,u^0)=(C,0)$ and, then, assume that $(\rho^{n-1},u^{n-1})$ was defined for $n\geq 1$, let $(\rho^n,u^n,\pi^n)$ be the unique global strong solution to the problem \eqref{regu}.

To prove the uniform bounds for the approximate solutions, we introduce the function $\cH_N(t)$ defined by
$$\cH_N(t):=\begin{cases}
\max_{1\leq n\leq N}\left(1+\norm{\rho^n}_{H^2}^2+\norm{v^n}_{H^1}^2+\norm{\rho^n_t}_{L^2}^2\right),&\text{case }\eqref{A}\text{ or }\eqref{B}\\
\max_{1\leq n\leq N}\left(1+\norm{\rho^n}_{H^2}^2+\norm{u^n}_{H^1}^2+\norm{\rho^n_t}_{L^2}^2\right),&\text{case }\eqref{C}
\end{cases}$$

Observe that, in all cases, it follows from the maximal principle and energy estimates that
\begin{equation}
\alpha\leq \rho^n\leq \beta, \quad\sup_{t\in[0,T]}\norm{\rho^n}_{L^2}^2+\int_0^T\norm{\nabla\rho^n}_{L^2}^2\leq C,\text{ for all }T\in (0,\infty).
\end{equation}

Moreover, let $N$ be a fixed large number, we have
\begin{Lemma}\label{Lemma61}
There exists a positive constant $C$ depending on $\Omega$, $c_0$, $\alpha$, $\beta$ and $\rho_0$ such that
\begin{equation}
\norm{\nabla\rho^n}_{L^2}^2(t)+\int_0^t\norm{\Delta\rho^n}_{L^2}^2\,ds\leq C+C\int_0^t\cH_N(s)^3\,ds,
\end{equation}
for all $n$, $1\leq n\leq N$. 
\end{Lemma}
\begin{proof}
Let $n\geq 2$. From \eqref{554},
\begin{equation*}
\begin{aligned}
\frac{d}{dt}\norm{\nabla\rho^n}_{L^2}^2+\nu \norm{\Delta\rho^n}_{L^2}^2&\leq C\left(\norm{v_\eta^{n-1}}_{L^4}^4+\norm{\nabla\rho_\eta^{n-1}}_{L^4}^4\right)\norm{\nabla\rho^n}_{L^2}^2\\
&\leq C\left(\norm{u^{n-1}}_{L^4}^4+\norm{\nabla\rho^{n-1}}_{L^4}^4\right)\norm{\nabla\rho^n}_{L^2}^2\\
&\leq C\cH_N(t)^3.
\end{aligned}
\end{equation*}
Then, we integrate from $0$ to $t$ with respect of time and finish the proof of lemma.
\end{proof}

The next Lemma concerns with the uniform bounds for case \eqref{C}.
\begin{Lemma}\label{Lemma62}
Let $(\rho,u)$ satisfy the condition \eqref{C}. There exists a positive constant $C$ depending on $\Omega$, $c_0$, $\alpha$, $\beta$ and $u_0$ such that
\begin{equation}\label{equation5544}
\begin{aligned}
&\left[\norm{u^n}_{H^1}^2(t)+\norm{\nabla\rho^n}_{H^1}^2(t)
+\norm{\rho^n_t}_{L^2}^2(t)\right]\\
&\quad+\int_0^t\left(\norm{u^n}_{H^2}^2+\norm{\Delta\rho^n}_{H^1}^2+\norm{\rho^n_t}_{H^1}^2\right)\,ds\\
&\leq C+C\int_0^t\cH_N(s)^4\,ds,
\end{aligned}
\end{equation}
for all $n$, $1\leq n\leq N$. 
\end{Lemma}

\begin{proof}
For the higher regularity, we apply $-\nabla\Delta\rho^n\nabla$ on both sides of \eqref{regu}  and, then, integrate over $\Omega$ to derive the analogue of \eqref{eq4.1}
\begin{equation*}
\begin{aligned}
&\frac{d}{dt}\int \frac{1}{2}\abs{\Delta\rho^n}^2 + \int\frac{c_0}{\rho_\eta^{n-1}}\abs{\nabla\Delta\rho^n}^2\\
&= \int\nabla\Delta\rho^n\cdot\nabla v_\eta^{n-1}\cdot\nabla\rho^n+ \int v_\eta^{n-1}\cdot\nabla^2\rho^n\cdot\nabla\Delta\rho^n\\
&\quad-\int\frac{2c_0}{(\rho_\eta^{n-1})^3}\abs{\nabla\rho_\eta^{n-1}}^2\nabla\rho^n\cdot\nabla\Delta\rho^n+\int\frac{c_0}{(\rho_\eta^{n-1})^2}\nabla(\nabla\rho^n\cdot\nabla\rho_\eta^{n-1})\cdot\nabla\Delta\rho^n\\
&\quad + \int\frac{c_0}{\rho_\eta^{n-1}}\Delta\rho^n\nabla\rho_\eta^{n-1}\cdot\nabla\Delta\rho^n.
\end{aligned}
\end{equation*}
Then, applying Lemma \ref{Lemma2.3}, we can obtain the following inequality which is similar with \eqref{equa5.11},
\begin{equation*}
\begin{aligned}
&\frac{d}{dt}\norm{\Delta\rho^n}_{L^2}^2+\nu\norm{\nabla\Delta\rho^n}_{L^2}^2\\
&\leq C_\varepsilon\left(\norm{\nabla\rho^{n-1}}_{L^4}^4+\norm{\nabla\rho^{n}}_{L^4}^4+\norm{v^{n-1}}_{L^4}^4\right)\\
&\quad\times\left(\norm{\Delta\rho^{n-1}}_{L^2}^2+\norm{\Delta\rho^n}_{L^2}^2+\norm{\nabla v^{n-1}}_{L^2}^2\right)+\varepsilon\norm{\nabla v^{n-1}}_{H^1}^2\\
&\leq C_\varepsilon\left(\norm{\nabla\rho^{n-1}}_{L^4}^4+\norm{\nabla\rho^{n}}_{L^4}^4+\norm{u^{n-1}}_{L^4}^4\right)\\
&\quad\times\left(\norm{\Delta\rho^{n-1}}_{L^2}^2+\norm{\nabla\rho^{n-1}}_{L^4}^4+\norm{\Delta\rho^n}_{L^2}^2+\norm{\nabla u^{n-1}}_{L^2}^2\right)+\varepsilon\norm{\nabla v^{n-1}}_{H^1}^2\\
&\leq C_\varepsilon\cH_N(t)^4+\varepsilon\left(\norm{\Delta u^{n-1}}_{L^2}^2+\norm{\nabla\Delta\rho^{n-1}}_{L^2}^2+\norm{\nabla\rho^{n-1}}_{L^4}^4\norm{\Delta\rho^{n-1}}_{L^2}^2\right),
\end{aligned}
\end{equation*}
which gives
\begin{equation}\label{eq7.44}
\frac{d}{dt}\norm{\Delta\rho^n}_{L^2}^2+\nu\norm{\nabla\Delta\rho^n}_{L^2}^2\leq C_{\varepsilon_1}\cH_N(t)^4+\varepsilon_1\norm{\Delta u^{n-1}}_{L^2}^2
\end{equation}
Moreover, from \eqref{equation616},we also have
\begin{equation}\label{eq7.45}
\frac{d}{dt}\norm{\rho^n_t}_{L^2}^2 +\norm{\nabla\rho^n_t}_{L^2}^2
\leq C_{\varepsilon_2}\cH_N(t)^3+\varepsilon_2\norm{u^{n-1}_t}_{L^2}^2,
\end{equation}

To get the bounds for $u^n$, noticing that the mass equation can be written as
\begin{equation*}
\rho^n_t+\dive(\rho^{n}u_\eta^{n-1})=c_0\Delta(\rho^n/\rho^{n-1}_\eta).
\end{equation*}
Then, we follow the proof from \eqref{equation518} to get, for all $n\geq 2$
\begin{equation}\label{eq7.48}
\begin{aligned}
&\frac{d}{dt}\norm{\sqrt{\rho^{n}}u^n}_{L^2}^2+\nu\norm{\nabla u^n}_{L^2}^2\\
&\leq C_{\varepsilon_3}\norm{\nabla\rho^n}_{L^2}^2+C\left(\norm{u^{n-1}}_{L^4}^4+\norm{\nabla\rho^{n-1}}^4_{L^4}+\norm{\nabla \rho^n}_{L^4}^4+\norm{\Delta\rho^n}_{L^2}^2\right)\\
&\quad+C\left(\norm{\nabla\rho^n}^4_{L^4}+\norm{\nabla\rho^{n-1}}^4_{L^4}+\norm{\Delta\rho^n}_{L^2}^2+\norm{\Delta\rho^{n-1}}_{L^2}^2\right)\norm{u^n}_{L^2}^2+\varepsilon_3\norm{u^n_t}_{L^2}^2\\
&\leq C_{\varepsilon_3}\cH_N(t)^3+\varepsilon_3\norm{u^n_t}_{L^2}^2.
\end{aligned}
\end{equation}
Similarly, for $u^n_t$, it follows from \eqref{equation520} that
\begin{equation}\label{7746}
\begin{aligned}
&\norm{u^n_t}_{L^2}^2+\frac{d}{dt}\norm{\sqrt{\mu_\epsilon^n}D(u^n)}_{L^2}^2\\
&\leq C\left(\norm{u^{n-1}}_{L^4}^2+\norm{\rho^n_t}_{L^2}+\norm{\nabla\rho^{n}}_{L^4}^2\right)\norm{\nabla u^n}^2_{L^4}+C\norm{Q^n_t}_{L^2}^2+C\norm{Q^n_t}_{L^2}\norm{\Delta u^n}_{L^2}\\
&\leq C_{\varepsilon_4}\left(\cH_N(t)^3+\norm{\nabla\rho^n_t}_{L^2}^2\right)+\varepsilon_4\norm{\Delta u^n}_{L^2}^2.
\end{aligned}
\end{equation}
Combining \eqref{eq7.48}--\eqref{7746}, we obtain
\begin{equation}\label{eq7.51}
\begin{aligned}
&\left(\norm{u^n_t}_{L^2}^2+\norm{\nabla u^n}_{L^2}^2\right)+\frac{d}{dt}\left(\norm{\sqrt{\mu_\epsilon^n}D(u^n)}_{L^2}^2+\norm{\sqrt{\rho^n}u^n}_{L^2}^2\right)\\
&\leq C_{\varepsilon_5}\left(\cH_N(t)^3+\norm{\nabla\rho^n_t}_{L^2}^2\right)+\varepsilon_5\norm{\Delta u^n}_{L^2}^2.
\end{aligned}
\end{equation}

Moreover, it follows from \eqref{0517} that
\begin{equation}\label{eq7.53}
\norm{\Delta u^n}^2_{L^2}+\norm{\pi^n}^2_{H^1}\leq C\left(\cH_N(t)^4+\norm{\nabla\Delta\rho^n}_{L^2}^2+\norm{u^n_t}_{L^2}^2\right).
\end{equation}
Plugging this into \eqref{eq7.44} and \eqref{eq7.51} and, then, combining two of them, we have
\begin{equation*}
\begin{aligned}
&\frac{d}{dt}\left(\norm{\sqrt{\mu_\epsilon^n}D(u^n)}_{L^2}^2+\norm{\Delta\rho^n}_{L^2}^2\right)+\nu\left(\norm{\nabla\Delta \rho^n}_{L^2}^2+\norm{u^n_t}_{L^2}^2\right)\\
&\leq C\norm{\nabla\rho_t^n}^2_{L^2}+C_{\varepsilon_4}\cH_N(t)^4+\varepsilon_4\left(\norm{\nabla\Delta \rho^{n-1}}_{L^2}^2+\norm{u^{n-1}_t}_{L^2}^2\right).
\end{aligned}
\end{equation*}
Alonging with \eqref{eq7.45} and choosing $\varepsilon_2$ small enough and $\varepsilon_4=1/2$, one has
\begin{equation}\label{66652}
\begin{aligned}
&\frac{d}{dt}\left(\norm{\sqrt{\mu_\epsilon^n}D(u^n)}_{L^2}^2+\norm{\Delta\rho^n}_{L^2}^2+2C\norm{\rho^n_t}_{L^2}^2\right)+\left(\norm{\nabla\Delta \rho^n}_{L^2}^2+\norm{u^n_t}_{L^2}^2+\norm{\nabla\rho^n_t}_{L^2}^2\right)\\
&\leq C\cH_N(t)^4+\frac{1}{2}\left(\norm{\nabla\Delta \rho^{n-1}}_{L^2}^2+\norm{u^{n-1}_t}_{L^2}^2\right).
\end{aligned}
\end{equation}
For simplicity, we denote by the above 
\begin{equation*}
\begin{aligned}
\frac{d}{dt}\cP^n(t)+\left(\norm{\nabla\Delta \rho^n}_{L^2}^2+\norm{u^n_t}_{L^2}^2\right)\leq C\cH_N(t)^4+\frac{1}{2}\left(\norm{\nabla\Delta \rho^{n-1}}_{L^2}^2+\norm{u^{n-1}_t}_{L^2}^2\right).
\end{aligned}
\end{equation*}
Then, integrating over $[0,t]$, one has
\begin{equation*}
\begin{aligned}
&\cP^n(t)+\int_0^t\left(\norm{\nabla\Delta \rho^{n}}_{L^2}^2+\norm{u^{n}_t}_{L^2}^2\right)\,ds\\
&\leq C\left(1+\int_0^t\cH_N(s)^4\,ds\right)+\frac{1}{2}\int_0^t\left(\norm{\nabla\Delta \rho^{n-1}}_{L^2}^2+\norm{u^{n-1}_t}_{L^2}^2\right)\,ds.
\end{aligned}
\end{equation*}
Using this recursive inequality for $\int_0^t\left(\norm{\nabla\Delta \rho^{n}}_{L^2}^2+\norm{u^{n}_t}_{L^2}^2\right)\,ds$, we obtain
\begin{equation*}
\begin{aligned}
\int_0^t\left(\norm{\nabla\Delta \rho^{n}}_{L^2}^2+\norm{u^{n}_t}_{L^2}^2\right)\,ds& \leq \left(1+\frac{1}{2}+\cdots+\frac{1}{2^n}\right)C\left(1+\int_0^t\cH_N(s)^4\,ds\right)\\
&\leq 2C\left(1+\int_0^t\cH_N(s)^4\,ds\right)
\end{aligned}	
\end{equation*}
and hence, turning back to \eqref{66652}, we get
\begin{equation*}
\begin{aligned}
\cP^n(t)+\int_0^t\left(\norm{\nabla\Delta \rho^n}_{L^2}^2+\norm{u^n_t}_{L^2}^2+\norm{\nabla\rho^n_t}_{L^2}^2\right)\,ds\leq C\left(1+\int_0^t\cH_N(s)^4\,ds\right),
\end{aligned}
\end{equation*}
for all $2\leq n\leq N$ and, thus, for all $1\leq n\leq N$. Finally, using \eqref{eq7.53}, we get the bounds for $\Delta u^n$ and $\pi^n$ which concludes the lemma.
\end{proof}

Next, we give the uniform estimates for condition \eqref{A} or \eqref{B}.
\begin{Lemma}\label{Lemma56}
Let $(\rho,v)$ satisfy the condition \eqref{A'} or \eqref{B'}. There exists a positive constant $C$ depending on $\Omega$, $c_0$, $\alpha$, $\beta$, $\rho_0$ and $v_0$ such that
\begin{equation}\label{equatio}
\begin{aligned}
&\left[\norm{v^n}_{H^1}^2(t)+\norm{\nabla\rho^n}_{H^1}^2(t)
+\norm{\rho^n_t}_{L^2}^2(t)\right]\\
&\quad+\int_0^t\left(\norm{v^n}_{H^2}^2+\norm{\Delta\rho^n}_{H^1}^2+\norm{\rho^n_t}_{H^1}^2\right)\,ds\\
&\leq C+C\int_0^t\cH_N(s)^8\,ds,
\end{aligned}
\end{equation}
for all $n$, $1\leq n\leq N$. 
\end{Lemma}
\begin{proof}
We still only give the proof for case \eqref{A'}. From \eqref{equation5522} and the proof of Lemma \ref{lemma6.5}, we get
\begin{equation}\label{equation5552}
\begin{aligned}
&\frac{d}{dt}\norm{\sqrt{\rho^n}v^n}_{L^2}^2+\nu\norm{\nabla v^n}_{L^2}^2\\
&\leq C\left(\norm{\Delta \rho^n}^2_{L^2}+\norm{\nabla\rho^n}^4_{L^4}+\norm{\nabla\mu_\epsilon^n}_{L^4}^4+\norm{v_\eta^{n-1}}_{L^4}^4+\norm{\nabla\rho^{n-1}_\eta}^4_{L^4}+\norm{\Delta\rho^{n-1}_\eta}^2_{L^2}\right)\norm{v^n}_{L^2}^2\\
&\quad +C\norm{\nabla\Delta(\rho^n)^{-1}}_{L^2}^2\\
&\leq C\cH_N(t)^3+C\norm{\nabla\Delta\rho^n}_{L^2}^2
\end{aligned}
\end{equation}
and
\begin{equation}\label{equation5553}
\begin{aligned}
&\norm{v^n_t}_{L^2}^2+ \frac{d}{dt}\left(\cM^n_1(t)+\norm{\sqrt\mu\curle v}_{L^2}^2\right)+\frac{d}{dt}\cM^n_2(t)\\
&\leq C\cH^N(t)^3+C\cH_N(t)^2\norm{\nabla v^n}_{L^4}^2\\
&\leq C\cH^N(t)^3+C\cH_N(t)^2\norm{\nabla v^n}_{L^2}\norm{v}_{H^2},
\end{aligned}
\end{equation}
while, for $\norm{v}_{H^2}$, we apply Lemma \ref{lemma8.4} for
\begin{equation}
\begin{aligned}
&-\dive[2\mu_\epsilon^nD(v^n)]+\nabla\pi^n\\
&=-\rho^{n} v^n_t+c_0\nabla\log\rho^n_t-\rho^{n}[v_\eta^{n-1}+c_0\nabla(\rho^{n-1})^{-1}]\cdot\nabla [v_\eta^{n}+c_0\nabla(\rho^{n})^{-1}]\\
&\quad+\dive[2\mu_\epsilon^n\nabla^2(\rho^n)^{-1}]:=F^n
\end{aligned}
\end{equation}
to obtain
\begin{equation*}
\begin{aligned}
\norm{v}_{H^2}+\norm{\pi}_{H^1}&\leq C\left(\norm{\nabla\mu^n_\epsilon}_{L^4}^2+1\right)\left(\norm{F^n}_{L^2}+\norm{\Delta(\rho^n)^{-1}}_{L^2}\right)+\norm{\nabla\mu^n_\epsilon}_{L^4}^2\norm{\nabla v^n}_{L^2}\\
&\leq C\cH_N(t)\left(\norm{v^n_t}_{L^2}+\norm{\nabla\log\rho^n_t}_{L^2}+\norm{\nabla\Delta(\rho^n)^{-1}}_{L^2}\right)\\
&\quad+C\cH_N(t)^\frac{3}{2}(\norm{\nabla v^n}_{L^4}+\norm{\Delta(\rho^{n})^{-1}}_{L^4})+C\cH_N(t)^\frac{3}{2}\\
&\leq C\cH_N(t)\left(\norm{v^n_t}_{L^2}+\norm{\nabla\rho^n_t}_{L^2}+\norm{\nabla\Delta\rho^n}_{L^2}+\cH_N(t)^\frac{3}{2}\right)\\
&\quad+C\left(\cH_N(t)^4+\cH_N(t)^\frac{3}{2}\right)+\frac{1}{2}\norm{v^n}_{H^2},
\end{aligned}
\end{equation*}
which leads to
\begin{equation}\label{equation5555}
\begin{aligned}
\norm{v}_{H^2}+\norm{\pi}_{H^1}\leq C\cH_N(t)\left(\norm{v^n_t}_{L^2}+\norm{\nabla\rho^n_t}_{L^2}+\norm{\nabla\Delta\rho^n}_{L^2}\right)+C\cH_N(t)^4.
\end{aligned}
\end{equation}
Substituting this into \eqref{equation5553}, we have
\begin{equation}\label{equation5556}
\begin{aligned}
&\norm{v^n_t}_{L^2}^2+ \frac{d}{dt}\left(\cM^n_1(t)+\norm{\sqrt\mu\curle v}_{L^2}^2\right)+\frac{d}{dt}\cM^n_2(t)\\
&\leq C\cH^N(t)^3+C\cH_N(t)^4\left(\norm{v^n_t}_{L^2}+\norm{\nabla\rho^n_t}_{L^2}+\norm{\nabla\Delta\rho^n}_{L^2}\right)+C\cH_N(t)^6\\
&\leq C_{\varepsilon_1}\cH_N^8(t)+\varepsilon_1\left(\norm{v^n_t}_{L^2}^2+\norm{\nabla\rho^n_t}^2_{L^2}+\norm{\nabla\Delta\rho^n}^2_{L^2}\right).
\end{aligned}
\end{equation}

On the other hand, following the proofs of \eqref{eq7.44}--\eqref{eq7.45}, one has
\begin{equation}\label{equation5557}
\begin{aligned}
&\frac{d}{dt}\norm{\Delta\rho^n}_{L^2}^2+\nu\norm{\nabla\Delta\rho^n}_{L^2}^2\\
&\leq C\cH_N(t)^4+C\norm{\nabla\rho^n}_{L^4}^2\norm{\nabla v^{n-1}}_{L^4}^2\\
&\leq C\cH_N(t)^4+C\norm{\nabla\rho^n}_{L^4}^2\norm{\nabla v^{n-1}}_{L^2}\norm{v^{n-1}}_{H^2}\\
&\leq C_{\varepsilon_2}\cH_N(t)^6+\varepsilon_2\left(\norm{v^{n-1}_t}_{L^2}^2+\norm{\nabla\rho^{n-1}_t}^2_{L^2}+\norm{\nabla\Delta\rho^{n-1}}^2_{L^2}\right)
\end{aligned}
\end{equation}
and
\begin{equation}\label{equation5558}
\frac{d}{dt}\norm{\rho^n_t}_{L^2}^2 +\norm{\nabla\rho^n_t}_{L^2}^2
\leq C_{\varepsilon_3}\cH_N(t)^3+\varepsilon_3\left(\norm{v^{n-1}_t}_{L^2}^2+\norm{\nabla\rho^{n-1}_t}_{L^2}^2\right),
\end{equation}

Therefore, combining Lemma \ref{Lemma61}, \eqref{equation5552} and \eqref{equation5556}--\eqref{equation5558} and, then, using the same recursive arguements at the end of the proof of Lemma \ref{Lemma62}, we can obtain the desire bound \eqref{equatio}. 
\end{proof}

\begin{Remark}\label{remark56}
It follows from \eqref{eq7.53} in Lemma \ref{Lemma62} that the constant $C$ in \eqref{equation5544} depends on $\epsilon\in(0,1]$, which indicates that we can only obtain the local existence for the case \eqref{C} with $\mu=\mu_\epsilon$, in particular, $\mu$ being a positive constant. However, from the proof of Lemma \ref{Lemma56}, since we used Lemma \ref{lemma8.4} to get the estimate \eqref{equation5555}, the constant $C$ in \eqref{equatio} is independent with $\epsilon$ and that is why we could extend the local existence for cases \eqref{A} and \eqref{B} to general viscosity coefficient $\mu(\rho)$.
\end{Remark}

In conclusion, we have the bounds
\begin{equation}
\cH_N(t)\leq C\left(1+\int_0^t\cH_N(s)^q\,ds\right), \text{ for some }q>1.
\end{equation}
Thanks to this integral inequality, we can easily show that there exists a time $T_1\in (0,T)$ depending only on $\Omega$, $c_0$, $\alpha$, $\beta$ and $u_0$ such that
\begin{equation}
\sup_{t\in[0,T_1]}\cH_N(t)\leq C_0,
\end{equation}
for some $C_0$ independing with $N$. Therefore, we obtain the bounds, for all $n\geq 1$,
\begin{equation}\label{eee559}
\begin{aligned}
\sup_{t\in[0,T^1]}\left(\norm{u^n}_{H^1}^2+\norm{\rho^n}_{H^2}^2+\norm{\rho^n_t}_{L^2}^2\right)+\int_0^{T_1}\left(\norm{u^n}_{H^2}^2+\norm{\rho^n}_{H^3}^2+\norm{\rho^n_t}_{H^1}^2\right)\,ds\leq C_0.
\end{aligned}
\end{equation}

\subsubsection{Convergence}\label{5.3.2}

We next show that the whole sequence $(\rho^n,u^n)$ converges to a solution to \eqref{origin} in a sufficiently strong sense. Let 
\begin{equation*}
\sigma^{n+1}:=\rho^{n+1}-\rho^n,\quad a^{n+1}:=u^{n+1}-u^n,\quad b^{n+1}:=v^{n+1}-v^n,\quad c^{n+1}:=Q^{n+1}-Q^n
\end{equation*}
and
\begin{equation*}
\cY^n(t):=
\begin{cases}
\norm{a^n}_{H^1}^2+\norm{\sigma^n}_{H^2}^2+\norm{\sigma^n_t}_{L^2}^2,&\text{case }\eqref{C}\\
\norm{b^n}_{H^1}^2+\norm{\sigma^n}_{H^2}^2+\norm{\sigma^n_t}_{L^2}^2,&\text{case }\eqref{A}\text{ or }\eqref{B}
\end{cases}
\end{equation*}
\begin{equation*}
\cZ^n(t):=
\begin{cases}
\norm{a^n_t}_{L^2}^2+\norm{\sigma^n}_{H^3}^2+\norm{\sigma^n_t}_{H^1}^2,&\text{case }\eqref{C}\\
\norm{b^n_t}_{L^2}^2+\norm{\sigma^n}_{H^3}^2+\norm{\sigma^n_t}_{H^1}^2, &\text{case }\eqref{A}\text{ or }\eqref{B}
\end{cases}
\end{equation*}

In addition, we always let $\cI^n(t)$ and $\cB^n(t)$ be generic functions associated with the bounds \eqref{eee559} such that
\begin{equation*}
\int_0^{T_1}\cI^n(t)\,dt+\sup_{t\in[0,T_1]}\cB^n(t)\leq C_0,
\end{equation*}
where $C_0$ is the constant as in \eqref{eee559}.

\noindent $\underline{\mathbf{Case\,\,\eqref{C}}}$:\\
\quad\\
It follows from the linearized mass equation that
\begin{equation}\label{eq6.65}
\begin{aligned}
&\sigma^{n+1}_t+v_\eta^{n}\cdot\nabla\sigma^{n+1}-c_0\dive\left(\frac{1}{\rho_\eta^{n}}\nabla\sigma^{n+1}\right)\\
&=-b_\eta^n\cdot\nabla\rho^n-c_0\dive\left(\frac{\sigma_\eta^n}{\rho_\eta^{n-1}\rho_\eta^{n}}\nabla\rho^n\right):=G^n
\end{aligned}
\end{equation}
where
\begin{equation*}
\begin{aligned}
\norm{G^n}_{H^{-1}}&\leq C\norm{\nabla\rho^n}_{L^4}\left(\norm{b_\eta^n}_{L^4}+\norm{\sigma_\eta^n}_{L^4}\right)\leq C\norm{\rho^n}_{H^2}\left(\norm{a^n}_{H^1}+\norm{\sigma^n}_{H^1}\right)\\
\norm{G^n}_{L^2}&\leq \norm{\nabla\rho^n}_{L^4}\norm{b_\eta^n}_{L^4}+C\norm{\Delta\rho^n}_{L^2}\norm{\sigma_\eta^n}_{L^\infty}+C\norm{\sigma_\eta^n}_{W^{1,4}}\norm{\nabla\rho^n}_{L^4}\\
&\leq C\norm{\rho^n}_{H^2}\left(\norm{a^n}_{H^1}+\norm{\sigma^n}_{H^2}\right)
\end{aligned}
\end{equation*}
\begin{equation*}
\begin{aligned}
\norm{\nabla G^n}_{L^2}&\leq \norm{\nabla\rho^n}_{L^4}\norm{\nabla b_\eta^n}_{L^4}+\norm{\Delta\rho^n}_{L^2}\norm{b_\eta^n}_{L^\infty}+C\norm{\nabla\rho^n}_{L^4}\norm{\Delta\sigma^n_\eta}_{L^4}\\
&\quad+C\norm{\Delta\rho^n}_{L^2}\norm{\nabla\sigma^n_\eta}_{L^\infty}+C\norm{(|\nabla\rho_\eta^n|+|\nabla\rho_\eta^{n-1}|)|\nabla\rho^n|}_{L^4}\norm{\nabla\sigma_\eta^n}_{L^4}\\
&\quad+C\norm{\left(|\nabla\rho^n_\eta|^2+|\nabla\rho^{n-1}_\eta|^2+|\nabla^2\rho^{n-1}_\eta|+|\nabla^2\rho^n_\eta|\right)|\nabla\rho^n|}_{L^2}\norm{\sigma_\eta^n}_{L^\infty}\\
&\quad +C\norm{(|\nabla\rho^n_\eta|+|\nabla\rho^{n-1}_\eta|)|\nabla^2\rho^n|}_{L^2}\norm{\sigma_\eta^n}_{L^\infty}+C\norm{\nabla\Delta\rho^n}_{L^2}\norm{\sigma^n_\eta}_{L^\infty}\\
&\leq C_{\eta}\left(\norm{\rho^n}_{H^2}+\norm{\rho^n}_{W^{1,4}}^2\right)\left(\norm{a^{n}}_{L^2}+\norm{\sigma^n}_{H^2}\right)\\
&\quad+C_\eta\norm{\nabla\Delta\rho^n}_{L^2}\norm{\sigma^n}_{H^1}.\quad\quad\quad\text{use }\eqref{base}
\end{aligned}
\end{equation*}
Then, using the simplified notations, the above bounds can be written as follows
\begin{equation}\label{ee658}
\begin{gathered}
\norm{G^n}_{H^{-1}}+\norm{G^n}_{L^2}\leq C\cB^n(t)\cY^n(t),\\
\norm{\nabla G^n}_{L^2}\leq C_\eta\cB^n(t)\cY^n(t)+C_\eta\sqrt{\cI^n(t)}\norm{\sigma^n}_{H^1}.
\end{gathered}
\end{equation}

Next, we are going to establish the bounds for $\sigma^{n+1}$ and $a^{n+1}$. Multiplying \eqref{eq6.65} by $\sigma^{n+1}$ and integrating over $\Omega$, we obtain
\begin{equation*}
\begin{aligned}
\frac{d}{dt}\norm{\sigma^{n+1}}_{L^2}^2+\nu\norm{\nabla \sigma^{n+1}}_{L^2}^2\leq C\norm{G^n}_{H^{-1}}\norm{\sigma^{n+1}}_{H^1},
\end{aligned}
\end{equation*}
then, using \eqref{ee658}, we deduce that
\begin{equation}\label{eee658}
\begin{aligned}
\frac{d}{dt}\norm{\sigma^{n+1}}_{L^2}^2+\nu\norm{\nabla \sigma^{n+1}}_{L^2}^2\leq C\norm{G^n}_{H^{-1}}^2\leq C\cB^n(t)\cY^n(t).
\end{aligned}
\end{equation}
Similar with \eqref{554}, multiplying $-\Delta\sigma^{n+1}$ on both sides of \eqref{eq6.65} and integrating over $\Omega$, one has
\begin{equation}\label{ee659}
\begin{aligned}
\frac{d}{dt}\norm{\nabla\sigma^{n+1}}_{L^2}^2+\nu\norm{\Delta\sigma^{n+1}}_{L^2}^2&\leq  C\left(\norm{v_\eta^n}_{L^4}^4+\norm{\nabla\rho_\eta^n}_{L^4}^4\right)\norm{\nabla\sigma^{n+1}}_{L^2}^2+C\norm{G^n}^2_{L^2}\\
&\leq C\cI^n(t)\norm{\nabla\sigma^{n+1}}_{L^2}^2+C\cB^n(t)\cY^n(t),
\end{aligned}
\end{equation}
where we have used \eqref{ee658} for the last inequality.
If we integrate \eqref{ee659} between $[0,t]$, $t<T_1$, we have
\begin{equation}\label{ee660}
\begin{aligned}
\norm{\nabla\sigma^{n+1}}_{L^2}^2(t)+\int_0^t\norm{\Delta\sigma^{n+1}}_{L^2}^2\,ds\leq C\int_0^t\cY^n(s)\,ds.
\end{aligned}
\end{equation}

For $\rho^n$ satisfying the Neumann condition, we copy the proof from \eqref{557} by applying $-\nabla\Delta\sigma^{n+1}\nabla$ on \eqref{eq6.65}, integrating over $\Omega$ and using \eqref{ee658}, that is
\begin{equation*}
\begin{aligned}
&\frac{d}{dt}\norm{\Delta\sigma^{n+1}}_{L^2}^2+\nu\norm{\nabla\Delta\sigma^{n+1}}_{L^2}^2\\
&\leq C\left(\norm{v_\eta^n}_{W^{1,4}}^4+\norm{\rho_\eta^n}_{W^{2,4}}^4+\norm{\nabla\rho_\eta^n}_{L^8}^8\right)\norm{\Delta\sigma^{n+1}}_{L^2}^2+C\norm{\nabla G^n}^2_{L^2}\\
&\leq C\cI^n(t)\norm{\Delta\sigma^{n+1}}_{L^2}^2+C\cB^n(t)\cY^n(t)+C\cI^n(t)\norm{\sigma^n}_{H^1}^2.
\end{aligned}
\end{equation*}
This, alonging with \eqref{ee660}, implies that
\begin{equation}\label{ee661}
\begin{aligned}
&\frac{d}{dt}\norm{\Delta\sigma^{n+1}}_{L^2}^2+\nu\norm{\nabla\Delta\sigma^{n+1}}_{L^2}^2\\
&\leq C\cI^n(t)\norm{\Delta\sigma^{n+1}}_{L^2}^2+C\cB^n(t)\cY^n(t)+C\cI^n(t)\int_0^t\cY^n(s)\,ds.
\end{aligned}
\end{equation}
For $\sigma^{n+1}_t$, we multiply $\sigma^{n+1}_t$ on the both sides of \eqref{eq6.65} and integrating over $\Omega$, it follows analogously from \eqref{5510} that
\begin{equation}\label{ee662}
\begin{aligned}
&\frac{d}{dt}\norm{\sigma^{n+1}_t}_{L^2}^2+\nu\norm{\nabla\sigma^{n+1}_t}_{L^2}^2\\
&\leq C_\varepsilon\left(\norm{u^n_{t}}_{L^2}^2+\norm{\rho^n_{t}}_{L^4}^4+\norm{\nabla\rho^n}_{L^4}^4+\norm{\nabla\rho^n_{t}}_{L^2}^2\right)\left(\norm{\sigma^{n+1}_t}_{L^2}^2+\norm{\nabla\sigma^{n+1}}_{L^2}^2\right)\\
&\quad+\varepsilon\left(\norm{\Delta\sigma^{n+1}}_{L^2}^2+\norm{a^n_t}_{L^2}^2\right)+C\norm{\Delta\rho^n}^2_{L^2}\norm{\sigma^n}^2_{H^1}\\
&\leq C_{\varepsilon}\cI^n(t)\left(\norm{\sigma^{n+1}_t}_{L^2}^2+\norm{\sigma^{n+1}}_{H^1}^2\right)+C\cB^n(t)\cY^n(t)+\varepsilon\left(\norm{\Delta\sigma^{n+1}}_{L^2}^2+\norm{a_t^n}^2_{L^2}\right)
\end{aligned}
\end{equation}

On the other hand, we differeniate the equations between those of $u^{n+1}$ and $u^n$ to get
\begin{equation}\label{ee663}
\begin{aligned}
&\rho^{n+1}a^{n+1}_t+\rho^{n+1}u_\eta^{n}\cdot\nabla a^{n+1}-\dive[2\mu_\eta^{n+1}D(a^{n+1})]+\nabla (\pi^{n+1}-\pi^n)\\
&=-\sigma^{n+1}(u^n_t+u^n_\eta\cdot \nabla u^n)-\rho^n a^n_\eta\cdot\nabla u^n+\dive[2(\mu_\epsilon^{n+1}-\mu_\epsilon^n)D(u^n)]:=K^n,
\end{aligned}
\end{equation}
where
\begin{equation*}
\begin{aligned}
\norm{K^n}_{H^{-1}}&\leq \norm{u^n_t+u^n_\eta\cdot \nabla u^n}_{L^2}\norm{\sigma^{n+1}}_{L^\infty}+\norm{\nabla u^n}_{L^2}\norm{a^n_\eta}_{L^\infty}+\norm{\nabla u^n}_{L^2}\norm{\sigma_\eta^{n+1}}_{L^\infty}\\
\norm{K^n}_{L^2}&\leq \norm{u^n_t+u^n_\eta\cdot \nabla u^n}_{L^2}\norm{\sigma^{n+1}}_{L^\infty}+\norm{\nabla u^n}_{L^2}\norm{a^n_\eta}_{L^\infty}+\norm{\Delta u^n}_{L^2}\norm{\sigma_\eta^{n+1}}_{L^\infty}\\
&\quad +C\norm{\nabla u^n}_{L^4}\norm{\nabla\rho^n_\eta}_{L^4}\norm{\sigma^{n+1}_\eta}_{L^\infty}+C\norm{\nabla u^n}_{L^4}\norm{\nabla \sigma_\eta^{n+1}}_{L^4},
\end{aligned}
\end{equation*}
that is,
\begin{equation}\label{0570}
\norm{K^n}_{H^{-1}}+\norm{K^n}_{L^2}\leq C\sqrt{\cI^n(t)}\norm{\sigma^{n+1}}_{H^2}+C_\eta\cB^n(t)\norm{a^n}_{L^2},
\end{equation}

Next, following the proof of \eqref{equation518}, we multilpy $a^{n+1}-c^{n+1}$ on both sides of \eqref{ee663}, integrate over $\Omega$ and use \eqref{0570} to obtain 
\begin{equation}\label{ee665}
\begin{aligned}
&\frac{d}{dt}\norm{\sqrt{\rho^{n+1}}a^{n+1}}_{L^2}^2+\nu\norm{\nabla a^{n+1}}_{L^2}^2\\
&\leq C\left(\norm{\Delta\rho^{n+1}}_{L^2}^2+\norm{\Delta\rho^{n}}_{L^2}^2+\norm{\nabla\rho^{n+1}}_{L^4}^4+\norm{\nabla\rho^{n}}_{L^4}^4\right)\norm{a^{n+1}}_{L^2}^2\\
&\quad +C\norm{u^{n}_\eta}_{L^4}^2\norm{c^{n+1}}_{L^4}^2 +C_\varepsilon\norm{c^{n+1}}_{H^1}^2+\varepsilon\norm{a^{n+1}_t}_{L^2}^2+C\norm{K^n}_{H^{-1}}^2\\
&\leq C_\varepsilon\cI^n(t)\left(\norm{a^{n+1}}_{L^2}^2+\norm{\sigma^{n+1}}^2_{H^2}\right)+C\cB^n(t)\norm{a^n}^2_{L^2}+C\norm{\Delta\sigma^{n+1}}^2_{L^2}+\varepsilon\norm{a^{n+1}_t}_{L^2}^2.
\end{aligned}
\end{equation}
Here, for the last inequality, we have used Lemma \ref{Lemma2.3} and 
\begin{equation*}
\begin{aligned}
\norm{c^{n+1}}_{L^p}&\leq C\norm{\sigma^{n+1}}_{W^{1,p}},\\
\norm{c^{n+1}}_{H^1}&\leq C\left(\norm{\nabla\rho^{n+1}}_{L^4}^2+\norm{\nabla\rho^{n}}_{L^4}^2+\norm{\nabla\rho^{n}}_{L^8}^4+\norm{\Delta\rho^{n}}_{L^4}^2\right)\norm{\sigma^{n+1}}_{H^1}\\
&\quad+C\norm{\Delta\sigma^{n+1}}_{L^2}\\
&\leq C\sqrt{\cI^n(t)}\norm{\sigma^{n+1}}_{H^1}+C\norm{\Delta\sigma^{n+1}}_{L^2}
\end{aligned}
\end{equation*}
To get the higher bound for $a^n$, multiplying \eqref{ee663} by $a^{n+1}_t-c^{n+1}_t$, it follows from \eqref{equation520} that
\begin{equation}\label{ee667}
\begin{aligned}
&\frac{d}{dt}\norm{\nabla a^{n+1}}_{L^2}^2+\nu\norm{a^{n+1}_t}_{L^2}^2\\
&\leq C\left(\norm{u^n_\eta}_{L^4}^2+\norm{\rho^{n+1}_{\eta,t}}_{L^2}+\norm{\nabla\rho_\eta^{n+1}}_{L^4}^2\right)\norm{\nabla a^{n+1}}_{L^4}^2+C\norm{c^{n+1}_t}_{L^2}^2\\
&\quad+C\norm{c^{n+1}_t}_{L^2}\norm{\Delta a^{n+1}}_{L^2}+C\norm{K^n}_{L^2}^2\\
&\leq C_\varepsilon\cI^n(t)\left(\norm{\nabla a^{n+1}}_{L^2}^2+\norm{\sigma^{n+1}}^2_{H^2}+\norm{\sigma^{n+1}_t}^2_{L^2}\right)+C_\varepsilon\norm{\nabla\sigma^{n+1}_t}^2_{L^2}\\
&\quad+C\cB^n(t)\norm{a^n}^2_{L^2}+\varepsilon\norm{\Delta a^{n+1}}_{L^2}^2
\end{aligned}
\end{equation}
where we have used the fact that
\begin{equation*}
\begin{aligned}
\norm{c^{n+1}_t}_{L^2}&\leq C\norm{\nabla\sigma^{n+1}_t}_{L^2}+C\left(\norm{\nabla\rho^n_t}_{L^2}+\norm{\rho^n_t}_{L^4}\norm{\nabla\rho^n}_{L^4}\right)\norm{\sigma^{n+1}}_{L^\infty}\\
&\quad+C\norm{\rho^n_t}_{L^4}\norm{\nabla\sigma^{n+1}}_{L^4}+C\norm{\nabla\rho^{n+1}}_{L^4}\norm{\sigma^{n+1}_t}_{L^4}\\
&\leq C\norm{\nabla\sigma^{n+1}_t}_{L^2}+C\sqrt{\cI^n(t)}\left(\norm{\sigma^{n+1}}_{H^2}+\norm{\sigma^{n+1}_t}_{L^2}\right).
\end{aligned}
\end{equation*}

At last, in order to get the estimate of $\Delta a^{n+1}$, we use \eqref{0517} with the additional term $K^n$,
\begin{equation}
\begin{aligned}
\norm{\Delta a^{n+1}}_{L^2}^2+\norm{\pi^{n+1}-\pi^n}_{H^1}^2&\leq C\norm{a^{n+1}_t}_{L^2}^2+C\norm{\nabla \Delta[(\rho^{n+1})^{-1}-(\rho^{n})^{-1}]}_{L^2}^2\\
&\quad+C\left(\norm{\nabla\rho_\eta^n}_{L^4}^4+\norm{u_\eta^n}_{L^4}^4\right)\norm{\nabla a^{n+1}}_{L^2}^2+C\norm{K^n}_{L^2}^2\\
&\leq C\norm{a^{n+1}_t}_{L^2}^2+C\norm{\nabla \Delta\sigma^{n+1}}_{L^2}^2\quad\text{use }\eqref{0570}\\
&\quad+C\cI^n(t)\left(\norm{\nabla a^{n+1}}_{L^2}^2+\norm{\sigma^{n+1}}_{H^2}^2\right)+C\cB^n(t)\norm{a^n}_{L^2}^2. 
\end{aligned}
\end{equation}
We substitute above into \eqref{ee667} to get
\begin{equation}\label{ee670}
\begin{aligned}
&\frac{d}{dt}\norm{\nabla a^{n+1}}_{L^2}^2+\nu\norm{a^{n+1}_t}_{L^2}^2\\
&\leq C_\varepsilon\cI^n(t)\left(\norm{\nabla a^{n+1}}_{L^2}^2+\norm{\sigma^{n+1}}^2_{H^2}+\norm{\sigma^{n+1}_t}^2_{L^2}\right)+C_\varepsilon\norm{\nabla\sigma^{n+1}_t}^2_{L^2}\\
&\quad+C\cB^n(t)\norm{a^n}^2_{L^2}+\varepsilon\norm{\nabla\Delta \sigma^{n+1}}_{L^2}^2
\end{aligned}
\end{equation}

Therefore, combining \eqref{eee658}--\eqref{ee659}, \eqref{ee661}--\eqref{ee662}, \eqref{ee665} and \eqref{ee670}, we eventually get
\begin{equation}\label{ee672}
\begin{aligned}
\frac{d}{dt}\cY^{n+1}(t)+\nu\cZ^{n+1}(t)&\leq C\left(\cI^n(t)\cY^{n+1}(t)+\cB^n(t)\cY^n(t)+\cI^n(t)\int_0^t\cY^n(s)\,ds\right).
\end{aligned}
\end{equation}
Then, applying the Gr$\mathrm{\ddot{o}}$nwall's inequality and recalling that $\cY^n(0)=0$ and the definitions of $\cI^n(t),\cB^n(t)$, one has, for all $t\in (0,T_1)$,
\begin{equation*}
\begin{aligned}
\cY^{n+1}(t)&\leq C_0\int_0^t\left(C\cB^n(s)\cY^n(s)\,ds+C\cI^n(s)\int_0^s\cY^n(\tau)\,d\tau\right)\,ds\\
&\leq C\int_0^t\cY^n(s)\,ds,
\end{aligned}
\end{equation*}
which reduces to the Volterra-type integral equation. After a simple recursive argument, we can show that
\begin{equation}
\sup_{t\in[0,T_1]}\cY^{n+1}(t)\leq C\frac{(CT_1)^{n-1}}{(n-1)!}\int_0^{T_1}\cY^1(t)\,dt.
\end{equation}
Applying the contraction mapping theorem and using this inequality together with \eqref{ee672}, we show that the sequence $(\rho^n,u^n)$ converges strongly to an unique limit $(\rho,u)$ and, as a consequence, $\pi^n$ converges strongly to a function $\pi$. More precisely, we have
\begin{equation*}
\begin{aligned}
&\rho^n\sconverge\rho\quad\mathrm{in }\,\,C([0,T];H^2)\cap L^2(0,T;H^3),\\
&u^n\sconverge u\quad\mathrm{in }\,\,C([0,T];H^1)\cap L^2(0,T;H^2),\\
&u_t^n\sconverge u_t\quad\mathrm{in }\,\,L^2([0,T];L^2),\\
&\pi^n\sconverge\pi\quad\mathrm{in }\,\,L^2(0,T;H^1).
\end{aligned}
\end{equation*}
Of course, $(\rho,u,\pi)$ is the unique strong solution in $\Omega\times(0,T_1)$ for \eqref{origin}. Furthermore, we can show $(\rho,u,\pi)$ is acually smooth. Indeed, sicne $u\in L^2(0,T;H^2)\cap H^1(0,T;L^2)$, $v_\eta, \rho_\eta\in H^1(0,T;H^\infty)$. With this regularity on $v_\eta, \rho_\eta$, using the regularity theories of parabolic equations for $\eqref{origin}_1$, we can derive that $\rho\in H^2([0,T];H^\infty)$. Then, applying the $L^p$-theory (\cite{solonnikov}) for $\eqref{origin}_2$, we get $u\in H^2(0,T;H^\infty)$ and, hence, we can bootstrap and gain more time regualrity on $v_\eta, \rho_\eta$ then $\rho$, which implies that $(\rho,u)\in C^\infty(\overline Q_T)$. Therefore, we finish the proof of Theorem \ref{theoLP}.
\\
\quad\\

\noindent$\underline{\mathbf{Case\,\,\eqref{A}\,\,or\,\,\eqref{B}}}$:\\
\quad\\

We only consider the case \eqref{A} here and case \eqref{B} can be proved identically. Firstly, it follows from \eqref{0570} that
\begin{equation}\label{eee579}
\begin{aligned}
\norm{K^n}_{H^{-1}}+\norm{K^n}_{L^2}\leq C\sqrt{\cI^n(t)}\norm{\sigma^{n+1}}_{H^2}+C_\eta\cB^n(t)(\norm{b^n}_{L^2}+\norm{\sigma^{n}}_{H^1}).
\end{aligned}
\end{equation}
Then, applying the estimates \eqref{equation5522} and \eqref{equation5528} with $\varphi=\rho_\eta^n$ and $\Phi=v_\eta^n$ and using \eqref{eee579}, we can obtain
\begin{equation}\label{eee580}
\begin{aligned}
&\frac{d}{dt}\norm{\sqrt{\rho^{n+1}}b^{n+1}}_{L^2}^2+\nu\norm{\nabla b^{n+1}}_{L^2}^2\\
&\leq C\left(\norm{\Delta \rho^{n+1}}^2_{L^2}+\norm{\nabla\rho^{n+1}}^4_{L^4}+\norm{\nabla\mu^{n+1}_\epsilon}_{L^4}^4\right)\norm{b^{n+1}}_{L^2}^2\\
&\quad+C\left(\norm{v^{n}_\eta}_{L^4}^4+\norm{\nabla\rho_\eta^n}^4_{L^4}+\norm{\Delta\rho_\eta^n}^2_{L^2}\right)\norm{b^{n+1}}_{L^2}^2\\
&\quad+C\norm{\nabla\Delta[(\rho^{n+1})^{-1}-(\rho^n)^{-1}]}_{L^2}^2+C\norm{K^n}_{H^{-1}}^2\\
&\leq C\cI^n(t)\left(\norm{b^{n+1}}_{L^2}^2+\norm{\sigma^{n+1}}_{H^2}^2\right)+C\norm{\nabla\Delta\sigma^{n+1}}_{L^2}^2+C\cB^n(t)\left(\norm{b^{n}}_{L^2}^2+\norm{\sigma^{n}}_{H^1}^2\right)
\end{aligned}
\end{equation}
and
\begin{equation}\label{eee581}
\begin{aligned}
&\norm{b^{n+1}_t}_{L^2}^2+ \frac{d}{dt}\left(\cM^n_1(t)+\norm{\nabla b^{n+1}}_{L^2}^2\right)+\frac{d}{dt}\cM^n_2(t)\\
&\leq C_\varepsilon\left(\norm{v_\eta^n}^4_{L^4}+\norm{\nabla\rho_\eta^n}^4_{L^4}+\norm{\mu^{n+1}_{\epsilon,t}}_{H^1}^2+\norm{\nabla\mu^{n+1}_\epsilon}_{L^4}^4+1\right)\norm{b^{n+1}}_{H^1}^2\\
&\quad+C_\varepsilon\left(\norm{v_\eta^n}_{L^4}^4+\norm{\nabla\rho_\eta^n}_{L^4}^4+\norm{\nabla\mu^{n+1}_\epsilon}_{L^4}^4\right)\norm{\Delta[(\rho^{n+1})^{-1}-(\rho^n)^{-1}]}^2_{L^2}\\
&\quad+C_\varepsilon\norm{\nabla\Delta[(\rho^{n+1})^{-1}-(\rho^n)^{-1}]}_{L^2}^2\\
&\quad+\varepsilon\left(\norm{\nabla(\log\rho^{n+1}-\log\rho^n)_t}_{L^2}^2+\norm{\nabla[(\rho^{n+1})^{-1}-(\rho^n)^{-1}]_t}_{L^2}^2\right)+C\norm{K^n}_{L^2}^2\\
&\leq C_\varepsilon\cI^n(t)\left(\norm{b^{n+1}}_{H^1}^2+\norm{\sigma^{n+1}}_{H^2}^2\right)+C_\varepsilon\norm{\nabla\Delta\sigma^{n+1}}_{L^2}^2+\varepsilon\norm{\nabla\sigma^{n+1}_t}_{L^2}^2\\
&\quad +C\cB^n(t)\left(\norm{b^{n}}_{L^2}^2+\norm{\sigma^{n}}_{H^1}^2\right),
\end{aligned}
\end{equation}
where
\begin{equation*}
\cM^n_1(t):=\int_\partial \mu_\epsilon^{n+1} b^{n+1}\cdot B\cdot b^{n+1},
\end{equation*}
\begin{equation*}
\cM^n_2(t):=\int c_0\mu_\epsilon^{n+1} \nabla^\perp(b^{n+1}\cdot n^\perp)\cdot B\cdot \nabla \left[(\rho^{n+1})^{-1}-(\rho^{n})^{-1}\right].
\end{equation*}

Therefore, combining \eqref{eee658}--\eqref{ee659}, \eqref{ee661}--\eqref{ee662}, \eqref{eee580}--\eqref{eee581}, we have
\begin{equation}\label{eee583}
\begin{aligned}
&\frac{d}{dt}\cY^{n+1}(t)+\frac{d}{dt}\cM^n_2(t)+\nu\cZ^{n+1}(t)\\
&\leq C\left(\cI^n(t)\cY^{n+1}(t)+\cB^n(t)\cY^n(t)+\cI^n(t)\int_0^t\cY^n(s)\,ds\right).
\end{aligned}
\end{equation}
Here, we have used the fact that $\cM^n_1(t)\geq 0$. 

Noticing that
\begin{equation*}
\begin{aligned}
|\cM^n_2(t)|&\leq \varepsilon\norm{b^{n+1}}_{H^1}^2(t)+C_\varepsilon\left(\norm{\nabla\sigma^{n+1}}_{L^2}^2(t)+\norm{\nabla\rho^n}_{L^4}^4(t)\norm{\sigma^{n+1}}_{L^2}^2(t)\right)\\
&\leq \varepsilon\norm{b^{n+1}}_{H^1}^2(t)+C_\varepsilon\norm{\sigma^{n+1}}_{H^1}^2(t)\\
&\leq \varepsilon\norm{b^{n+1}}_{H^1}^2(t)+C_\varepsilon\int_0^t\cY^{n}(t) \quad\text{use }\eqref{ee660}
\end{aligned}
\end{equation*}
Applying Gr$\mathrm{\ddot{o}}$nwall's inequality and using \eqref{eee583}, we finally get the Volterra-type integral equation
\begin{equation}
\cY^{n+1}(t)\leq C\int_0^t\cY^n(s)\,ds,\quad t\in (0,T_1),
\end{equation}
and, hence, following the proof of case \eqref{C}, we complete the proof for the case \eqref{A}.

In conclusion, we finish the proof for Theorem \ref{theoLP}.

\subsection{Proofs of Theorem \ref{theo:Serrin-type}: Recover $\epsilon$ and $\eta$}\label{P}

We temporarily fix $\epsilon\in (0,1]$ to let $\eta\to 0^+$. We still first consider the case \eqref{C}. 

The recovering process is standard. Using Theorem \ref{theoLP}, we get a smooth sequence 
$$(\rho^{\epsilon,\eta}, u^{\epsilon,\eta},\pi^{\epsilon,\eta})\in C^\infty(\overline Q_{T_1})$$ 
which solves the problem \eqref{origin} for each $\epsilon,\,\eta\in (0,1]$ (for simplicity, we use the notation $(\rho^{\eta}, u^{\eta},\pi^{\eta})$).  Next, we can follow the proofs in Lemma \ref{Lemma61}--\ref{Lemma62} step by step and the uniform bounds \eqref{eq5.3.5} to obtain the following control
\begin{equation}\label{ee674}
\frac{d}{dt}\cF^\eta(t)+\nu\cG^\eta(t)\leq C\cF^\eta(t)^3,
\end{equation}
where 
\begin{gather*}
\cF^\eta(t):=\norm{u^\eta}_{L^2}^2+\norm{\rho^\eta_t}_{L^2}^2+\norm{\rho^\eta}_{H^2}^2,\\ 
\cG^\eta(t):=\norm{u^\eta_t}_{L^2}^2+\norm{u^\eta}_{H^2}^2+\norm{\rho^\eta}_{H^3}^2+\norm{\rho^\eta_t}_{H^1}^2.
\end{gather*}
and $C$ is a constant which is not depend on $\eta$. Using the inequality \eqref{ee674}, we can easily deduce that there eixsts a positive time $T_2$ such that 
\begin{equation}\label{ee675}
\sup_{t\in [0,T_2]}\cF^\eta(t)+\int_0^{T_2}\cG^\eta(t)\,dt\leq C_2.
\end{equation}
Therefore, using the above uniform esitmate and Lemma \ref{Lemma2.6}, we can derive that $(\rho^\eta,u^\eta,\pi^\eta)$ converges in some proper sense to the limit $(\rho,u,\pi)$ such that
\begin{equation}\label{origin2}
\begin{cases}
\rho_t+\dive(\rho u)=0,\\
\rho u_t+\rho u\cdot\nabla u-\dive[2\mu_\epsilon D(u)]+\nabla\pi=0,\\
\dive u=c_0\Delta\rho^{-1}.
\end{cases}
\end{equation}
The convergence is easy to check, we left it to the reader. Of course, as a special case, we can let $\mu_\epsilon$ be a constant $\mu$ and, thus, we have proved the uniqueness and existence of the local strong solutions for the case \eqref{C}.

For the case \eqref{A} or \eqref{B}, the proof is basically the same. However, the difference lies in this case is that we can recover $\epsilon\to 0^+$ because of the uniform estimates of $(\rho^\epsilon,u^\epsilon,\pi^\epsilon)$, $\epsilon\in (0,1]$, see Remark \ref{remark56}. The convergence is easy to check and we omit it. Thus, we have completed the proof of the existence results for Theorem \ref{theo:Serrin-type}.

It remains to check the uniqueness for the case \eqref{A} or \eqref{B}. However, this can be done by following the proof in \ref{5.3.2}. Indeed, for example, if we consider the case \eqref{A} (another case can be proved analogously), let $(\rho_i,u_i,\pi_i)$, $i=1,2$, be two strong solutions on $\Omega\times(0,T_1)$ with same initial data and set
\begin{equation*}
\sigma:=\rho_1-\rho_2,\quad a:=u_1-u_2,\quad b:=v_1-v_2,\quad c:=Q_1-Q_2,
\end{equation*}
\begin{equation*}
\cY(t):=\norm{a}_{H^1}^2+\norm{\sigma}_{H^2}^2+\norm{\sigma_t}_{L^2}^2,\quad
\cZ(t):=\norm{a_t}_{L^2}^2+\norm{\sigma}_{H^3}^2+\norm{\sigma_t}_{H^1}^2.
\end{equation*}
Then, we can derive the similar type of equations to \eqref{eq6.65} and \eqref{ee663}, that is,
\begin{equation*}
\begin{cases}
\sigma_t+v_2\cdot\nabla\sigma-c_0\dive\left(\rho_2^{-1}\nabla\sigma\right)=-b\cdot\nabla\rho_2-c_0\dive\left(\sigma\rho_1^{-1}\rho_2^{-1}\nabla\rho_2\right),\\
\rho_1a_t+\rho_1u_1\cdot\nabla a-\dive[2\mu D(a)]+\nabla (\pi_1-\pi_2)=-\sigma(u_{1,t}+u_1\cdot \nabla u_1)-\rho_2 a\cdot\nabla u_1.
\end{cases}
\end{equation*}
Applying the same discussions from \ref{5.3.2}, we can get the following type inequality
\begin{equation*}
\frac{d}{dt}\cY(t)+\cZ(t)\leq C\cI(t)\cY(t),
\end{equation*}
where $\cI$ stands for some integrable functions on time interval $(0,T_1)$. Thus, using Gr$\mathrm{\ddot{o}}$nwall's inequality and the fact that $\cY(0)=0$, we can easily deduce that $\cY(t)\equiv 0$, which yields the uniqueness.

\section{Proof of Theorem \ref{theorem1.2}--\ref{theorem1.6}}\label{proof6}

\subsection{Proof of Theorem \ref{theorem1.3}--\ref{theorem1.6}}
Since we have already show the existence and uniqueness of strong solution on $\Omega\times(0,T_1)$ for some positive times $T_1$, the proof of global ones is quite standard with the a priori estiamtes we obtained in Section \ref{section3}--\ref{Section8}. One thing we should mention is that there is a gap between the local existence and the global one when $(\rho,u)$ satisfies the condition \eqref{C} in that we only established the unique local strong solution for $\mu=\mu_\epsilon$. In every case that follows, one should first recover $\epsilon\to 0^+$ to get the global existence and, then, show their uniqueness under the smallness assumption $\norm{\nabla u_0}_{L^2}\leq \delta$. Fourtunately, the proof of either is simpe and indentical with that in subsection \ref{P}. The only thing one should notice is that, under the restriction $\norm{\nabla u_0}_{L^2}\leq \delta$, Proposition \ref{prop8.1} holds and, thus, we always have 
$$\sup_{t\in [0,T]}\norm{\nabla\rho}_{L^4}\leq 1,$$
which allows us to use Lemma \ref{lemma8.4} (in such case, there is no difference between the estimates of Lemma \ref{Lemma5.2} and those of Lemma \ref{lemma8.4}) and get the uniqueness.

\subsection{Proof of Theorem \ref{theorem1.2}}

Following the construction process in subsection \ref{PR}, one can find a smooth sequence $(\rho_0^n,v_0^n)$ such that
\begin{equation}
\begin{gathered}
\rho_0^n\sconverge\rho_0\quad\text{in}\,\,H^1,\quad \alpha\leq \rho_0^n\leq \beta,\\
v_0^n\sconverge v_0\quad\text{in}\,\,L^2,\quad \dive v_0^n=0,\quad v_0^n\cdot n=0\quad\text{on }\partial\Omega,\\
(\rho^n_0, v_0^n)\text{ satisfying }\eqref{A'}\text{ or }\eqref{B'}.
\end{gathered}
\end{equation}

If we define $u_0^n:=v_0^n+c_0\nabla(\rho_0^n)^{-1}$, it is easy to check that $u_0^n$ is smooth and $(\rho_0^n,u_0^n)$ satisfies all the conditions in Theorem \ref{theorem1.3}. Thus, by using Theorem \ref{theorem1.3}, there exists a sequence of global strong solutions $(\rho^n,u^n)$ of \eqref{equation1.1} with initial data $(\rho^n_0,u^n_0)$. Then, using the uniform bounds we get from subsection \ref{LOE}, extracting subsequences if necessary, we can derive a weak convergent subsequence satisfying 
\begin{equation}\label{6..2}
\begin{cases}
\rho^n\wsconverge \rho\,\,&\mathrm{in}\,\,L^\infty(0,T;H^1),\\
\rho^n\wconverge \rho\,\,&\mathrm{in}\,\,L^2(0,T;H^2),\\
\rho^n_t\wconverge \rho_t\,\,&\mathrm{in}\,\,L^2(0,T;L^2),\\
u^n\wsconverge u\,\,&\mathrm{in}\,\,L^\infty(0,T;L^2),\\
u^n\wconverge u\,\,&\mathrm{in}\,\, L^2(0,T;H^1).
\end{cases}
\end{equation}
Next, we can apply Lemma \ref{lemma5.1} to obtain \eqref{equation5.6}--\eqref{equa4.7}. With these hold in hand, one can immediately get 
\begin{equation}\label{6..3}
u^n\sconverge u\quad \text{in }L^2(0,T;L^2).
\end{equation}
Indeed, since $v^n\sconverge v\text{ in }L^2(0,T;L^2)$, it suffices to show the strong convergence for $\nabla(\rho^n)^{-1}$, that is,
\begin{equation}\label{6..4}
\nabla(\rho^n)^{-1}\sconverge \nabla\rho^{-1}\quad \text{in }L^2(0,T;L^2).
\end{equation}
However, 
\begin{equation*}
\nabla(\rho^n)^{-1}=-(\rho^n)^{-2}\nabla\rho^n
\end{equation*}
and $\rho^n\sconverge\rho\text{ in }L^2(0,T;H^1)$, since we have $\eqref{6..2}_2$--$\eqref{6..2}_3$ and, then, use Lemma \ref{Lemma2.6}. Therefore, \eqref{6..4} is an easy consequence of \eqref{equation5.6} and Egorov theorem.

Finally, using \eqref{equation5.6}, \eqref{6..2}--\eqref{6..3}, we can recover the weak solutions $(\rho,u)$ for system \eqref{equation1.1} and complete the prove of Theorem \ref{theorem1.2}.

\section{Proof of Theorem \ref{Theorem1.7}}\label{section9}

In the last section, we come to prove the blowup criterion for $(\rho,u)$ satisfying one of three conditions \eqref{A}, \eqref{B} and \eqref{C}.  Let $(\rho,u,\pi)$ be a local strong solution as being described in Theorem \ref{theo:Serrin-type} and suppose that \eqref{serrin1} or \eqref{serrin2} was false, that is, for some $r$ and $s$ satisfying \eqref{rs},
\begin{equation}\label{4.1}
\lim_{T\to T^*}\norm{\nabla\rho}_{L^s(0,T;L^r)}\leq M_0<\infty.
\end{equation}
or
\begin{equation}\label{serrin816}
\lim_{T\to T^*}\norm{u}_{L^s(0,T;L^r)}\leq M_0<\infty.
\end{equation}
We also let $\tilde C$ be a positive generic constant depending on $\Omega$, $c_0$, $\alpha$, $\beta$, $T^*$, $M_0$ and $\norm{u_0}_{H^1}$. Then, our goal is proving the following estimate.

\begin{Proposition}\label{prop4.1}
Suppose that \eqref{4.1} holds for $(\rho,u)$ satisfying the condition \eqref{A} or \eqref{B} and \eqref{serrin816} holds for $(\rho,u)$ satisfying the condition \eqref{C}. Then, one has, for all $T\in(0,T^*)$,
\begin{equation}\label{88.3}
\sup_{t\in[0,T]}\left(\norm{\rho_t}_{L^2}^2+\norm{\rho}^2_{H^2}+\norm{u}^2_{H^1}\right)+\int_0^T\left(\norm{\rho_t}_{H^1}^2+\norm{\rho}_{H^3}^2+\norm{u}_{H^2}^2\right)\,dt\leq \tilde C.
\end{equation}
\end{Proposition}

Before proving the proposition, let us show how to derive the blowup criterion in Theorem \ref{theo:Serrin-type} from Proposition \ref{prop4.1}. 
\begin{proof}[Proof of Theorem \ref{Theorem1.7}]
For simplicity, we give the prove for the case when $(\rho,u)$ satisfies the condition \eqref{A}, since other cases can be proved identically. Note that $\tilde C$, in \eqref{88.3}, is uniformly bounded for all $T\leq T^*$, so
$$(\rho,u)(x,T^*):=\lim_{t\to T^*}(\rho,u)(x,t)\text{ in the sense of } H^2\times H^1$$
satisfying the conditions imposed on the initial data, that is, $\alpha	\leq \rho_0\leq \beta,\,\,u_0\in H^1_\omega$, at the time $t=T^*$. Furthermore,
\begin{equation*}
\begin{cases}
\dive u|_{t=T^*}=c_0\Delta\rho^{-1}|_{t=T^*},&x\in \Omega\\
u|_{t=T^*}\cdot n=n\cdot\nabla\rho^{-1}|_{t=T^*},&x\in \partial\Omega
\end{cases}
\end{equation*}
Thus, $(\rho,u)(x,T^*)$ satisfies \eqref{compat} also. Therefore, we can take $(\rho,u)(x,T^*)$ as the initial data and apply the existence result in Theorem \ref{theo:Serrin-type} to extend the local strong solution beyond $T^*$. This contradicts the maximality of $T^*$ and, hence, we finish the proof of Theorem \ref{theo:Serrin-type}.
\end{proof}

\subsection{Case for $(\rho,u)$ satisfying \eqref{A} or \eqref{B}}\label{sub8.1}
In this subsection, we always let $(\rho,u)$ satisfy the condition \eqref{A} or \eqref{B}. Recall that it is also equivalent to require $(\rho,v)$ satisfying the condition \eqref{A'} or \eqref{B'}.

The proof for the first part of Proposition \ref{prop4.1} will be separated into the following few steps. The key of the proof is obtaining the lower order estimates for $(\rho,v)$, that is, $(\nabla\rho,v)\in C([0,T];L^2)\cap L^2(0,T;H^1)$, then, following the proof in Section \ref{Section6}, the weak solution is automatically a strong one.

The first lemma is just the combination of Lemma \ref{lemma4.1} and \ref{lemma43}, we give it here for convenience.
\begin{Lemma}\label{lemma4.2}
The following bounds hold for all $T\in(0,T^*)$, that is,
\begin{equation}
\alpha\leq\rho\leq\beta,\quad\sup_{t\in[0,T]}\norm{\rho}_{L^2}^2+\nu\int_0^T\norm{\nabla\rho}_{L^2}^2\,dt\leq C.
\end{equation}
\end{Lemma}

The next crucial lemma gives the lower bounds of $(\rho,v)$, that is,
\begin{Lemma}\label{lemma44}
Suppose that \eqref{4.1} holds and $(\rho,v)$ satisfies the condition \eqref{A'} or \eqref{B'}, then one has
\begin{equation}\label{equation4.4}
\sup_{t\in[0,T]}\left(\norm{\nabla\rho}^2_{L^2}+\norm{v}^2_{L^2}\right)+\int_0^T\left(\norm{\Delta\rho}_{L^2}^2+\norm{\nabla v}_{L^2}^2\right)\,dt\leq \tilde C.
\end{equation}
\end{Lemma}
\begin{proof}
We first follow the proof of Lemma \ref{Lemma4.3}, applying Lemma \ref{Lemma2.3} and \ref{lemma4.2}, to get
\begin{equation*}
\begin{aligned}
\frac{d}{dt}\norm{\nabla\rho}_{L^2}^2 +\nu\norm{\Delta\rho}_{L^2}^2&\leq C\int \left(\abs{\nabla\rho}^2+\abs{v}\abs{\nabla\rho}\right)\abs{\Delta\rho}\\
&\leq C\norm{\nabla\rho}_{L^r}\left(\norm{\nabla\rho}_{L^{\frac{2r}{r-2}}}+\norm{v}_{L^{\frac{2r}{r-2}}}\right)\norm{\Delta\rho}_{L^2}\\
&\leq C\norm{\nabla\rho}^{\frac{2r}{r-2}}_{L^r}\left(\norm{\nabla\rho}_{L^2}^2+\norm{v}_{L^2}^2\right)+\frac{\nu}{2}\norm{\Delta\rho}_{L^2},
\end{aligned}
\end{equation*}
which implies that
\begin{equation}\label{equ85}
\frac{d}{dt}\norm{\nabla\rho}_{L^2}^2 +\nu\norm{\Delta\rho}_{L^2}^2\leq C(\norm{\nabla\rho}^{s}_{L^r}+1)\left(\norm{\nabla\rho}_{L^2}^2+\norm{v}_{L^2}^2\right).
\end{equation}

On the onther hand, as we did in \eqref{4412}, multiplying $v$ on both sides of $\eqref{equation3.2}_2$ and integrating over $\Omega$,
\begin{equation}\label{equ88}
\begin{aligned}
\frac{d}{dt}\int\frac{1}{2}\rho\abs{v}^2-\int \dive{[2\mu D(v)]}\cdot v =\sum_{i=1}^3I_i,
\end{aligned}
\end{equation}
where $I_i$, $i=1,2,3$, as in \eqref{4412}. From \eqref{equ412}, applying Lemma \ref{Lemma2.3}, \ref{Lemma2.2} and \ref{lemma4.2}, the second term on the left-hand side can be controlled by
\begin{equation}
\begin{aligned}
-\int \dive{[2\mu D(v)]}\cdot v
&\geq\underline{\mu}\int\abs{\curle v}^2- C\left(\norm{\nabla\rho}_{L^r}\norm{\sqrt\rho v}_{L^{\frac{2r}{r-2}}}\norm{\nabla v}_{L^2}\right)\\
&\geq\nu\norm{\nabla v}_{L^2}^2- \left[C_\varepsilon(\norm{\nabla\rho}_{L^r}^s+1)\norm{\sqrt\rho v}_{L^2}^2+\varepsilon\norm{\nabla v}_{L^2}^2\right].
\end{aligned}
\end{equation}
Following the proof from \eqref{equa3.14} to \eqref{equ3.12}, since 
\begin{equation}
\begin{aligned}
|J_1'|&= \abs{\int_{\partial}\phi(\rho)(n\cdot\nabla\rho)(v\cdot n^\perp)n^\perp\cdot\nabla\rho}\\
&=\abs{\int\nabla^\perp[\phi(\rho)(n\cdot\nabla\rho)]\cdot\nabla\rho(v\cdot n^\perp) + \int\phi(\rho)(n\cdot\nabla\rho)\nabla\rho\cdot\nabla^\perp(v\cdot n^\perp)}\\
&\leq C_{\varepsilon_1}\norm{\nabla\rho}_{L^r}^2\left(\norm{\sqrt\rho v}_{L^{\frac{2r}{r-2}}}^2+\norm{\nabla\rho}_{L^{\frac{2r}{r-2}}}^2\right)+\varepsilon_1\left(\norm{\nabla v}_{L^2}^2+\norm{\Delta\rho}_{L^2}^2\right) \\
&\leq C_{\varepsilon_1}(\norm{\nabla\rho}_{L^r}^s+1)\left(\norm{\sqrt\rho v}_{L^2}^2+\norm{\nabla\rho}_{L^2}^2\right)+\varepsilon_1\left(\norm{\nabla v}_{L^2}^2+\norm{\Delta\rho}_{L^2}^2\right) 
\end{aligned}
\end{equation}
\begin{equation}
\begin{aligned}
|J_2'|&= \abs{\int_{\partial}\phi(\rho)(v\cdot n^\perp)n^\perp\cdot\nabla n\cdot\nabla\rho}\\
&= \left|\int\nabla^\perp\phi(\rho)\cdot(\nabla n\cdot\nabla\rho)(v\cdot n^\perp)-\int_{\Omega}\phi(\rho)\nabla^\perp\cdot(\nabla n\cdot\nabla\rho)(v\cdot n^\perp)\,dx\right.\\
&\quad\left.-\int\phi(\rho)\nabla^\perp(v\cdot n^\perp)\cdot(\nabla n\cdot\nabla\rho)\right|\\
&\leq C_{\varepsilon_2}\left(\norm{v}_{L^2}^2 +\norm{\nabla\rho}_{L^2}^2\right) + C_{\varepsilon_2}\norm{\nabla\rho}_{L^r}^2\norm{\sqrt\rho v}_{L^{\frac{2r}{r-2}}}^2+\varepsilon_2\left(\norm{\nabla v}_{L^2}^2+\norm{\Delta\rho}_{L^2}^2\right)\\
&\leq C_{\varepsilon_2}\norm{\nabla\rho}_{L^2}^2+ C_{\varepsilon_2}(\norm{\nabla\rho}_{L^r}^s+1)\norm{\sqrt\rho v}_{L^2}^2+\varepsilon_2\left(\norm{\nabla v}_{L^2}^2+\norm{\Delta\rho}_{L^2}^2\right). 
\end{aligned}
\end{equation}
and
\begin{equation}
\begin{aligned}
|J_3|&=\abs{\int 2c_0\mu(\rho)\partial_{ij}\rho^{-1}\partial_jv_i}\\
&=\abs{\int_\partial 2c_0\mu(\rho)\nabla\rho^{-1}\cdot\nabla v\cdot n-\int 2c_0\mu'\nabla\rho^{-1}\cdot \nabla v\cdot \nabla\rho}\\
&=\abs{-\int_\partial 2c_0\mu(\rho)\nabla\rho^{-1}\cdot\nabla n\cdot v-\int 2c_0\mu'\nabla\rho^{-1}\cdot \nabla v\cdot \nabla\rho}\\
&\leq  C_{\varepsilon_2}(\norm{\nabla\rho}_{L^r}^s+1)\left(\norm{\sqrt\rho v}_{L^2}^2+\norm{\nabla\rho}_{L^2}^2\right)+\varepsilon_2\left(\norm{\nabla v}_{L^2}^2+\norm{\Delta\rho}_{L^2}^2\right),
\end{aligned}
\end{equation}
we deduce that
\begin{equation}\label{equ813}
\begin{aligned}
|I_1|&\leq  C_{\varepsilon}(\norm{\nabla\rho}_{L^r}^s+1)\left(\norm{\sqrt\rho v}_{L^2}^2+\norm{\nabla\rho}_{L^2}^2\right)+\varepsilon\left(\norm{\nabla v}_{L^2}^2+\norm{\Delta\rho}_{L^2}^2\right),
\end{aligned}
\end{equation}

Similarly, for $I_2$--$I_3$, one has
\begin{equation}\label{equ814}
\begin{aligned}
|I_2|&\leq C_\varepsilon(\norm{\nabla\rho}_{L^r}^s+1)\norm{\sqrt{\rho}v}_{L^2}^2+\varepsilon\norm{\nabla v}_{L^2}^2,\\
|I_3|&\leq C_\varepsilon(\norm{\nabla\rho}_{L^r}^s+1)\norm{\nabla\rho}_{L^2}^2+\varepsilon\left(\norm{\nabla v}_{L^2}^2+\norm{\Delta\rho}_{L^2}^2\right).
\end{aligned}
\end{equation}
Substituting \eqref{equ813}--\eqref{equ814} into \eqref{equ88} and, then, alonging with \eqref{equ85} leads to
\begin{equation}\label{equ814}
\begin{aligned}
&\frac{d}{dt}\left(\norm{\sqrt{\rho}v}_{L^2}^2+\norm{\nabla\rho}_{L^2}^2\right) +\nu\left(\norm{\nabla v}_{L^2}^2+\norm{\Delta\rho}_{L^2}^2\right)\\
&\leq C(\norm{\nabla\rho}^{s}_{L^r}+1)\left(\norm{\nabla\rho}_{L^2}^2+\norm{\sqrt{\rho}v}_{L^2}^2\right).
\end{aligned}
\end{equation}
Finally, applying the Gr$\mathrm{\ddot{o}}$nwall's inequality to \eqref{equ814}, we finish the proof of Lemma \ref{lemma44}.
\end{proof}

Now, we can prove the first part of Proposition \ref{prop4.1}.
\begin{proof}[Proof of Proposition \ref{prop4.1}]
It follows from Lemma \ref{lemma4.2} and \ref{lemma44} that 
\begin{equation}\label{equ815}
\sup_{t\in[0,T]}\left(\norm{\rho}^2_{H^1}+\norm{v}^2_{L^2}\right)+\int_0^T\left(\norm{\rho}_{H^2}^2+\norm{v}_{H^1}^2\right)\,dt\leq \tilde C.
\end{equation}
Thus, by Lemma \ref{Lemma2.3}, we get the bounds 
$$\int_0^T\norm{(\nabla\rho,v)}_{L^4}^4\,dt\leq \tilde C.$$
This, together with \eqref{equ815}, allows us to follow the proof of Proposition \ref{theorem6.3} step by step, since the lower order bounds are enough to deduce the higher ones, according to Lemma \ref{lemma6.4} and \ref{lemma6.5} (noticing that, the proofs of  Lemma \ref{lemma6.4} and \ref{lemma6.5} are merely based on the smallness assumption we derived from Proposition \ref{theorem4.5}, that is, $\norm{\nabla\rho_0}_{L^2}\leq \delta$, without any additional restriction, see also Remark \ref{remark53}). We omit the remaining proof here and leave it to the reader. 
\end{proof}

\subsection{Case for $(\rho,u)$ satisfying \eqref{C}}
Now, we assume that $(\rho,u)$ satisfies the condition \eqref{C}.
One should notice that condition \eqref{serrin816} is also equivalent with 
\begin{equation}\label{serrin816'}
\lim_{T\to T^*}\left(\norm{v}_{L^s(0,T;L^r)}+\norm{\nabla\rho}_{L^s(0,T;L^r)}\right)\leq \tilde M_0<\infty,
\end{equation}
since $\rho$ is bounded from above and below and the identity \eqref{equation3.1}.

Our aim is proving the rest of Proposition \ref{prop4.1} under \eqref{serrin816}. First, we give the following lemma, which concludes some results we need later. This nothing but a directly application of Lemma \ref{lemma61} and Lemma \ref{lemma4.2}. 
\begin{Lemma}\label{lem8.4}
Let $(\rho,u,\pi)$ be a local strong solution as being described in Theorem \ref{theo:Serrin-type}. Then, Lemma \ref{lemma4.2} still holds. Moreover, under the condition \eqref{serrin816} (or, equivalently, \eqref{serrin816'}), one has $\rho\in C^{\gamma,\frac{\gamma}{2}}(\overline Q_T)$ for some $\gamma\in (0,1)$ and for all $T\in(0,T^*)$.
\end{Lemma}

Next, with help of the Serrin's condition \eqref{serrin816}, one can get the lower bound of $\rho$.
\begin{Lemma}\label{lemma45}
Suppose that \eqref{serrin816'} holds and $(\rho,u)$ satisfies \eqref{C}, then
\begin{equation}\label{4.11}
\begin{aligned}
\sup_{t\in[0,T]}\norm{\nabla\rho}_{L^2}^2+\int_0^T\left(\norm{\nabla\rho}_{L^4}^4+\norm{\Delta\rho}_{L^2}^2\right)\,dt\leq \tilde C.
\end{aligned}
\end{equation}
\end{Lemma}

\begin{proof}
As we did in Lemma \eqref{Lemma4.3}, applying Lemma \ref{Lemma2.3},
\begin{equation*}
\begin{aligned}
\frac{d}{dt}\norm{\nabla\rho}_{L^2}^2 +\nu\norm{\Delta\rho}_{L^2}^2&\leq C_\varepsilon\left(\norm{\nabla\rho}_{L^r}^2 + \norm{v}^2_{L^r}\right)\norm{\nabla\rho}^2_{L^{\frac{2r}{r-2}}}+\varepsilon\norm{\Delta\rho}_{L^2}^2\\
&\leq C_\varepsilon\left(\norm{\nabla\rho}_{L^r}^s + \norm{v}^s_{L^r}+1\right)\norm{\nabla\rho}^2_{L^2}+\varepsilon\norm{\Delta\rho}_{L^2}^2,
\end{aligned}
\end{equation*}
that is,
\begin{equation}
\frac{d}{dt}\norm{\nabla\rho}_{L^2}^2 +\nu\norm{\Delta\rho}_{L^2}^2\leq C\left(\norm{\nabla\rho}_{L^r}^s + \norm{v}^s_{L^r}+1\right)\norm{\nabla\rho}^2_{L^2}.
\end{equation}
Thus, using Gr$\mathrm{\ddot{o}}$nwall's inequality and Lemma \ref{Lemma2.3}, we conclude the proof.
\end{proof}
\begin{Remark}
With this Lemma \eqref{lemma45} and condition \eqref{serrin816'}, we deduce from $\eqref{equation3.2}_1$ that 
\begin{equation}
\int_0^T\norm{\rho_t}_{L^2}^2\,dt\leq \tilde C.
\end{equation}
\end{Remark}

Now, we can prove Proposition \ref{prop4.1}.
\begin{proof}[Proof of Proposition \ref{prop4.1}]
We start with \eqref{equation8.16}
\begin{equation}\label{19}
\begin{aligned}
\frac{d}{dt}\int \frac{1}{2}\rho|u|^2+\int 2\mu(\rho)|D(u)|^2:=\sum_{i=1}^3 S_i,
\end{aligned}
\end{equation}
where $S_i$ as in \eqref{equation8.16}. Using Lemma \ref{Lemma2.3},
\begin{equation}\label{20}
\begin{cases}
|S_1|\leq C\norm{Q}_{L^2}\norm{u_t}_{L^2}\leq C_{\varepsilon_1}\norm{\nabla\rho}^2_{L^2}+\varepsilon_1\norm{u_t}^2_{L^2},\\
|S_2|\leq C\norm{Q}_{L^r}\norm{u}_{L^{\frac{2r}{r-2}}}\norm{\nabla u}_{L^2}\leq C_{\varepsilon_2}\left(\norm{\nabla\rho}^s_{L^r}+1\right)\norm{u}_{L^2}^2+\varepsilon_2\norm{\nabla u}_{L^2}^2,\\
|S_3|\leq C\norm{\nabla Q}_{L^2}\norm{\nabla u}_{L^2}\leq C_{\varepsilon_3}\left(\norm{\Delta\rho}^2_{L^2}+\norm{\nabla\rho}^4_{L^4}\right)+\varepsilon_3\norm{\nabla u}^2_{L^2}.
\end{cases}
\end{equation}
Combining \eqref{19} and \eqref{20} leads to, 
\begin{align}\label{88.22}
\frac{d}{dt}\norm{u}_{L^2}^2+\nu\norm{\nabla u}_{L^2}^2&\leq C_{\varepsilon}\left(\norm{\nabla\rho}^s_{L^r}+1\right)\norm{u}_{L^2}^2+ C_\varepsilon\left(\norm{\Delta\rho}^2_{L^2}+\norm{\nabla\rho}_{L^4}^4\right)+\varepsilon\norm{u_t}_{L^2}^2,
\end{align}

Similarly, we deduce from \eqref{8818}--\eqref{equation821} that 
\begin{equation}\label{88.23}
\begin{aligned}
\frac{d}{dt}\normf{\sqrt{\mu(\rho)}|D(u)|}_{L^2}^2+\nu\norm{u_t}_{L^2}^2&\leq C_\varepsilon\left(\norm{u}_{L^r}^s+\norm{\nabla\rho}_{L^4}^4+\norm{\rho_t}_{L^2}^2+1\right)\norm{\nabla u}_{L^2}^2\\
&\quad+C_\varepsilon\norm{\nabla\rho_t}_{L^2}^2+\varepsilon\norm{\Delta u}_{L^2}^2.
\end{aligned}
\end{equation}

By Lemma \ref{lem8.4}, $\mu(\rho)\in C(\overline Q_T)$, hence, we can apply Lemma \ref{Lemma5.2} for \eqref{5533} with $\Phi=-c_0\nabla\rho^{-1}$ and, then, use Lemma \ref{Lemma2.3} and \ref{lem8.4} to deduce that
\begin{equation*}
\begin{aligned}
\norm{v}^2_{H^2}+\norm{\pi}_{H^1}^2&\leq C\left(\norm{F}_{L^2}^2+\norm{\nabla\Delta\rho^{-1}}_{L^2}^2\right)\\
& \leq C\left(\norm{v}_{L^r}^s+\norm{\nabla\rho}_{L^4}^4+1\right)\left(\norm{\nabla v}^2_{L^2}+\norm{\Delta\rho}^2_{L^2}+\norm{\rho_t}_{L^2}^2\right)\\
&\quad+C\left(\norm{v_t}^2_{L^2}+\norm{\nabla\rho_t}^2_{L^2}\right) +C\norm{\nabla\Delta\rho}^2_{L^2},
\end{aligned}
\end{equation*}
which gives
\begin{equation}\label{88.24}
\begin{aligned}
\norm{\Delta u}^2_{L^2}+\norm{\pi}_{H^1}^2& \leq C\left(\norm{u}_{L^r}^s+\norm{\nabla\rho}_{L^4}^4+1\right)\left(\norm{\nabla u}^2_{L^2}+\norm{\Delta\rho}^2_{L^2}+\norm{\rho_t}_{L^2}^2\right)\\
&\quad+C\left(\norm{u_t}^2_{L^2}+\norm{\nabla\rho_t}^2_{L^2}\right) +C\norm{\nabla\Delta\rho}^2_{L^2},
\end{aligned}
\end{equation}
Plugging \eqref{88.24} into \eqref{88.23} and choosing $\varepsilon$ sufficiently small, we have, for some positive constant $\nu$ depending on $\Omega$, $c_0$, $\alpha$ and $\beta$,
\begin{equation}\label{88.25}
\begin{aligned}
&\frac{d}{dt}\normf{\sqrt{\mu(\rho)}|D(u)|}_{L^2}^2+\nu\left(\norm{\Delta u}_{L^2}^2+\norm{u_t}_{L^2}^2\right)\\
&\leq C\left(\norm{u}_{L^r}^s+\norm{\nabla\rho}_{L^4}^4+\norm{\rho_t}_{L^2}^2+1\right)\norm{\nabla u}_{L^2}^2+C\left(\norm{\nabla\rho_t}_{L^2}^2+\norm{\nabla\Delta \rho}_{L^2}^2\right).
\end{aligned}
\end{equation}

On the other hand, following the proof from \eqref{eq4.1} to \eqref{equa5.11}, replacing $\norm{v}_{L^4}^4$ by $\norm{v}_{L^r}^s$, then, replacing $v$ by $u$ via \eqref{equation3.1} and applying Lemma \ref{Lemma2.3}, one has
\begin{equation*}
\frac{d}{dt}\norm{\Delta\rho}_{L^2}^2+\nu\norm{\nabla\Delta\rho}_{L^2}^2\leq C_\varepsilon\left(\norm{\nabla\rho}_{L^4}^4+\norm{u}_{L^r}^s+1\right)\left(\norm{\Delta\rho}_{L^2}^2+\norm{\nabla u}_{L^2}^2\right)+\varepsilon\norm{\Delta u}_{L^2}^2,
\end{equation*}
Alonging with \eqref{8824}, we deduce that
\begin{equation}\label{88.26}
\begin{aligned}
&\frac{d}{dt}\left(\norm{\Delta\rho}_{L^2}^2+\norm{\rho_t}_{L^2}^2\right)+\norm{\nabla\Delta\rho}_{L^2}^2 +\norm{\nabla\rho_t}_{L^2}^2\\
&\leq C_{\varepsilon}\left(\norm{u}_{L^r}^s+\norm{\nabla\rho}_{L^4}^4+\norm{\Delta\rho}_{L^2}^2+1\right)\left(\norm{\Delta\rho}_{L^2}^2+\norm{\rho_t}_{L^2}^2+\norm{\nabla u}_{L^2}^2\right)\\
&\quad+\varepsilon\left(\norm{\Delta u}_{L^2}^2+\norm{u_t}_{L^2}^2\right).
\end{aligned}
\end{equation}

Combining \eqref{88.22}, \eqref{88.25} and \eqref{88.26}, then, applying the Gr$\mathrm{\ddot{o}}$nwall's inequality, condition \eqref{serrin816} and Lemma \ref{lemma45}, we get, for all $T\in (0,T^*)$,
$$\sup_{t\in[0,T]}\left(\norm{u}_{H^1}^2+\norm{\rho_t}_{L^2}^2+\norm{\Delta\rho}_{L^2}^2\right)+\int_0^T\left(\norm{\nabla u}_{H^1}^2+\norm{\nabla\rho_t}_{L^2}^2+\norm{\nabla\Delta\rho}_{L^2}^2\right)\,dt\leq \tilde C.$$
Then, we can turn back to \eqref{88.24} to get 
$$\int_0^T\norm{\pi}_{H^1}^2\,dt\leq \tilde C.$$
Therefore, we complete the proof of Proposition \ref{prop4.1}.
\end{proof}

\bibliographystyle{abbrv}
\bibliography{reference}

\end{document}